\documentclass[11pt]{article} 

\usepackage{amsthm,amsmath,amsfonts,amssymb,enumerate}
\usepackage{graphicx,tikz}
\usepackage{wrapfig}
\usepackage{blindtext}
\usepackage{geometry}
\usepackage{float}
\usepackage{multicol}

\usetikzlibrary{positioning,arrows}
\usepackage[utf8]{inputenc} 
\usepackage[T1]{fontenc}
\usepackage{hyperref}
\usepackage{theoremref}
\usepackage{tikz-cd}

\numberwithin{equation}{section}
\theoremstyle{plain}
\newtheorem{lem}[equation]{Lemma}

\newtheorem{prop}[equation]{Proposition}
\newtheorem{thm}[equation]{Theorem}
\newtheorem{cor}[equation]{Corollary}

\newtheorem{conj}[equation]{Conjecture}
\theoremstyle{definition}
\newtheorem{definition}[equation]{Definition}

\newtheorem{remark}[equation]{Remark}

\newtheorem{claim}[equation]{Claim}

\newtheorem*{claim*}{Claim}

\newcommand{\type}{\operatorname{Type}}
\newcommand{\cq}{\mathcal{Q}}
\newcommand{\od}{\widehat{\Sigma}}
\newcommand{\Si}{\Sigma}

\newcommand{\bC}{\mathsf C}
 
\newcommand{\cS}{\mathcal S}

\newcommand{\prj}{\operatorname{Proj}} 
\newcommand{\s}{\operatorname{Sal}} 

\newcommand{\vertex}{\operatorname{Vert}} 
\newcommand{\lk}{\operatorname{lk}}

\newcommand{\cp}{\mathcal {P}}

\newcommand{\Ga}{\Gamma}

\newcommand{\wtP}{\widetilde P}

\newcommand{\wtQ}{\widetilde Q}

\newcommand{\whC}{\widehat C}

\newcommand{\act}{\curvearrowright}

\newcommand{\Sii}{\Sigma^{(1)}}

\newcommand{\supp}{\operatorname{Supp}}
\newcommand{\id}{\operatorname{Id}}
\newcommand{\length}{\operatorname{length}}
\newcommand{\w}{\operatorname{wd}}
\newcommand{\aut}{\operatorname{Aut}}
\newcommand{\PM}{\operatorname{PMod}}
\begin{document}

\title{On spherical Deligne complexes of type $D_n$}
\author{Jingyin Huang}
\maketitle
\begin{abstract}
Let $\Delta$ be the Artin complex of the Artin group of type $D_n$. This complex is also called the spherical Deligne complex of type $D_n$. 
We show certain types of 6-cycles in the 1-skeleton of $\Delta$ either have a center, which is a vertex adjacent to each vertex of the 6-cycle, or a quasi-center, which is a vertex adjacent to three of the alternating vertices of the 6-cycle. This will be a key ingredient in proving $K(\pi,1)$-conjecture for several classes of Artin groups in a companion article.

As a consequence, we also deduce that certain 2-dimensional relative Artin complex inside the $D_n$-type Artin complex, endowed with the induced Moussong metric, is CAT$(1)$.
\end{abstract}

\section{Introduction}

\subsection{Background and motivation}
The \emph{Artin group} (or Artin-Tits group) with generating set $S=\{s_1,\ldots,s_n\}$, denoted by $A_S$, is a group with the following presentation:
$$
A_S=\langle s_1,\ldots, s_n\mid  \underbrace{s_is_js_i\cdots}_{m_{ij}}=\underbrace{s_js_is_j\cdots}_{m_{ij}}\rangle 
$$
where $i\neq j$ and $m_{ij}$ is either an integer $\ge 2$ or $\infty$. When $m_{ij}=\infty$, it means there is no relation between $s_i$ and $s_j$. The associated \emph{Coxeter group}, denoted by $W_S$, is the quotient of $A_S$ with extra relation $s^2_i=1$ for $1\le i\le n$. An Artin group is \emph{spherical} if the associated Coxeter group is finite. For example, the braid group on $n$-strand is a spherical Artin group, with the associated Coxeter group being the symmetry group on $n$ letters. 

Classical examples of Coxeter groups come from taking a geodesic triangle in $\mathbb S^2,\mathbb E^2$ or $\mathbb H^2$ with its angles being $\pi$ divided by an integer $\ge 2$, and considering the group generated by reflections along the three sides of the triangle. Geometrically, this leads to a tilling of $\mathbb S^2,\mathbb E^2$ or $\mathbb H^2$ by geodesic triangles. The underlying simplicial complex of this tilling is called the \emph{Coxeter complex}, which encodes fundamental combinatorial and geometric properties of Coxeter groups and can be defined for all Coxeter groups (even outside the classical cases) in a purely group theoretical way as follows. Given $s\in S$, let $W_{\hat s}$ be the subgroup of $W$ generated by $S\setminus\{s\}$. Then vertices of the Coxeter complex are in 1-1 correspondence with left cosets of form $\{g W_{\hat s}\}_{g\in W_S,s\in S}$, and a collection of vertices span a simplex if the corresponding collection of left cosets has a non-empty common intersection.

For an Artin group $A_S$, there is an analogues complex, called the \emph{Artin complex} and was defined in \cite{CharneyDavis}, whose vertices correspond to left cosets $\{g A_{\hat s}\}_{g\in A_S,s\in S}$, and simplices are defined in the same way as before. 
We say an Artin complex is \emph{spherical} if the associated Artin group is spherical.
In a Coxeter complex, each codimensional one face is contained in exactly two top-dimensional simplices (corresponding to the generators having order two); however, in an Artin complex, each codimensional one face is contained in finitely many top-dimensional simplices (corresponding to the generators having order $\infty$). In general, the geometry of Artin complexes is much more intricate.

While Artin groups and Coxeter groups have similar presentations, our knowledge of Artin groups is much more sparse compared to the Coxeter groups side. Actually, very basic questions on Artin groups remain widely open \cite{MR3203644}. 

Artin groups arise as fundamental groups of certain complex hyperplane arrangement complements, and a central conjecture in the study of Artin groups, due to Arnol'd, Brieskorn, Pham and Thom, predicts that these arrangement complements are $K(\pi,1)$-spaces for Artin groups. We refer to the survey article for more details \cite{paris2014k}. The $K(\pi,1)$-conjecture is also widely open. Deligne settled this conjecture for spherical Artin groups \cite{deligne}, where spherical Artin complexes play a key role in his work, so these complexes are also called the \emph{spherical Deligne complexes} by other authors.

Charney and Davis \cite{charney1995k} proved that if we know the Artin complexes associated with \emph{spherical} Artin groups are CAT$(1)$ with respect to a naturally defined metric, then the $K(\pi,1)$-conjecture holds true for \emph{all} Artin groups. Thus, it is of great interest to understand the geometry of spherical Artin complexes. However, there are no good methods of showing CAT$(1)$ in higher dimensional complexes. To circumvent this difficulty, we exploit different notions of curvature in simplicial complexes and proposed a strategy in \cite{huang2023labeled} of reducing the $K(\pi,1)$-conjecture for Artin groups to understanding short cycles in the 1-skeleton of spherical Artin complexes. In particular, to treat fairly general families of Artin groups, we need to understand how cycles of length $\le 6$ in spherical Artin complexes can be filled in the 2-skeleton.

In a previous article \cite{huang2023labeled}, we were able to understand the minimal filling of all 4-cycles in spherical Artin complexes; they follow a very simple pattern:
\begin{thm}
	\label{thm:4cycle}
	Suppose $A_S$ is an irreducible spherical Artin group. Then any embedded 4-cycles in the associated Artin complex can be filled in one of the following two ways in Figure~\ref{fig:4cycle}, one with two 2-simplices, and another with four 2-simplices.
\end{thm}

\begin{figure}[h]
	\centering
	\includegraphics[scale=0.9]{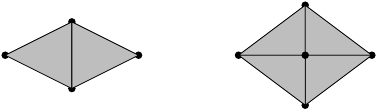}
	\caption{Filling 4-cycles.}
	\label{fig:4cycle}
\end{figure} 

However, the method in \cite{huang2023labeled} completely breaks down for filling longer cycles. The goal of this article is to develop new methods to fill 6-cycles in Artin complexes associated with the type $D_n$ Artin groups, which will be used as important ingredients in a companion article \cite{huang2024} to settle new cases of $K(\pi,1)$-conjectures. This companion article also contains results on filling 6-cycles in other types of spherical Artin complexes. However, we separate the case $D_n$ here as the method used in this case is of a different flavor compared to other cases.

\begin{figure}[h]
	\centering
	\includegraphics[scale=1]{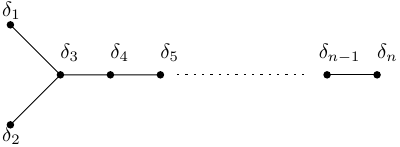}
	\caption{Dynkin diagram of type $D_n$.}
	\label{fig:D}
\end{figure}	

\subsection{Main results}
Let $\Lambda$ be the Dynkin diagram of the Artin group of type $D_n$; see Figure~\ref{fig:D}, with the generating set $S=\{\delta_1,\ldots,\delta_n\}$. Recall that $\Lambda$ encodes the presentation of $A_S$ in the following way: vertices of $\Lambda$ are in one to one correspondence with generators of $A_S$; $\delta_i$ and $\delta_j$ are adjacent in $\Lambda$ if there is a relation of the form $\delta_i\delta_j\delta_i=\delta_j\delta_i\delta_j$; and $\delta_i$ and $\delta_j$ are not adjacent if $\delta_i$ and $\delta_j$ commute. Let $\Delta_\Lambda$ be the Artin complex of type $D_n$, i.e. the Artin complex associated with the $D_n$-type Artin group. We assume $n\ge 3$. Then $\Delta_\Lambda$ is simply-connected, hence each 6-cycle in the 1-skeleton can be filled by a union of triangles in the 2-skeleton, though a priori the filling could be quite complicated. Our goal is to prove that, similar to Theorem~\ref{thm:4cycle}, for certain types of 6-cycles in $\Delta_\Lambda$, it is possible to construct fillings of them with a very simple combinatorial pattern. 

A vertex of $\Delta_\Lambda$ is of type $\hat \delta_i$ if it corresponds to a left coset of form $gA_{\hat \delta_i}$. Due to the limited symmetry of the Dynkin diagram of type $D_n$, the types of vertices in a 6-cycle encode subtle geometric information about this 6-cycle, and the combinatorial structure of minimal filling of 6-cycles is very sensitive to the types of its vertices.

It is natural to focus our attention on 6-cycles in $\Delta_\Lambda$ that are embedded and \emph{induced}, i.e., if two vertices of the 6-cycle are not adjacent in the 6-cycle, then they are not adjacent in $\Delta_\Lambda$; otherwise we are reduced to Theorem~\ref{thm:4cycle}. A \emph{center} of an $n$-cycle in $\Delta_\Lambda$ is a vertex that is adjacent to each vertex of this cycle.
Theorem~\ref{thm:4cycle} can be reformulated as each embedded and induced 4-cycle in $\Delta_\Lambda$ has a center. It is natural to ask whether the same holds for 6-cycles. While the answer is no in general, we identify the following classes of 6-cycles that do have a center. 

\begin{thm}(=Theorem~\ref{thm:weakflagD})
	\label{thm:weakflagDintro}
	Let $\Lambda$ be the Dynkin diagram of type $D_n$ with vertex set as in Figure~\ref{fig:ad}. Then any 6-cycle $\omega$ in the Artin complex $\Delta_\Lambda$ with its vertices alternating between types $\hat \delta_1$ and $\hat \delta_3$ has a center of type $\hat \delta_2$.
\end{thm}

A conjecture of Charney and Davis asserts that the $D_n$-type $\Delta_\Lambda$ endowed with the Moussong metric is CAT$(1)$ \cite{charney1995k,charney2004deligne}. Theorem~\ref{thm:weakflagDintro} has the following consequence, providing positive evidence towards this conjecture by showing a non-trivial cobounded subcomplex is CAT$(1)$.
\begin{cor}(=Corollary~\ref{cor:cat1})
	\label{cor:cat(1)dintro}
	Let $\Lambda$ be the Dynkin diagram of type $D_n$. Let $\Delta'\subset \Delta_\Lambda$ be the full subcomplex of $\Delta_\Lambda$ spanned by vertices whose types are in $\{\hat \delta_1,\hat \delta_2,\hat \delta_3\}$. Then $\Delta'$ with the induced Moussong metric  is CAT$(1)$.
\end{cor}
The $n=3$ case of this corollary is the main theorem of Charney \cite{charney2004deligne}.

An unexpected ingredient in the proof of Theorem~\ref{thm:weakflagDintro} is a result by Lyndon and Sch{\"u}tzenberger \cite{MR0162838} in the 1960s on solving certain types of equations on free groups.

Now we handle more general 6-cycles. A vertex $y$ is a \emph{quasi-center} for a 6-cycle with consecutive vertices $\{x_i\}_{i\in \mathbb Z/6\mathbb Z}$ in $\Delta_\Lambda$, if there exists $i\in \mathbb Z/6\mathbb Z$ such that $y$ is adjacent to each of $\{x_i,x_{i+2},x_{i+4}\}$. A quasi-center of a 6-cycle breaks up the 6-cycle into three 4-cycles. In view of Theorem~\ref{thm:4cycle}, we are done with producing a filling if we are able to find a quasi-center, and such filling is often optimal in an appropriate sense.

In general, there are 6-cycles in $\Delta_\Lambda$ that do not have a quasi-center. However, for studying the $K(\pi,1)$-conjecture, we only need to deal with a particular kind of 6-cycles, called zigzag 6-cycles. We postpone the somewhat technical definition of zigzag 6-cycles to Definition~\ref{def:zigzag}, and simply mention that conjecturally all the zigzag 6-cycles have a quasi-center (this is a reformulation of a conjecture of Haettel, see Section~\ref{sec:remark}), and this enables us to produce fillings that are efficient enough for application to the $K(\pi,1)$-conjecture. Cycles in Theorem~\ref{thm:weakflagDintro} are the simplest type of zigzag 6-cycles. We are able to handle more complicated zigzag 6-cycles when $n=4$.

\begin{thm}(=Theorem~\ref{thm:flagD4})
	\label{thm:flagD4intro}
	Let $\Lambda$ be the Dynkin diagram of type $D_4$ with a chosen leaf vertex $c$. Given a 6-cycle $\omega$ in the Artin complex $\Delta_\Lambda$ with its vertices alternating between having type $\hat c$ and not having type $\hat c$. Then there exists a quasi-center of $\omega$ that is adjacent to each of the type $\hat c$ vertices of $\omega$.
\end{thm}

\begin{cor}(=Corollary~\ref{cor:zigzag})
	\label{cor:zigzagintro}
	Let $\Lambda$ be the Dynkin diagram of type $D_n$ with $n=3,4$. Then any zigzag 6-cycle $\omega$ in  $\Delta_\Lambda$ has a quasi-center that is adjacent to each of the local max vertices of $\omega$. 
\end{cor}

Corollary~\ref{cor:zigzagintro} will be used as a key ingredient in \cite{huang2024} to answer a question of J. McCammond on the $K(\pi,1)$-conjecture of Artin groups with a complete bipartite Dynkin diagram. Theorem~\ref{thm:weakflagDintro} will be used as a key ingredient in \cite{huang2024} to settle the $K(\pi,1)$-conjecture for some hyperbolic type Artin groups.

\subsection{Discussion of proofs}
Let $\Lambda$ be as in Figure~\ref{fig:D}.
Let $\Delta_\Lambda=\Delta_S$ be the type $D_n$ Artin complex. Given a 6-cycle $\omega$ in $\Delta_\Lambda$ with consecutive vertices $\{x_i\}_{i\in \mathbb Z/6\mathbb Z}$ such that $x_i$ has type $\hat \delta_{f(i)}$ where $f:\mathbb Z/6\mathbb Z\to \{1,\ldots, n\}$, the information of this 6-cycle corresponds to the following equation:
\begin{equation}
	\label{eq}
	w_1w_2\cdots w_6=1,
\end{equation}
where $w_i\in A_i$ and $A_i$ is the subgroup generated by $S\setminus \{\delta_{f(i)}\}$.
Producing a minimal filling of this 6-cycle corresponds to describing the solution sets of this equation in a particular way. 

To prove both Theorem~\ref{thm:weakflagDintro} and Theorem~\ref{thm:flagD4intro}, we first try to establish a coarser ``a priori estimation'' to bound the complexity of each term of \eqref{eq}, or bound the complexity of other combinatorial information of \eqref{eq}. Then we carry out a fine combinatorial analysis to obtain the desired statement.

The idea to establish a priori estimation is to use projections/retractions in $A_S$ and $\Delta_S$. Think of \eqref{eq} as a loop in the Cayley graph of $A_S$, such that each term of the equation corresponds to a subpath of the loop that is contained in a left $A_{i}$-coset. Suppose the first term $w_1$ corresponds to the identity coset $A_{1}$ in $A_S$, and there is a (not necessarily group theoretic) retraction $r:A_S\to A_{1}$ such that the $r$-images of other cosets are ``small enough'' in $A_{1}$ and can be computed very explicitly, then we will gain control of the term $w_1$. 
There are several kinds of retraction $r:A_S\to A_1$ defined in the literature, either using non-positive curvature, or Garside structure, or some other group theoretic structure \cite{brady2000three,altobelli1998word,godelle2012k,charney2014convexity,blufstein2023parabolic}. One advantage of the retraction in \cite{godelle2012k,charney2014convexity,blufstein2023parabolic,godelle2023parabolic} is that it is easier to carry out a very explicit computation of the retraction image. However, the retraction image of other cosets in $A_{ab}$ is not always small enough, so sometimes such retraction does not give useful information. 

In the case of Theorem~\ref{thm:weakflagDintro}, it turns out that either the above retraction is enough to give enough information to solve the equation \eqref{eq} (see Proposition~\ref{prop:tight 6-cycle}), or we carry out a finer analysis as follows. 
Recall that the type $D_n$ Artin group can be written as a semi-direct product of a braid group and a free group; see, e.g., \cite{allcock2002braid,crisp2005artin}. So solving an equation like \eqref{eq} can be converted to finding an explicit set of solutions for certain equations in free groups. We showed that in the case that the retraction method does not work, one can convert \eqref{eq} into an equation of form $x^M=y^Nz^P$ in a finitely generated free group for $M,N,P\ge 2$, whose solution must be the obvious one by work of Lyndon and Sch{\"u}tzenberger \cite{MR0162838}; or \eqref{eq} reduces to a somewhat degenerate equation with only four terms, and the method in our previous work \cite{huang2023labeled} applies.

We have difficulty proving Theorem~\ref{thm:flagD4intro} using similar methods, due to the fact that converted version of \eqref{eq} involves term in both the free group and its automorphism group, which is much harder to solve.
Alternatively, besides the interpretation of type $D_n$-Artin groups as subgroups of Aut$(F_n)$ \cite{bardakov2007representations} as in the previous paragraph, they can also be interpreted as orbifold braid groups \cite{allcock2002braid} and subgroups of mapping class groups of surfaces \cite{perron1996groupe,crisp2005artin,soroko2020linearity}. 
We use the last interpretation: in the special case $n=4$ it was shown by Soroko \cite{soroko2020linearity} that $A_S$ is isomorphic to the pure mapping class groups of a surface $\cS$ which is the torus with two punctures and one boundary component. Using this, we can realize certain types of 6-cycles in the Artin complex $\Delta_S$ as 6-cycles in the arc complex of $\cS$. We caution the reader that the desired property we want for 6-cycles in $\Delta_S$, is not true for general 6-cycles in arc complexes, as well as curve complexes (it actually fails terribly for arc complexes \cite{webb2020contractible}). This does not mean the property we want to prove is not true, as we only need to deal with specific types of 6-cycles in the arc complex.

Given a vertex $x\in \Delta_S$, 
as a starting point, we can invoke either the subsurface projection \cite{masur2000geometry} or the Hatcher flow \cite{hatcher1991triangulations} to define a map from a suitable subcomplex $\Delta_S$ to the link of $x$ in $\Delta_S$. This gives us some control over each term in \eqref{eq}. The control is much weaker than what we want, which is not surprising because of the discussion in the previous paragraph. Though this gives a somewhat useful a priori estimation. 

The majority of the proof is to strengthen this control to the form we want, by using the finer combinatorial properties of these arcs to show that if some of these 6-arcs have complicated intersections, then it will force other arcs to have relatively few intersections. 
 
 \subsection{Speculation for filling more general $n$-cycles in $\Delta_S$} We end the introduction by formulating a conjecture on what should be the structure of minimal filling of more general $n$-cycles in more general spherical Artin complexes (regardless of whether we need these cycles in the proof of the $K(\pi,1)$-conjecture or not), based on the evidence we have from the previous results, from \cite{appel1983artin,charney2004deligne,haettel2021lattices,huang2023labeled,huang2024} and from an unpublished work of Crisp and McCammond on lattices of cut curves (which is explained in \cite{haettel2021lattices}).
 Let $A_S$ be an irreducible spherical Artin group. Let $\Delta_S$ be the associated Artin complex. The \emph{type sequence} of an $n$-cycle in $\Delta_S$, is defined to be the sequence of types of consecutive vertices in this $n$-cycle.
 It is natural to ask, given an $n$-cycle $\omega$ in $\Delta_\Lambda$, before we prove anything, how would we even guess what could be the possible combinatorial structure of a minimal filling disk for $\omega$ in the 2-skeleton of $\Delta_\Lambda$. To this end, we propose the following principle/conjecture.
 
 Let $\mathsf C_S$ be the associated Coxeter complex, with the natural action of $W_S$. A \emph{wall} of $\mathsf C_S$ is the fix point set of a reflection of $W_S$. Topologically, $\mathsf C_S$ is homeomorphic to a sphere, and a wall is homeomorphic to a codimension 1 sphere. So the complement of each wall has two connected components, called the \emph{open halfspaces} associated with this wall. An $n$-cycle $\omega$ is \emph{admissible} if all the $n$-cycles in $\mathsf C_S$ with the same type sequence as $\omega$ are contained in some open halfspaces. For example, any 4-cycle in $\Delta_S$ is admissible, and any zigzag 6-cycle in a type $D_n$ Artin complex is admissible.
 
 \begin{conj}
 	For any admissible $n$-cycle $\omega$ of $\Delta_S$, there is a companion $n$-cycle $\omega'$ in $\mathsf C_S$ with the same type sequence as $\omega$ such that the minimal disk filling $\omega$ in the 2-skeleton of $\Delta_S$ is no more complicated than the minimal disk filling $\omega'$ in the 2-skeleton of $\mathsf C_S$.
 \end{conj}
 
 \subsection{Structure of the article}
 In Section~\ref{sec:prelim} we collect some preliminaries. In Section~\ref{sec:AD}, we prove Theorem~\ref{thm:weakflagDintro} and Corollary~\ref{cor:cat(1)dintro}. In Section~\ref{sec:6cycle}, Section~\ref{sec:D4} and Section~\ref{sec:abccycle}, we prove Theorem~\ref{thm:weakflagD}. In Section~\ref{sec:remark}, we deduce Corollary~\ref{cor:zigzagintro} and discuss some ending remarks.
 
 \subsection{Acknowledgment}
 
 We thank Mladen Bestvina, Nathan Broaddus, Piotr Przytycki, and Nick Salter for teaching the author many things about arcs and curves on surfaces and for helpful and inspiring conversations. We thank Ruth Charney and Mike Davis for their helpful conversation on Artin groups and spherical Deligne complexes. We thank Thomas Haettel for explaining to the author his conjecture (Conjecture~\ref{conj:h}), which has been a source of inspiration. This article takes considerable influence from Charney's article \cite{charney2004deligne}, which we acknowledge with pleasure.
 
 The author is partially supported by a Sloan fellowship and NSF DMS-2305411. 
 The author thanks the Centre de Recherches Mathématiques in Montreal for the hospitality where part of the work was done.

\setcounter{tocdepth}{1}
\tableofcontents
\section{Preliminary}
\label{sec:prelim}
\subsection{Artin groups and Coxeter groups}
\label{subsec:Artin}
We use $A_S$ to denote the Artin group with generating set $S=\{s_1,\ldots,s_n\}$. Recall that the associated Dynkin diagram $\Lambda$ of $A_S$ is a simplicial graph with edge labeling such that its vertex set is $S$, moreover, 
\begin{enumerate}
	\item $s_i$ and $s_j$ are joined by an edge labeled by $m_{ij}$, if $\underbrace{s_is_js_i\cdots}_{m_{ij}}=\underbrace{s_js_is_j\cdots}_{m_{ij}}$ and $4\le m_{ij}<\infty$;
	\item $s_i$ and $s_j$ are joined by an edge without any label if $s_is_js_i=s_js_is_j$;
	\item $s_i$ and $s_j$ are joined by an edge labeled by $\infty$ if there are no relations between $s_i$ and $s_j$ in the standard presentation.
\end{enumerate}
It follows from the definition of Dynkin diagram that if two vertices in the diagram are not adjacent, then the associated generators commute.

We will also use $A_\Lambda$ to denote $A_S$.
For $S'\subset S$, we use $A_{S'}$ to denote the subgroup of $A_S$ generated by $S'$. Then $A_{S'}$ is also an Artin group \cite{lek}. We use $W_S$ to denote the Coxeter group associated with $A_S$. For $S'\subset S$, we define $W_{S'}$ in a similar way.
\subsection{Davis complexes}
\label{subsec:complex}
By a cell, we always mean a closed cell unless otherwise specified.

\begin{definition}[Davis complex]
	Given a Coxeter group $W_S$, let $\mathcal{P}$ be the poset of left cosets of spherical standard parabolic subgroups in $W_S$ (with respect to inclusion) and let $b\Si_S$ be the geometric realization of this poset (i.e.\ $b\Si_S$ is a simplicial complex whose simplices correspond to chains in $\mathcal{P}$). Now we modify the cell structure on $b\Si_S$ to define a new complex $\Sigma_S$, called the \emph{Davis complex}. The cells in $\Sigma_S$ are induced subcomplexes of $b\Si_S$ spanned by a given vertex $v$ and all other vertices which are $\le v$ (note that vertices of $b\Si_S$ correspond to elements in $\mathcal{P}$, hence inherit the partial order).
\end{definition}

Suppose $W_S$ is finite with $n$ generators. Then there is a canonical faithful orthogonal action of $W_S$ on the Euclidean space $\mathbb E^n$. Take a point in $\mathbb E^n$ with trivial stabilizer, then the convex hull of the orbit of this point under the $W_S$ action (with its natural cell structure) is isomorphic to $\Sigma_S$. In such case, we call $\Sigma_S$ a \emph{Coxeter cell}. In general Davis complex is a union of Coxeter cells.

The 1-skeleton of $\Sigma_S$ is the unoriented Cayley graph of $W_S$ (i.e.\ we start with the usual Cayley graph and identify the double edges arising from $s^2_i$ as single edges), and $\Sigma_S$ can be constructed from the unoriented Cayley graph by filling Coxeter cells in a natural way. Each edge of $\Sigma_S$ is labeled by a generator of $W_S$. We endow $\Sigma^{(1)}_S$ with the path metric with edge length $1$. 

A \emph{reflection} of $W_S$ is a conjugate of one of its generators. There is a natural action of $W_S$ on $b\Si_\Ga$ by simplicial automorphisms. The fix point set $H$ of a reflection $r$ is a subcomplex of $b\Si_\Ga$ (also viewed as a subset of $\Si_\Ga$), which is called a \emph{wall}. Then $\Si_\Ga\setminus H$ has exactly two connected components, exchanged by the action of $r$. Two vertices of $\Sigma_\Ga$ are \emph{separated} by a wall $H$ if they are in different connected components of $\Sigma_S\setminus H$. The distance between any two vertices with respect to the path metric on $\Sigma^{(1)}_S$ is the number of walls separating these two vertices. A wall is \emph{dual} to an edge if the wall and the edge have nonempty intersection. Two edges are \emph{parallel} if they are dual to the same wall.

If $S'\subset S$ is an induced subgraph, then $W_{S'}\to W_{S}$ induces an embedding $\Sigma_{S'}\to \Sigma_{S}$. The image of this embedding and their left translations are called \emph{standard subcomplexes} of type $S'$. There is a one to one correspondence between standard subcomplexes of type $S'$ in $\Sigma_{S}$ and left cosets of $W_{S'}$ in $W_{S}$.

Lemma~\ref{lem:gate} and Lemma~\ref{lem:pair gate} below are standard, see e.g.\ \cite{bourbaki2002lie} or \cite{davis2012geometry}, also \cite{dress1987gated}. We will use $d$ to denotes the path metric on the 1-skeleton of $\Si_\Ga$, with each edge having length 1.
Lemma~\ref{lem:gate} describes nearest point projection into the vertex set of a standard subcomplex.
\begin{lem}
	\label{lem:gate}
	Let $F$ be a standard subcomplex of $\Si_\Ga$ and let $x\in \Si_S$ be a vertex. Then there exists a unique vertex $x_F\in F$ such that $d(x,x_F)\le d(x,y)$ for any vertex $y\in F$, where $d$ denotes the path metric on the 1-skeleton of $\Sigma_S$. The vertex $x_F$ is called the \emph{projection} of $x$ to $F$, and is denoted $\prj_F(x)$. 	Moreover, let $\vertex F$ be the vertex set of a face $F$ of $\Si_S$. Let $E$ be another face of $\Si_S$. Then $\prj_E(\vertex F)=\vertex E'$ for some face $E'\subset E$. In this case we write $E'=\prj_E(F)$.
\end{lem}

\begin{definition}
	\label{def:projection1}
	Let $F$ be a fact of $\Si_S$. Lemma~\ref{lem:gate} gives a map $\pi:\vertex\Si_S\to\vertex F$ which extends to a retraction $\Pi_F:\Si_S\to F$ as follows. Note that for each face $E$ of $\Si_S$, $\pi(\vertex E)$ is the vertex set of a face $E'\subset F$. Then we extends $\pi$ to a map $\pi'$ from the vertex set of $b\Si_S$ to the vertex set of $bF$, by sending the barycenter of $E$ to the barycenter of $E'$. As $\pi'$ map vertices in a simplex to vertices in a somplex, it extends linearly to a map $\Pi_F:b\Si_S\cong \Si_S\to bF\cong F$.
\end{definition}

Now we consider the properties of nearest point sets between two faces.
\begin{lem}
	\label{lem:pair gate}
	Let $E$ and $F$ be faces of $\Si_S$. Define $$X=\{x\in \vertex E\mid d(x,\vertex F)=d(\vertex E,\vertex F)\}$$ and $$Y=\{y\in \vertex F\mid d(y,\vertex E)=d(\vertex E,\vertex F)\}.$$ Then 
	\begin{enumerate}
		\item there are faces $E'\subset E$ and $F'\subset F$ such that $X=\vertex E'$ and $Y=\vertex F'$;
		\item $\prj_E(\vertex F)=X$  and $\prj_F(\vertex E)=Y$;
		\item $\prj_E|_{\vertex F'}$ and $\prj_F|_{\vertex E'}$ gives a bijection and its inverse between $E'$ and $F'$;
		\item if $\mathcal W(E')$ is the collection of walls dual to an edge in $E'$, then $\mathcal W(E')=\mathcal W(F')=\mathcal W(E)\cap  \mathcal W(F)$;
		\item if $\mathcal W(E)=\mathcal W(F)$, then $E=E'$ and $F=F'$.
	\end{enumerate}
	In the situation of this lemma we will write $E'=\prj_E(F)$.
\end{lem}

In the situation of Lemma~\ref{lem:pair gate} (5), we will say $E$ and $F$ are \emph{parallel}. In this case, the bijection between $\vertex E$ and $\vertex F$ given by $\prj_E|_{\vertex F}$ and $\prj_F|_{\vertex E}$ are called \emph{parallel translation} between $E$ and $F$.

\begin{definition}
	\label{def:adj}
	Parallel faces $F$ and $F'$ of $\Si_S$ are \emph{adjacent} if $F\neq F'$ and if they are contained in a face $F_0$ with $\dim(F_0)=\dim(F)+1$.
\end{definition}

\begin{definition}
	\label{def:elementary segment}
	Let $\cq_S$ be the collection of all possible intersection of walls in $\Si_S$. Given an element $B\in \cq_S$, we say a face $F$ of $\Si_S$ is \emph{dual} to $B$ if $F\cap B$ is the barycenter of $F$. Note that two faces dual to the same element in $\cq_S$ are parallel.
	
	Let $B\in\cq_S$. Let $F$ and $F'$ be two adjacent parallel faces of $\Si_S$ that are dual to $B$. An \emph{elementary $B$-segment}, or an \emph{$(F,F')$-elementary $B$-segment} is a minimal positive path from a vertex $x\in F$ to $x'=p(x)\in F'$, where $p:F\to F'$ is parallel translation.
\end{definition}

\subsection{Oriented Davis complexes and Salvetti complexes}
\label{subsec:Sal}
Let $\mathcal P$ be the poset of faces of $\Si_S$ (under containment), and let $V$ be the vertex set of $\Si_S$. We now define the \emph{oriented Davis complex} $\widehat\Si_S$ as follows.
Consider the set of pairs $(F,v)\in \cp \times V$.  Define  an equivalence relation $\sim$ on this set by $$(F,v)\sim (F,v') \iff F=F' \text{\ and\ } \prj_F(v') = \prj_F(v).$$
Denote the equivalence class of $(F,v')$ by $[F,v']$.   Note  that each equivalence class $[F,v']$ contains a unique representative of the form $(F,v)$, with $v\in \vertex F$.  The \emph{oriented Davis complex} $\widehat\Si_S$ is defined as the regular CW complex given by taking  $\Si_S\times V$ (i.e., a disjoint union of copies of $\Si_S$) and then identifying faces $F\times v$ and $F\times v'$ whenever $[F,v]=[F,v']$, i.e.,
\begin{equation}
	\widehat\Si_S=( \Si_S\times V)/ \sim \ .
\end{equation}
For example, for each edge $F$ of $\Si_S$ with endpoints $v_0$ and $v_1$, we get two $1$-cells $[F,v_0]$ and $[F,v_1]$ of $\widehat\Si_S$ glued together along their endpoints $[v_0,v_0]$ and $[v_1,v_1]$.  So, the $0$-skeleton of $\widehat\Si_S$ is equal to the $0$-skeleton of $\widehat\Si_S$ while its $1$-skeleton is formed from the $1$-skeleton of $\widehat\Si_S$ by doubling each edge.  
There is a natural map $\pi:\widehat\Si_S\to\Si_S$ defined by ignoring the second coordinate. 

The definition of oriented Davis complex traced back to work of Salvetti \cite{s87}, so it is also called Salvetti complex by many other authors. The naming ``oriented Davis complex'' comes from an article of J. McCammond \cite{mccammond2017mysterious}, clarifying the relation between Salvetti's work and Davis complex, which suits better for our latter discussion. We will reserve the term ``Salvetti complex'' for a quotient of the oriented Davis complex.

Each edge of $\od_S$ has a natural orientation, namely, if $F=\{v_0,v_1\}$ is an edge of $\Si_S$, then $[F,v_0]$ is oriented so that $[v_0,v_0]$ is its initial vertex and $[v_1,v_1]$ is its terminal vertex.  An edge path in the $\od_S$ is \emph{positive} if each of its edges is positively oriented.

\begin{definition}
	\label{def:label}
	As the 1-skeleton of $\Si_S$ can be identified with the unoriented Cayley graph of $W_S$, each edge of $\Si_S$ is labeled by an element in the generating set $S$.
	We pull back the edge labeling from $\Si_S$ to $\od_S$ via the map $\pi:\od_S\to \Si_S$. For a subset $E$ in $\Sigma_S$ or $\od_S$, we define $\supp(E)$ to be the collection of labels of edges in $E$. Let $u$ be an edge path in $\Sii$ or a positive path in $\od_S$. Then reading off labels of edges of $u$ gives a word in the free monoid generated by $S$, which we denote by $\w(u)$. If $u$ is an arbitrary edge path in $\od_S$, then when an edge travels opposite to its orientation, we read off the inverse of the associated label. Then $\w(u)$ gives a word in the free group on $S$. 
\end{definition}

For each subcomplex $Y$ of $\Si_S$, we write $\widehat Y=p^{-1}(Y)$ and call $\widehat Y$ the subcomplex of $\od_S$ associated with $Y$.
A \emph{standard subcomplex} of $\widehat\Si_S$ is a subcomplex of $\od_S$ associated with a standard subcomplex of $\Si_S$. In other words, if $F\subset \Si_S$ is a standard subcomplex, then $\widehat F$ is the union of faces of form $E\times v$ in $\widehat \Si_S$ with $E\subset F$ and $v$ ranging over vertices in $\Si_S$.

\begin{lem}
	\label{lem:compactible}
	Let $E$ be a face of $\Si_S$ and let $F$ be a standard subcomplex of $\Si_S$.
	If $[E,v_1]=[E,v_2]$, then $[\prj_F(E),v_1]=[\prj_F(E),v_2]$.
\end{lem}

\begin{proof}
	Note that $[E,v_1]=[E,v_2]$ if and only if for each wall $H$ with $H\cap E\neq\emptyset$, $v_1$ and $v_2$ are in the same side of $H$. Thus for each wall $H$ dual to $\prj_F(E)$, $v_1$ and $v_2$ are in the same side of $H$. Now the lemma follows.
\end{proof}

We will be make use of the following important construction of Godelle and Pairs in \cite{godelle2012k}.
\begin{definition}
	\label{def:retraction}
	Let $F$ be a face in $\Si_S$. Then there is a retraction map $\Pi_{\widehat F}:\widehat\Si_S\to \widehat F$ defined as follows. Recall that $\widehat\Si_S=( \Si_S\times V)/ \sim$. For each $v\in V$, let $(\Si_S)_v$ be the union of all faces in $\widehat\Si_S$ of form $E\times v$ with $E$ ranging over faces of $\Si_S$. By Definition~\ref{def:projection1}, there is a retraction $(\Pi_F)_v:(\Si_S)_v\to F\times v$ for each $v\in V$. It follows from Lemma~\ref{lem:compactible} that these maps $\{(\Pi_F)_v\}_{v\in V}$ are compatible in the intersection of their domains. Thus they fit together to define a retraction $\Pi_{\widehat F}:\widehat\Si_S\to \widehat F$.
\end{definition}
The following is a direct consequence of the definition.
\begin{lem}
	\label{lem:retraction property}
	Take standard subcomplexes $E,F\subset \Si_S$. Then $\Pi_{\widehat F}(\widehat E)=\widehat{\Pi_F(E)}$.
\end{lem}

The action of $W_S$ on $\Si_S$ gives a free action of $W_S$ on $\Si_S$, whose quotient complex is denoted by $\s_S$. The fundamental group of $\s_S$ is $A_S$ and its 2-skeleton is the presentation complex of $A_S$, see e.g. \cite{paris2012k}. The fundamental group of $\od_S$ is the \emph{pure Artin subgroup} of $A_S$, i.e. the kernel of $A_S\to W_S$, as $\od_S$ is a regular cover of $\s_S$ corresponding such a subgroup.

\subsection{Artin complexes and Coxeter complexes}
\label{subsec:Coxeter}
Let $A_S$ be an Artin group with generating set $S$ and Dynkin diagram $\Lambda$. Let $W_S$ (or $W_\Lambda$) be the associated Coxeter group.
Recalled that the \emph{Artin complex}, introduced in \cite{CharneyDavis} and further studied in \cite{godelle2012k,cumplido2020parabolic}, defined as follows. For each $s\in S$, let $A_{\hat s}$ be the standard parabolic subgroup generated by $S\setminus\{s\}$. Let $\Delta_S$ be the simplicial complex whose vertex set is corresponding to left cosets of $\{A_{\hat s}\}_{s\in S}$. Moreover, a collection of vertices span a simplex if the associated cosets have nonempty common intersection. Then the complex $\Delta_S$ is called the \emph{Artin complex} associated with $A_S$. We will also write $\Delta_\Lambda$ for Artin complex. It follows from \cite[Proposition 4.5]{godelle2012k} that $\Delta_S$ is a flag complex.

 The Artin complex is an analogue of \emph{Coxeter complex} in the setting of Artin group. The definition of a Coxeter complex $\bC_S$ (or $\bC_\Lambda$) of a Coxeter group $W_S$ is almost identical to Artin complex, except one replaces $A_{\hat s}$ by $W_{\hat s}$, which is the standard parabolic subgroup of $W_S$ generated by $S\setminus \{s\}$. 
 
 Each vertex of $\bC_S$ or $\Delta_S$ corresponding a left coset of $W_{\hat s}$ or $A_{\hat s}$ has a \emph{type}, which is defined to be $\hat s=S\setminus \{s\}$. The \emph{type} of each face of $\bC_S$ or $\Delta_S$ is defined to be the subset of $S$ which is the intersection of the types of the vertices of the face. In particular, the type of each top-dimensional simplex is the empty set.

We record the following description of the Artin complex $\Delta_S$ in terms of $\od_S$, which will be used later.
\begin{remark}
	\label{rmk:alternative}
	Let $X$ be the universal cover of $\od_S$. A \emph{lift} of a standard subcomplex in $\od_S$ is a connected component of the inverse image of this subcomplex under the map $X\to\od_S$.
	Vertices of $\Delta_S$ are in 1-1 correspondence with lifts standard subcomplexes of $\od_S$ of type $\hat s$ for some $s\in S$. A collection of vertices span a simplex if their associated lifts have non-trivial common intersection. 
\end{remark}
We need to following procedure of interpreting a cycle in a Artin complex as a sequence of words in the associated Artin groups.
\begin{definition}
	\label{def:ncycle}
	Suppose $\Delta=\Delta_\Lambda$ is the Artin complex of the Artin group $A_\Lambda$ with Dynkin diagram $\Lambda$ and generating set $S$.
	A \emph{chamber} in $\Delta$ is a top-dimensional simplex in $\Delta$. There is a 1-1 correspondence between chambers in $\Delta$ and elements in $A_\Lambda$.	
	Let $\{x_n\}_{i=1}^4$ be consecutive vertices of an $n$-cycle $\omega$ in $\Delta$ and suppose $x_i$ has type $\hat a_i$ with $a_i\in \Lambda$. 
	For each edge of $\omega$, take a chamber of $\Delta$ containing this edge. We name these chambers by $\{\Theta_i\}_{i=1}^n$ with $\Theta_1$ containing the edge $\overline{x_1x_2}$. Each $\Theta_i$ gives an element $g_i\in A_\Lambda$. Then for $i\in \mathbb Z/n\mathbb Z$, $g_i=g_{i-1}w_{i}$ for $w_i\in A_{\hat a_i}$ (recall that $A_{\hat a_i}$ is defined to be $A_{S\setminus\{a_i\}}$). Thus $w_1w_2\cdots w_n=1$.
	The word $w_1\cdots w_n$ depends on the choice of $\{\Theta_i\}_{i=1}^n$. A different choice would lead to a word of form $u_1\cdots u_n$ such that there exist elements $q_i\in A_{S\setminus\{a_i,a_{i+1}\}}$ such that $u_i=q^{-1}_{i-1}w_i q_i$ for $i\in\mathbb Z/n\mathbb Z$. In this case we will say the words $u_1\cdots u_n$ and $w_1\cdots w_n$ are equivalent. 
\end{definition}

\subsection{Relative Artin complexes}
We recall the following notion from \cite{huang2023labeled}.
\begin{definition}
	Let $A_S$ be an Artin group with Dynkin diagram $\Lambda$. Let $S'\subset S$. The \emph{$(S,S')$-relative Artin complex $\Delta_{S,S'}$} is defined to be the induced subcomplex of the Artin complex $\Delta_S$ of $A_S$ spanned by vertices of type $\hat s$ with $s\in S'$. In other words, vertices of $\Delta_{S,S'}$ correspond to left cosets of $\{A_{\hat s}\}_{s\in S'}$, and a collection of vertices span a simplex if the associated cosets have nonempty common intersection.
	
	Let $\Lambda'$ be the induced subgraphs of $S$ spanned by $S'$. Then we will also refer an $(S,S')$-relative Artin complex as $(\Lambda,\Lambda')$-relative Artin complex, and denote it by $\Delta_{\Lambda,\Lambda'}$.
\end{definition}

The links of vertices in relative Artin complexes can be computed via the following simple observation \cite[Lemma 6.4]{huang2023labeled}.

\begin{lem}
	\label{lem:link}
	Let $\Delta$ be the $(\Lambda,\Lambda')$-relative Artin complex, and let $v\in \Delta$ be a vertex of type $\hat s$ with $s\in \Lambda'$. Let $\Lambda_s$ and $\Lambda'_s$ be the induced subgraph of $\Lambda$ and $\Lambda'$ respectively spanned all the vertices which are not $s$. 
	Then the following are true.
	\begin{enumerate}
		\item There is a type-preserving isomorphism between $\lk(v,\Delta)$ and the $(\Lambda_s,\Lambda'_s)$-relative Artin complex.
		\item Let $I_s$ be the union of connected components of $\Lambda_s$ that contain at least one component of $\Lambda'_s$. Then $\Lambda'_s\subset I_s$ and there is a type-preserving isomorphism between $\lk(v,\Delta)$ and the $(I_s,\Lambda'_s)$-relative Artin complex.
		\item Let $\{I_i\}_{i=1}^k$ be the connected components of $I_s$. Then $\lk(v,\Delta)=K_1*\cdots*K_k$ where $K_i$ is the induced subcomplex of $\lk(v,\Delta)$ spanned by vertices of type $\hat t$ with $t\in I_i$.
	\end{enumerate}
\end{lem}

\begin{definition}
	\label{def:labeled 4-wheel}
	Let $\Lambda$ be a Dynkin diagram which is a tree, with its vertex set $S$. Let $Z$ be a simplicial complex of type $S$.
	Let $X$ be the 1-skeleton of $Z$ with its vertex types as explained above. We say $Z$ satisfies the \emph{labeled 4-wheel condition} if for any induced 4-cycle in $X$ with consecutive vertices being $\{x_i\}_{i=1}^4$ and their types being $\{\hat s_i\}_{i=1}^4$, there exists a vertex $x\in X$ adjacent to each of $x_i$ such that the type $\hat s$ of $x$ satisfies that $s$ is in the smallest subtree of $\Lambda'$ containing all of $\{s_i\}_{i=1}^4$.
\end{definition}

\begin{thm}(\cite[Proposition 2.8]{huang2023labeled})
	\label{thm:4 wheel}
	Suppose $A_S$ is an irreducible spherical Artin group. Then $\Delta_S$ satisfies the labeled 4-wheel condition.
\end{thm}

\section{Six-cycles of type $(\hat \delta_1,\hat\delta_2)$}
\label{sec:AD}
The goal of this section is to prove weakly flagness for certain relative Artin complexes associated to Artin groups of type $D_n$, see Theorem~\ref{thm:weakflagD}.
\subsection{A criterion for filling certain 6-cycles}
\begin{prop}
	\label{prop:tight 6-cycle}
	Suppose $\Lambda$ is an irreducible spherical Dynkin diagram. Let $\Lambda'\subset \Lambda$ be a linear subgraph of type $A_3$ with its consecutive vertices being $\{t_1,t_2,t_3\}$. Let $V$ be the vertex set of the relative Artin complex $\Delta_{\Lambda,\Lambda'}$.  Take vertices $\{x_i\}_{i=1}^3$ in $V$ of type $\hat t_1$ and $\{y_i\}_{i=1}^3$ of type $\hat t_2$. Suppose $y_i$ is adjacent to both $x_i$ and $x_{i+1}$.  Let $\omega$ be the 6-cycle $x_1y_1x_2y_2x_3y_3$ in $\Delta_{\Lambda,\Lambda'}$, and let $\pi:\Delta_\Lambda\to\mathsf C_\Lambda$ be the map induced by quotienting $\Delta_\Lambda$ by the action of the pure Artin group. If $\pi(\omega)$ is not a single edge in $\Delta_\Lambda$, then $\omega$ has a center of type $\hat t_3$.
\end{prop}

\begin{proof}
It suffices to show $\{x_1,x_2,x_3\}$ is adjacent in $\Delta_{\Lambda,\Lambda'}$ to a common vertex $z$ in $V$. Indeed, if $z$ is of type $\hat t_3$, then we are done by Theorem~\ref{thm:4 wheel}. If $z$ is of type $\hat t_2$, then we take $z'$ to be a vertex of type $\hat t_3$ which is adjacent to $z$. By Lemma~\ref{lem:link}, $z'$ is adjacent to each of $\{x_1,x_2,x_3\}$, reducing to the previous case.

	We assume $\{x_1,x_2,x_3\}$ is pairwise distinct, so is $\{y_1,y_2,y_3\}$, otherwise the proposition reduces to Theorem~\ref{thm:4 wheel}.
	We only prove (1) as (2) is similar. Let $\omega$ and $\pi$ be as defined before. Then we can identify the Coxeter complex $\mathsf C_\Lambda$ as the dual of $\Si_\Lambda$.
	Suppose $\bar x_i=\pi(x_i)$ and $C_{\bar x_i}$ be the top-dimensional cell in $\Si_\Lambda$ dual to $\bar x_i$. Let $\whC_{\bar x_i}$ be the associated standard subcomplex of $\od_\Lambda$. Similarly, we define $\bar y_i, C_{\bar y_i}$ and $\whC_{\bar y_i}$. Let $w=w_1u_1w_2u_2w_3u_3$ be the word arising from the 6-cycle $\omega$ as in Definition~\ref{def:ncycle}. As $w=\id$, we know $w$ gives a null-homotopic loop $P_1Q_1P_2Q_2P_3Q_3$ in $\od_\Lambda$, with $P_i\subset \whC_{\bar x_i}$ and $Q_i\subset \whC_{\bar y_i}$ for each $i$. 
	
	Recall that the support of a subset of $\Si_\Lambda$ or $\od_\Lambda$ is defined in Definition~\ref{def:label}.
	We claim it suffices to show $P_1$ is homotopic rel endpoints in $\whC_{\bar x_1}$ to a concatenation of three paths $P_{11}P_{12}P_{13}$ with $t_1,t_2\notin \supp(P_{11}\cup P_{13})$ and $t_1,t_3\notin \supp(P_{12})$. Indeed, $P_1Q_1P_2Q_2P_3Q_3$ lifts to a loop $\wtP_1\wtQ_1\wtP_2\wtQ_2\wtP_3\wtQ_3$ in the universal cover $\Theta$ of $\od_\Lambda$. At the same time, $P_{11}P_{12}P_{13}$ lifts to a path $\wtP_{11}\wtP_{12}\wtP_{13}$ sharing the same starting point and endpoint of $\wtP_1$. As $P_{1j}$ is contained in a standard subcomplex of type $\hat t_2$ for $j=1,3$, and $P_{12}$ is contained in a standard subcomplex of type $\hat t_3$. By Remark~\ref{rmk:alternative}, we know the loop $\wtP_{11}\wtP_{12}\wtP_{13}(\wtP_1)^{-1}$ gives a 4-cycle $x_{11}x_{12}x_{13}x_1$ in $\Delta_\Lambda$, with $x_{1j}$ being of type $\hat t_2$ for $j=1,3$ and $x_{12}$ being of type $\hat t_3$.
	For $j=1,3$, $\wtP_{1j}$ and $\wtQ_j$ are contained in a common standard subcomplex of $\Theta$ of type $\hat t_2$. Thus $x_{11}=y_1$ and $x_{13}=y_3$ by Remark~\ref{rmk:alternative}. As $y_3$ is adjacent to both $x_3$ and $x_{12}$, by Lemma~\ref{lem:link} (3) applying to $\lk(y_3,\Delta_\Lambda)$, we know $x_{12}$ is adjacent to $x_3$ in $\Delta_\Lambda$, hence also in $\Delta_{\Lambda,\Lambda'}$. Similarly by $y_1$ is adjacent to both $x_2$ and $x_{12}$, we know that $x_{12}$ is adjacent to $x_2$.  So $x_{12}$ is adjacent to each of $\{x_1,x_2,x_3\}$.

	Note that $\pi(\omega)$ is bipartite between type $\hat t_1$ vertices and $\hat t_2$ vertices. Thus $\pi(\omega)$ can not contain any 3-cycle, which leaves us the following cases to consider.
	
	\smallskip
	\noindent
	\underline{Case 1: $\pi(\omega)$ is two edges.} There are two subcases. First we consider $\pi(\omega)$ has two type $\hat t_1$ vertices and one type $\hat t_2$ vertex. We can assume without loss of generality that $\{\bar x_1,\bar y_1,\bar x_2\}$ are pairwise distinct, $\bar y_2=\bar y_1=\bar y_3$ and $\bar x_3=\bar x_1$. By Lemma~\ref{lem:projection location} below, $\Pi_{C_{\bar x_2}}(C_{\bar x_1})\subset C_{\bar y_1}\cap C_{\bar x_2}$. Moreover, as $C_{\bar y_1}\cap C_{\bar x_2}\neq\emptyset$, we know $\Pi_{C_{\bar y_1}}(C_{\bar x_2})=C_{\bar y_1}\cap C_{\bar x_2}$. Thus by Lemma~\ref{lem:retraction property}, these statements hold with $C$ replaced by $\whC$. Consider $\Pi_{\whC_{\bar x_2}}(P_1Q_1P_2Q_2P_3Q_3)$, which is null-homotopic in $\whC_{\bar x_2}$. Note that $$\Pi_{\whC_{\bar x_2}}(P_1Q_1)\cup \Pi_{\whC_{\bar x_2}}(Q_2P_3Q_3)\subset \whC_{\bar y_1}\cap \whC_{\bar x_2},$$ and $\Pi_{\whC_{\bar x_2}}(P_2)=P_2$. Thus $P_2$ is homotopic rel endpoints in $\whC_{\bar x_2}$ to a path in $\whC_{\bar y_1}\cap \whC_{\bar x_2}$. This implies $y_1=y_2$, which contradicts our assumption in the beginning. The case when $\pi(\omega)$ has two type $\hat t_2$ vertices and one type $\hat t_1$ vertex is similar - we deduce two of $\{x_1,x_2,x_3\}$ are identical, leading to a contradiction. Thus this case is ruled out.
	
	\smallskip
	\noindent
	\underline{Case 2: $\pi(\omega)$ is three edges sharing a common vertex $\bar x$.} Then either $\bar x=\bar x_1=\bar x_2=\bar x_3$, or $\bar x=\bar y_1=\bar y_2=\bar y_3$. We will only treat the former case as the latter is similar. By Lemma~\ref{lem:projection location} below, $\Pi_{C_{\bar y_1}}(C_{\bar y_i})\subset C_{\bar y_1}\cap C_{\bar x_1}$ for $i=2,3$.  By Lemma~\ref{lem:retraction property}, the same statement holds with $C$ replaced by $\whC$.
	Consider $\Pi_{\whC_{\bar y_1}}(P_1Q_1P_2Q_2P_3Q_3)$, which is null-homotopic in $\whC_{\bar y_1}$. Note that $$\Pi_{\whC_{\bar x_2}}(P_1)\cup \Pi_{\whC_{\bar x_2}}(P_2Q_2P_3Q_3)\subset \whC_{\bar y_1}\cap \whC_{\bar x_1},$$ and $\Pi_{\whC_{\bar y_1}}(Q_1)=Q_1$. Thus $Q_1$ is homotopic rel endpoints in $\whC_{\bar y_1}$ to a path in $\whC_{\bar y_1}\cap \whC_{\bar x_1}$, which implies that $x_1=x_2$, contradiction. Thus this case is ruled out.
	
	\smallskip
	\noindent
	\underline{Case 3: $\pi(\omega)$ is a path with three edges.} We assume without loss of generality that $\bar x_1\bar y_1\bar x_2\bar y_2$ is a path with three edges, and $\bar x_3=\bar x_2$, $\bar y_3=\bar y_1$. By Lemma~\ref{lem:projection location},
	\begin{enumerate}
		\item $\Pi_{C_{\bar x_1}}(C_{\bar y_1})=C_{\bar x_1}\cap C_{\bar y_1}$;
		\item $\Pi_{C_{\bar x_1}}(C_{\bar x_2})\subset C_{\bar x_1}\cap C_{\bar y_1}$;
		\item either $t_3\notin \supp(\Pi_{C_{\bar x_1}}(C_{\bar y_2}))$ or $t_2\notin \supp(\Pi_{C_{\bar x_1}}(C_{\bar y_2}))$.
	\end{enumerate}
	By Lemma~\ref{lem:retraction property}, the above three items hold true if we replace $C$ by $\whC$. 
	Consider $\Pi_{\whC_{\bar x_1}}(P_1Q_1P_2Q_2P_3Q_3)$, which is null-homotopic in $\whC_{\bar x_1}$. 
	Note that $\Pi_{\whC_{\bar x_1}}(Q_1P_2)\subset \whC_{\bar x_1}\cap \whC_{\bar y_1}$, hence $t_1,t_2\notin\supp(\Pi_{\whC_{\bar x_1}}(Q_1P_2))$. Similarly,
	$t_1,t_2\notin \supp(\Pi_{\whC_{\bar x_1}}(P_3Q_3))$. 
	\begin{itemize}
		\item If $t_3\notin \supp(\Pi_{C_{\bar x_1}}(C_{\bar y_2}))$, then $t_1,t_3\notin \supp(\Pi_{\whC_{\bar x_1}}(Q_2))$. Thus assumption of the claim before Case 1 is satisfied and we are done.
		\item If $t_2\notin \supp(\Pi_{C_{\bar x_1}}(C_{\bar y_2}))$, then $t_1,t_2\notin \supp(\Pi_{\whC_{\bar x_1}}(Q_2))$. Then $P_1$ is homotopic rel endpoints in $\whC_{\bar x_1}$ to a path in $\whC_{\bar x_1}\cap \whC_{\bar y_1}$, which implies $y_1=y_3$, contradiction.
	\end{itemize}
	
	\smallskip
	\noindent
	\underline{Case 4: $\pi(\omega)$ is made of four or five edges.} If $\pi(\omega)$ is made of four edges, as $\pi(\omega)$ can not contain 3-cycle, it is an embedded 4-cycle of type $\hat t_1\hat t_2\hat t_1\hat t_2$ in $\mathsf C_\Lambda$. However, Theorem~\ref{thm:4 wheel} implies that any 4-cycle in $\Delta_\Lambda$ of type $\hat t_1\hat t_2\hat t_1\hat t_2$ is degenerate. As there is an embedding $\mathsf C_\Lambda\to \Delta_\Lambda$ preserving type of vertices, we know $\pi(\omega)$ is a degenerate 4-cycle. Thus the case that $\pi(\omega)$ is made of four edges is ruled out. If $\pi(\omega)$ is made of five edges, then it is an embedded 4-cycle together with an additional edge glued to one vertex of this 4-cycle. By the same argument as before, the 4-cycle must be degenerate.

	\smallskip
	\noindent
	\underline{Case 5: $\pi(\omega)$ is made of six edges.} Then $\pi(\omega)$ is an embedded 6-cycle. By Lemma~\ref{lem:projection location},
	\begin{itemize}
		\item $\Pi_{C_{\bar x_1}}(C_{\bar y_1})=C_{\bar x_1}\cap C_{\bar y_1}$ and $\Pi_{C_{\bar x_1}}(C_{\bar y_3})=C_{\bar x_1}\cap C_{\bar y_3}$;
		\item $\Pi_{C_{\bar x_1}}(C_{\bar x_2})\subset C_{\bar x_1}\cap C_{\bar y_1}$ and $\Pi_{C_{\bar x_1}}(C_{\bar x_3})\subset C_{\bar x_1}\cap C_{\bar y_3}$;
		\item either $t_3\notin \supp(\Pi_{C_{\bar x_1}}(C_{\bar y_2}))$ or $t_2\notin \supp(\Pi_{C_{\bar x_1}}(C_{\bar y_2}))$.
	\end{itemize}
	We still consider the projection of $P_1Q_1P_2Q_2P_3Q_3$ to $\whC_{\bar x_1}$, and the discussion is identical to Case 3.
\end{proof}

\begin{lem}
	\label{lem:projection location}
	Let $\Si$ be the Davis complex of a finite Coxeter group generated by $S=\{s_1,\ldots,s_n\}$. We label edges of $\Si$ by elements in $S$. Let $\{s_i,s_{i+1},s_{i+2}\}$ be three generators such that they generates a Coxeter group of type $A_3$, and $s_i,s_{i+2}$ commute.

	Let $\{C_i\}_{i=1}^4$ be four pairwise distinct  faces of $\Si$ such that
	\begin{enumerate}
		\item $C_1$ and $C_3$ are of type $\hat s_i$, and $C_2$ and $C_4$ are of type $\hat s_{i+1}$;
		\item $C_i\cap C_{i+1}\neq\emptyset$ for $i=1,2,3$.
	\end{enumerate}
	Then the following are true:
	\begin{enumerate}
		\item $\Pi_{C_{1}}(C_{3})\subset C_{1}\cap C_{2}$, and $\Pi_{C_2}(C_4)\subset C_2\cap C_3$;
		\item either $s_{i+2}\notin \supp(\Pi_{C_{1}}(C_{4}))$ or $s_{i+1}\notin \supp(\Pi_{C_{1}}(C_{4}))$.
	\end{enumerate}
\end{lem}
\begin{proof}
In the following proof, we will write $S_1\perp S_2$ for subsets $S_1,S_2\subset S$ if $S_1\cap S_2=\emptyset$ and each element in $S_1$ commutes with every element in $S_2$.
	
	For (1), we only prove the first part, as the second part is similar. Let $u$ be a geodesic path in the 1-skeleton of $\Si$ connecting a vertex in $C_{1}$ and a vertex in $C_3$ such that $\length(u)=d(C_{1},C_3)$. As $u$ leaves $C_{1}$, $s_i\in \supp(u)$. Given a subset $S'$ of $S$, an \emph{irreducible component} of $S'$ is a maximal subset $S''$ of $S'$ such that $W_{S''}$ is irreducible.
	Let $I$ be an irreducible component of $\supp(\Pi_{C_{1}}(C_3))$ which is not contained in $\supp(C_2)$. Then $s_{i+1}\in I$. Note that both $C_{1}$ and $C_3$ have non-empty intersection with $C_{2}$, then we can apply \cite[Lemma 5.8 (2)]{huang2023labeled} to deduce $\supp(u)\perp I$. As $s_i\in \supp(u)$, this is a contradiction. So $\supp(\Pi_{C_{1}}(C_{3}))\subset \supp(C_{2})$. We also know $\Pi_{C_{1}}(C_{3})$ is a face of $\Si$ that contains a vertex in $C_1\cap C_2$ by \cite[Lemma 5.8 (1)]{huang2023labeled}. Then (1) follows. 
	
We prove two claims before going to (2). Define a word $w$ in the free monoid on $S$ has property $(*)$ if $w$ is a concatenation of $w_1$ and $w_2$ such that $s_{i+1}\notin \supp(w_1)$ and $s_i\notin \supp(w_2)$.  The first claim is that if two reduced words $w$ and $w'$ represent the same element of $W_\Lambda$ and $w$ has property $(*)$, then $w'$ has property $(*)$. To see the claim, we only need to consider the special case when $w'$ is obtained from $w$ by applying a single relation of $W_\Gamma$ of form $s_js_{j'}\cdots=s_{j'}s_j\cdots$. If $s_j$ and $s_{j'}$ commute, then either $s_i\notin \{s_j,s_{j'}\}$ or $s_{i+1}\notin \{s_j,s_{j'}\}$. By enlarging or shrinking $w_1$, we can assume the subword $s_js_{j'}$ is contained in one of $w_1$ or $w_2$ and the rest is clear. Now we assume $s_j$ and $s_{j'}$ do not commute. It only matters when the left side $w=s_js_{j'}\cdots$ of the relator do not appear as a subword of $w_1$ as well as $w_2$. If $w_1$ contains the first two letters of $w$, then we can enlarge $w_1$ so that $w$ is a subword of $w_1$. If $w_1$ only contains the first letter of $w$, then $w_2$ contains the last two letters of $w$ and we can enlarge $w_2$ so that it contains $w$ as a subword.

	The second claim is that there is an edge path $u$ from a vertex in $C_{1}$ to a vertex in $C_4$ such that $\length(u)=d(C_{1},C_4)$ and 
	$\w(u)$ satisfies property $(*)$ ($\w(u)$ is defined in Definition~\ref{def:label}). To see this, consider an edge path $\bar u=\bar u_1\bar u_2$ from a vertex in $C_{1}$ to a vertex in $C_4$ such that $\bar u_1\subset C_2$ and $\bar u_2\subset C_{3}$. Such $\bar u$ always exists. Let $\bar w=\w(\bar u)$ and $\bar w_i=\w(\bar u_i)$ for $i=1,2$. Then $\bar w$ satisfies property $(*)$.
	By the deletion condition for Coxeter groups, we can assume $\bar w$ is a reduced word, while maintaining property $(*)$. This means we replace $\bar u$ by a geodesic path (still denote $\bar u$) connecting the endpoints of $\bar u$. Let $e$ be the first edge of $\bar u$ such that $e$ is parallel to an edge in $C_{1}$. We write $\bar u=\bar u' e\bar u''$. Let $\hat u$ be the image of $\bar u'$ under the reflection along the hyperplane dual to $e$. Then $\w(\hat u)=\w(\bar u')$. We replace $\bar u$ by $\hat u\bar u''$ (still denoted by $\bar u$), which is still an edge path from a vertex of $C_{1}$ to a vertex of $C_4$. This has the effect of removing a letter from $\bar w$, thus the new $\bar w$ still has property $(*)$. Repeating this procedure will produce a geodesic edge path $u$ from a vertex of $C_{1}$ to a vertex of $C_4$ such that no edges of $u$ is parallel to an edge in $C_{1}$ or an edge in $C_4$, thus $\length(u)=d(C_{1},C_4)$. By construction, $\w(u)$ satisfies condition $(*)$, which proves the claim.
	
	Now we argue by contradiction that $s_{i+1},s_{i+2}\in \supp(\Pi_{C_{1}}(C_{4}))$. Let $u$ be as in the second claim with starting point $x\in C_1$ and endpoint $y\in C_4$. Let $e_j$ be the edge containing $x$ labeled by $s_j$ for $j=i+1,i+2$.
	Let $F$ be the 2-face spanned by $e_{i+1}$ and $e_{i+2}$. Let $B$ be the element in $\cq_S$ dual to $F$ (see Definition~\ref{def:elementary segment}).
	As $\Pi_{C_{1}}(C_{4})$ is a face of $C_{1}$, $F\subset \Pi_{C_{1}}(C_{4})$. By \cite[Lemma 5.4]{huang2023labeled}, there exists a geodesic path $u'$ such that
	\begin{enumerate}
		\item $u'$ and $u$ have the same endpoints;
		\item $u'=f_1f_2\cdots f_k$ where each $f_i$ is an elementary $B$-segment (see Definition~\ref{def:elementary segment}).
	\end{enumerate}
	As $\w(f_1)$ must start with $s_i$ and $W_{s_i,s_{i+1},s_{i+2}}$ is a Coxeter group of type $A_3$, we know $\w(f_1)=s_is_{i+1}s_{i+2}$. As $\w(u')$ satisfies condition $(*)$ by the first claim, we know $s_i\notin \supp(f_i)$ for $i\ge 2$. Let $F_i$ be the face parallel to $F$ containing the endpoint of $f_i$. Then $\supp(F_1)=\{s_i,s_{i+1}\}$. If the label of the initial edge of $f_2$ does not commute with at least one of $\{s_i,s_{i+1}\}$, then  $s_i\in \supp(f_2)$, which is a contradiction. Thus we have $\supp(f_2)\perp \{s_i,s_{i+1}\}$. Thus $\supp(F_2)=\supp(F_1)$. Repeating this argument, we know $\supp(F_k)=\supp(F_1)$. This is a contradiction as $F_k\subset C_4$, but $s_{i+1}\notin \supp(C_4)$ by the first assumption of the lemma. 
\end{proof}

\subsection{The $D_n$-type case}
We first review the process of treating $D_n$ as a semi-direct product of a free group $F_{n-1}$ of rank $n-1$ and an $n$-strand braid group, through the braid monodromy representation of the braid group into $\aut(F_{n-1})$ \cite{crisp2005artin}. 
Let $\{\delta_1,\ldots,\delta_n\}$ be generators of Artin group $A(D_n)$ of type $D_n$ as in Figure~\ref{fig:ad}. Let $\{\alpha_1,\ldots,\alpha_{n-1}\}$ be generators of Artin group $A(A_{n-1})$ of type $A_{n-1}$ as in Figure~\ref{fig:ad}. Let $\{\beta_1,\ldots,\beta_{n-1}\}$ be generators of a free group $F_{n-1}$ of rank $n-1$. 

\begin{figure}[h]
	\includegraphics[scale=1]{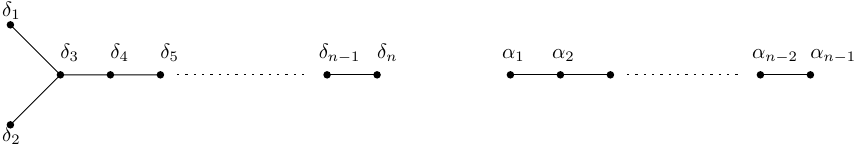}
	\caption{Dynkin diagram of type $D_n$ and $A_{n-1}$.}
	\label{fig:ad}
\end{figure}

Following \cite[Section 2]{crisp2005artin}, define a homomorphism $\rho:A_{n-1}\to \aut(F_{n-1})$ by:
\begin{equation*}
	\rho(\alpha_1): 
	\begin{cases}
		\beta_1\to \beta_1\\
	\beta_i\to \beta^{-1}_1\beta_i \ \ \ i\ge 2
	\end{cases}
\end{equation*}
and, for $2\le i\le n-1$,
\begin{equation*}
	\rho(\alpha_i): 
	\begin{cases}
		\beta_{i-1}\to \beta_i\\
		\beta_i\to \beta_i\beta^{-1}_{i-1}\beta_i\\
		\beta_j\to \beta_j \ \ \ j\neq i,i-1.
	\end{cases}
\end{equation*}
Define $\varphi: \{\delta_1,\ldots,\delta_n\}\to F_{n-1}\rtimes_\rho A(A_{n-1})$ by 
\begin{equation*}
	\begin{cases}
	\varphi(\delta_1)=\beta_1\alpha_1\\
	\varphi(\delta_i)=\alpha_{i-1} \ \ \ 2\le i\le n.
	\end{cases}
\end{equation*}
Define $\phi: \{\beta_1,\ldots,\beta_{n-1}\}\cup \{\alpha_1,\ldots,\alpha_{n-1}\}\to A(D_n)$ by
\begin{equation*}
	\begin{cases}
	\phi(\beta_i)=\delta_{i+1}\delta_i\cdots\delta_3\delta_1\delta^{-1}_2\delta^{-1}_3\cdots \delta^{-1}_i\delta^{-1}_{i+1}\ \ \ 1\le i\le n-1\\
	\phi(\alpha_i)=\delta_{i+1} \ \ \ 1\le i\le n-1.
	\end{cases}
\end{equation*}

\begin{prop}[\cite{crisp2005artin}]
	\label{prop:iso}
The map $\varphi$ extends to an isomorphism from $A(D_n)$ to $F_{n-1}\rtimes_\rho A(A_{n-1})$. The map $\phi$ extends to an isomorphism from $F_{n-1}\rtimes_\rho A(A_{n-1})$ to $A(D_n)$. The maps $\varphi$ and $\phi$ are inverses of each other.
\end{prop}

\begin{thm}
	\label{thm:weakflagD}
Let $\Lambda$ be the Dynkin diagram of type $D_n$ with vertex set $S=\{\delta_i\}_{i=1}^n$ as in Figure~\ref{fig:ad}. Given a 6-cycle $\omega$ with its vertices alternating between type $\hat \delta_1$ and $\hat \delta_3$. Then $\omega$ has a center  of type $\hat\delta_2$.
\end{thm}

\begin{proof}
Let consecutive vertices of $\omega$ be $x_1,y_1,x_2,y_2,x_3,y_3$ such that $\{x_i\}_{i=1}^3$ are of type $\hat\delta_1$ and $\{y_i\}_{i=1}^3$ are of type $\hat\delta_3$. We assume $\omega$ is an embedded 6-cycle,  otherwise the theorem reduces to Theorem~\ref{thm:4 wheel}. Let $w=w_1u_1w_2u_2w_3u_3$ be the word arising from the 6-cycle $\omega$ as in Definition~\ref{def:ncycle}. Then for each $i$, $w_i\in A_{\hat \delta_1}$ and $u_i\in A_{\hat \delta_3}$. As $A_{\hat \delta_3}=A_{\delta_1}\oplus A_{\delta_2}\oplus A_{\{\delta_4,\delta_5,\ldots,\delta_n\}}$, we know each $u_i$ can be written as a product two commuting elements, one of them is a power of $\delta_1$, another belongs to $A_{\hat \delta_1}$. Thus up to passing to an equivalent word as in the sense of Definition~\ref{def:ncycle}, we can assume $w=w_1\delta^{k_1}_1w_2\delta^{k_2}_1w_3\delta^{k_3}_1$. As $\{x_1,x_2,x_3\}$ is pairwise distinct, $k_i\neq 0$ for $1\le i\le 3$.

Through the isomorphism between $A(D_n)$ and $F_{n-1}\rtimes_\rho A(A_{n-1})$ as in Proposition~\ref{prop:iso}, we view $w_i$ and $\delta_1$ as elements in $F_{n-1}\rtimes_\rho A(A_{n-1})$. Each element in $F_{n-1}\rtimes_\rho A(A_{n-1})$ can be uniquely written as $a\cdot b$ with $a\in F_{n-1}$ and $b\in A(A_{n-1})$. Then $w_i=1\cdot w_i$ for $i=1,2,3$ and $\delta_1=\beta_1\cdot \alpha_1$. From $$w_1\delta^{k_1}_1w_2=\delta^{-k_3}_1w^{-1}_3\delta^{-k_2}_1,$$
we know that  
$$
w_1(\beta_1\alpha_1)^{k_1}w_2=(\beta_1\alpha_1)^{-k_3}w^{-1}_3(\beta_1\alpha_1)^{-k_2}.
$$
As $\beta_1$ and $\alpha_1$ commute (by looking at their image in $A(D_n)$), 
$$
(w_1\beta_1w^{-1}_1)^{k_1}\cdot w_1\alpha^{k_1}_1w_2= \beta^{-k_3}_1(\alpha^{-k_3}_1w^{-1}_3\beta_1w_3\alpha^{k_3}_1)^{-k_2}\cdot \alpha^{-k_3}_1w^{-1}_3\alpha^{-k_2}_1.
$$
Note that $w_1\beta_1w^{-1}_1$ and $\alpha^{-k_3}_1w^{-1}_3\beta_1w_3\alpha^{k_3}_1$ are non-trivial words in $F_{n-1}$, and the following equation holds true in $F_{n-1}$:
$$
(w_1\beta_1w^{-1}_1)^{k_1}=\beta^{-k_3}_1(\alpha^{-k_3}_1w^{-1}_3\beta_1w_3\alpha^{k_3}_1)^{-k_2}.
$$
By a result of Lyndon and Sch{\"u}tzenberger \cite{MR0162838}, at least one of the following two possibilities happen:
\begin{enumerate}
	\item there exists $1\le i\le 3$ such that $|k_i|=1$;
	\item both $w_1\beta_1w^{-1}_1$ and $\alpha^{-k_3}_1w^{-1}_3\beta_1w_3\alpha^{k_3}_1$ are nonzero powers of $\beta_1$.
\end{enumerate}
First suppose the first possibility happens. We assume without loss of generality that $|k_1|=1$. Let $P_1Q_1P_2Q_2P_3Q_3$ be the path as in the beginning of the proof of Proposition~\ref{prop:tight 6-cycle} corresponding to $w_1\delta^{k_1}_1w_2\delta^{k_2}_1w_3\delta^{k_3}_1$. Then $Q_1$ is a single edge labeled by $\hat\delta_1$. Let $\bar x_1,\bar x_2$, $\whC_{\bar x_1}$ and $\whC_{\bar x_2}$ be as in the proof of Proposition~\ref{prop:tight 6-cycle}. Then starting from a vertex of $\whC_{\bar x_1}$, crossing an edge labeled by $\hat\delta_1$ corresponding to $Q_1$, leading us to a vertex of $\whC_{\bar x_2}$. This means $\whC_{\bar x_1}\neq \whC_{\bar x_2}$. In particular $\bar x_1\neq \bar x_2$. Then we are done by Proposition~\ref{prop:tight 6-cycle}.
 
Now we consider the second possibility, namely both $w_1\beta_1w^{-1}_1$ and $\alpha^{-k_3}_1w^{-1}_3\beta_1w_3\alpha^{k_3}_1$ are nonzero powers of $\beta_1$. We first remind some facts about parabolic subgroups as follows. For each standard parabolic subgroup $A_{S'}$ of $A_S$, let $\Delta_{S'}$ be its Garside element, and let $c_{S'}$ be the smallest positive power of $\Delta_{S'}$ that is contained in the center of $A_{S'}$. If $ \mathsf{P}=gA_{S'}g^{-1}$ is a parabolic subgroup of $A_S$, we define $c_\mathsf{P}=gc_{S'}g^{-1}$. 	By \cite[Lemma 33]{cumplido2019minimal}, for $X,Y\subset S$, $g^{-1}A_Xg=g^{-1}A_Yg$ if and only if $g^{-1}C_Xg=c_Y$, in particular, if $gA_{S'}g^{-1}=g_1A_{S'}g^{-1}_1$, then $gc_{S'}g^{-1}=g_1c_{S'}g^{-1}_1$, hence $c_\mathsf{P}$ is well-defined.

As $w_1\beta_1w^{-1}_1$ is a primitive element in $F_{n-1}$, $w_1\beta_1w^{-1}_1=\beta^{\pm 1}_1$. Let $\mathsf{P}$ be the smallest parabolic subgroup of $A(D_n)$ containing $\beta_1$, whose existence is guaranteed by \cite{cumplido2019parabolic}. Note that $\mathsf{P}$ is also the smallest parabolic subgroup of $A(D_n)$ containing $\beta^{-1}$.
As $\beta=\delta_1\delta^{-1}_2$, $\mathsf{P}$ is contained in the standard parabolic subgroup $\mathsf{P}_{12}$ generated by $\delta_1$ and $\delta_2$. As $\mathsf{P}_{12}$ is abelian, we deduce from \cite{blufstein2023parabolic} that $\mathsf{P}=\mathsf{P}_{12}$. The minimality of $\mathsf{P}$ implies that $w_1\mathsf{P}_{12}w^{-1}=\mathsf{P}_{12}$. By \cite[Lemma 33]{cumplido2019minimal}, $w_1c_{\mathsf{P}}w^{-1}=c_{\mathsf{P}}$. Now let $\mathsf{P}'$ be the smallest parabolic subgroup of $w_1$. Then $\mathsf{P}'\subset A_{\hat \delta_1}$. Moreover, by the same argument as before, $c_{\mathsf{P}}c_{\mathsf{P'}}c^{-1}_{\mathsf{P}}=c_{\mathsf{P'}}$. As $c_{\mathsf{P'}}\in A_{\hat \delta_1}$ and $c_{\mathsf{P}}\in A_{\hat \delta_3}$, by \cite[Proposition 8.7]{huang2023labeled}, there exists $g\in A_{\hat \delta_1}\cap A_{\hat \delta_3}$ and $s\in \{\delta_1,\delta_3\}$ such that 
$$ gc_{\mathsf{P}}g^{-1}\in A_{\hat s}\ \  \ \mathrm{and}\  \ \  gc_{\mathsf{P}'}g^{-1}\in A_{\hat s}.
$$	
If $s=\delta_1$, then $c_{\mathsf{P}}\in g^{-1}A_{\hat\delta_1}g=A_{\hat \delta_1}$. On the other hand, however, this is impossible as $c_{\mathsf{P}}=\delta_1\delta_2$. If $s=\delta_3$, then $c_{\mathsf{P}'}\in g^{-1}A_{\hat \delta_3}g=A_{\hat\delta_3}$. By \cite[Proposition 2.1]{godelle2003normalisateur}, $w_1\in \mathsf{P}'\subset A_{\hat\delta_3}$. Thus $w_1\in A_{\hat \delta_3}\cap A_{\hat \delta_1}$, which implies that $y_1=y_3$, contradicting our assumption that $\{y_1,y_2,y_3\}$ are pairwise distinct. In summary $w_1\beta_1w^{-1}_1$ can not be a nonzero power of $\beta_1$. This finishes the proof.
\end{proof}
We record the following curious consequence of Theorem~\ref{thm:weakflagD}, Theorem~\ref{thm:4 wheel} and \cite[Theorem 5.2]{goldman2023cat}.
\begin{cor}
	\label{cor:cat1}
Let $\Lambda$ be the Dynkin diagram of type $D_n$ with vertex set $S$. Let $\Lambda'\subset\Lambda$ be the linear subgraph spanned by $\{\delta_1,\delta_2,\delta_3\}$ in Figure~\ref{fig:ad}. Then $\Delta_{\Lambda,\Lambda'}$ with induced Moussong metric from $\Delta_\Lambda$ is CAT$(1)$.
\end{cor}

\section{A remark on 5-cycles and 6-cycles in the $D_4$ Artin complex}
\label{sec:6cycle}
In Section~\ref{sec:6cycle}, Section~\ref{sec:D4} and Section~\ref{sec:abccycle}, we will prove the following. 

\begin{thm}
	\label{thm:flagD4}
	Let $\Lambda$ be the Dynkin diagram of type $D_4$ with a chosen leaf vertex $c$. Given a 6-cycle $\omega$ in the Artin complex $\Delta_\Lambda$ with its vertices alternating between having type $\hat c$ and not having type $\hat c$. Then there exists a quasi-center of $\omega$ which is adjacent to each of the type $\hat c$ vertices of $\omega$.
\end{thm}

\begin{definition}
	Let $\omega$ be an embedded 6-cycle in $\Delta_\Lambda$ with consecutive vertices being $\{x_i\}_{i=0}^5$. We say $\omega$ is of \emph{type I} if $\{x_0,x_2,x_4\}$ are of type $\hat a$ and $\{x_1,x_3,x_5\}$ are of type $\hat c$. We say $\omega$ is of \emph{type II} if $\{x_0,x_4\}$ are of type $\hat a$, $\{x_1,x_3,x_5\}$ are of type $\hat c$ and $x_2$ is of type $\hat b$. Suppose $\omega$ is either of type I or type II. We define $\omega$ satisfies property $(*)$ if there exists a vertex $z$ of type $\hat a$ or $\hat b$ such that $z$ is adjacent to each of $\{x_1,x_3,x_5\}$. 
\end{definition}

\begin{lem}
	\label{lem:center of 6cycle}
Suppose all type I and type II 6-cycles in the $D_4$ type Artin complex have property $(\ast)$. Then Theorem~\ref{thm:flagD4} holds.
\end{lem}

\begin{proof}
Let $\omega$ be a 6-cycle as in Theorem~\ref{thm:flagD4} with consecutive vertices $\{x_i\}_{i\in \mathbb Z/6\mathbb Z}$. Suppose $x_1,x_3$ and $x_5$ are of type $\hat c$. If $x_2$ is of type $\hat d$, then let $x'_2$ a vertex of type $\hat a$ which is adjacent to $x_2$. By Lemma~\ref{lem:link} (3), $x'_2$ is adjacent to each of $\{x_1,x_3\}$. Thus up to possible replacement, we can assume $x_2,x_4,x_6$ are of type either $\hat b$ or $\hat a$. Now Theorem~\ref{thm:flagD4} follows by the symmetry of the Dynkin diagram.
\end{proof}

Thus to prove Theorem~\ref{thm:flagD4}, it suffices to prove type I cycles has property $(\ast)$, which is Proposition~\ref{prop:typeI} and type II cycles has property $(\ast)$, which is Corollary~\ref{cor:type II}. 

In the rest of this section, we collect several observations for preparation of the proof of Proposition~\ref{prop:typeI} and Corollary~\ref{cor:type II}. Throughout this section, $\Lambda$ will be the Dynkin diagram of type $D_4$ with its vertex set $\{a,b,c,d\}$ such that $\{a,b,c\}$ are leaf vertices. 


\begin{lem}
	\label{lem:special4cycle}
	Let $\{x_i\}_{i=1}^4$ be consecutive vertices of an embedded 4-cycle in $\Delta_\Lambda$ such that the type of each vertex in the cycle belongs to $\{\hat a,\hat b,\hat c\}$. Suppose $\type(x_1)=\type(x_3)$; and $\type(x_2),\type(x_3),\type(x_4)$ are mutually distinct. Then $x_2$ and $x_4$ are adjacent.
\end{lem}

\begin{proof}
	By Theorem~\ref{thm:4 wheel}, either $x_2$ and $x_4$ are adjacent; or $x_1$ and $x_3$ are adjacent, which is impossible as adjacent vertices have different types; or there exists a vertex $z$ in $\Delta_\Lambda$ such that $z$ is adjacent to each of $\{x_1,x_2,x_3,x_4\}$. As adjacent vertices have different types, we now $z$ has type $\hat d$. Now $x_2$ and $x_4$ are adjacent by considering $\lk(z,\Delta_\Lambda)$ and applying Lemma~\ref{lem:link} (3).
\end{proof}

\begin{lem}
	\label{lem:typeIIblocal}
	Let $\omega$ be an embedded cycle of type II in $\Delta_\Lambda$, where $\Lambda$ is of type $D_4$. Then at least one of the following is true:
	\begin{enumerate}
		\item  there exists an edge path with consecutive vertices $x_1,y_1,y_2,y_3,x_3$ in $\lk(x_2,\Delta_\Lambda)$ such that $y_1$ and $y_3$ are of type $\hat a$ and $y_2$ is of type $\hat c$;
		\item there exists a vertex $y\in \Delta_\Lambda$ of type $\hat a$ such that $y$ is adjacent to each of $\{x_1,x_2,x_3\}$.
	\end{enumerate}
\end{lem}

\begin{proof}
Let $\Si_\Lambda$ be the Davis complex. Then we can identify the Coxeter complex $\mathsf C_\Lambda$ as the dual to $\Si_\Lambda$. Let $\bar x_i$ be the image of $x_i$ under the map $\Delta_\Lambda\to \mathsf C_\Lambda$, and let $C_i$ be the top-dimensional cell in $\Si_\Lambda$ dual to $\bar x_i$. Let $\whC_i$ be the associated standard subcomplex of $\od_\Lambda$. Let $w=w_0w_1w_2w_3w_4w_5$ be the word arising from the 6-cycle $\omega$ as in Definition~\ref{def:ncycle}. As $w=\id$, we know $w$ gives a null-homotopic loop $P=P_0P_1P_2P_3P_4P_5$ in $\od_\Lambda$, with $P_i\subset\whC_i$ for each $i$.
	
	It suffices to show either $w_2=u_1v_1v_2v_3u_3$ where $u_1,v_2,u_3\in A_{a,d}$ and $v_1,v_3\in A_{c,d}$, which corresponds to Case 1 of the lemma; or $w_2=u_1u_2u_3$ with $u_1,u_3\in A_{a,d}$ and $u_2\in A_{c,d}$, which corresponds to Case 2 of the lemma. Here $A_{a,d}$ means the standard parabolic subgroups generated by $a$ and $d$.
	
	We consider $\Pi_{\whC_2}(P)$ and let $P'_i=\Pi_{\whC_2}(P_i)$. Note that $P'_2=P_2$. Let $w'_i$ be the word we read off from the path $P'_i$. As $P$ is null-homotopic in $\od_\Lambda$, we know $\Pi_{\whC_2}(P)$ is null-homotopic in $\whC_2$. Thus $w'_0w'_1w'_2w'_3w'_4w'_5$ represents the trivial element. Then
	$$
	w_2=(w'_1)^{-1}(w'_0)^{-1}(w'_5)^{-1}(w'_4)^{-1}(w'_3)^{-1}.
	$$
	By Lemma~\ref{lem:projection location} and Lemma~\ref{lem:retraction property}, $\Pi_{\whC_2}(\whC_i)=\whC_2\cap\whC_i$ for $i=1,3$. Thus $w'_i\in A_{a,d}$ for $i=1,3$. As $\bar x_2$ is of type $\hat b$ in $\mathsf C_\Lambda$, the antipodal point of $\bar x_2$ in $\mathsf C_\Lambda$ is  also of type $\hat b$ by the geometry of Coxeter group of type $D_4$. Thus none of $\bar x_i$ is antipodal to $\bar x_2$ in $\mathsf C_\Lambda$ for $0\le i\le 5$. Thus none of $C_i$ ($i\neq 2$) is parallel to $C_2$. Thus $\Pi_{C_2}(C_i)$ is a face of $C_i$ of dimension $\le 2$. This for $i=0,4,5$, either $w'_i\in A_{a,d}$ or $w'_i\in A_{c,d}$ or $w'_i\in A_{a,c}$. In each of the case, we can rewrite $(w'_1)^{-1}(w'_0)^{-1}$ as $u_1v_1$ with $u_1\in A_{a,d}$ and $v_1\in A_{c,d}$, using that $a$ and $c$ commute. Note that we do allow $u_1$ and $v_1$ to be the trivial element. Similarly, we can rewrite $(w'_4)^{-1}(w'_3)^{-1}$ as $v_3u_3$ with $u_3\in A_{a,d}$ and $v_3\in A_{c,d}$. If $w'_5\in A_{a,d}$, then we are in Case 1 of the lemma. If $w'_5\in A_{c,d}$, then $v_1(w'_5)^{-1}v_3\in A_{c,d}$ and we are in Case 2 of the lemma. If $w'_5\in A_{a,c}$, then we write $(w'_5)^{-1}=c^*a^*$, absorbing the $c^*$ into $v_1$ so $v_1$ is still in $A_{c,d}$ and define $v_2=a^*$, leading to Case 1 of the lemma.
\end{proof}
Note that the two possibilities in Lemma~\ref{lem:typeIIblocal} are not mutually exclusive.
The next lemma can be proved in the same way as Lemma~\ref{lem:typeIIblocal}.
\begin{lem}
	\label{lem:5cycleblocal}
	Let $\omega$ be an embedded cycle of $5$-cycle in $\Delta_\Lambda$ with its consecutive vertices being $\{x_i\}_{i=0}^4$. Suppose $x_0$ has type $\hat b$, $x_1$ and $x_3$ have type $\hat a$ ,and $x_2$ and $x_4$ have type $\hat c$. Then at least one of the following is true:
	\begin{enumerate}
		\item  there exists an edge path with consecutive vertices $x_4,y_1,y_2,x_1$ in $\lk(x_0,\Delta_\Lambda)$ such that $y_1$ is of type $\hat a$ and $y_2$ is of type $\hat c$;
		\item $x_4$ and $x_1$ are adjacent.
	\end{enumerate}
\end{lem}

\begin{lem}
	\label{lem:filling 5 cycle}
Suppose every type I embedded 6-cycle in $\Delta_\Lambda$ satisfies property $(*)$ with $\Lambda$ being type $D_4$. Let $\omega$ be an embedded 5-cycle as in Lemma~\ref{lem:5cycleblocal}. Then either $x_4$ and $x_1$ are adjacent, or $x_0$ is adjacent to each $x_i$ for $1\le i\le 4$.
\end{lem}

\begin{proof}
We will show if $x_4$ and $x_1$ are not adjacent, then $x_0$ is adjacent to each $x_i$ for $1\le i\le 4$.
By Lemma~\ref{lem:5cycleblocal}, we know there exists an edge path with consecutive vertices $x_1,y_1,y_2,x_4$ in $\lk(x_0,\Delta_\Lambda)$ such that $y_1$ is of type $\hat c$ and $y_2$ is of type $\hat a$. As the 6-cycle $x_1,y_1,y_2,x_4,x_3,x_2$ satisfies property $(*)$, there is a vertex $z$ of type $\hat a$ or $\hat b$ such that $z$ is adjacent to each of $\{y_1,x_4,x_2\}$. We first consider the case when $z$ has type $\hat a$.
Consider the 4-cycle $y_1,x_0,x_4,z$ of vertex type $\hat c,\hat b,\hat c,\hat a$. As $y_1\neq x_4$ (otherwise $x_1$ and $x_4$ are adjacent), this 4-cycle is embedded. Thus by Lemma~\ref{lem:special4cycle}, $x_0$ is adjacent to $z$. Now we consider the 4-cycle $x_0,x_1,x_2,z$ of type $\hat b,\hat a,\hat c,\hat a$. As $x_1\neq z$ (otherwise $x_1$ and $x_4$ are adjacent), this 4-cycle is embedded. Then by Lemma~\ref{lem:special4cycle}, $x_0$ is adjacent to $x_2$. By applying Lemma~\ref{lem:special4cycle} to the embedded 4-cycle $x_0x_2x_4x_3$, we know $x_0$ and $x_3$ are adjacent. Now we consider the case when $z$ has type $\hat b$. Consider the 4-cycle $x_1,y_1,z,x_2$ of type $\hat a,\hat c,\hat b,\hat c$. 
If $y_1=x_2$, then $x_0$ is adjacent to $x_2$ and we conclude that $x_0$ is adjacent to each $x_i$ with $1\le i\le 4$ as before. If $y_1\neq x_2$, then Lemma~\ref{lem:special4cycle}, $z$ is adjacent to $x_1$. Now consider the 4-cycle $x_0,x_1,z,x_4$ of type $\hat b,\hat a,\hat b,\hat c$. 
If $x_0=z$, then $x_0$ is adjacent to $x_2$, hence $x_0$ is adjacent to $x_i$ for $1\le i\le 4$. If $x_0\neq z$, then Lemma~\ref{lem:special4cycle}, $x_1$ is adjacent to $x_4$, which is ruled out by our assumption.
\end{proof}

\section{Six-cycles of type I in the $D_4$ complex}
\label{sec:D4}
Throughout this section, $\Lambda$ is the Dynkin diagram of type $D_4$ with its vertex set $\{a,b,c,d\}$ such that $\{a,b,c\}$ are leaf vertices.
The goal of this section is to prove Proposition~\ref{prop:typeI}, which asserts that each 6-cycle of type I in $\Delta_\Lambda$ has property $(\ast)$.

\subsection{Artin group of type $D_4$ as a mapping class group}
\label{subsec:translation}
Let $\cS$ be the surface with genus 1 with two punctures $p_a$ and $p_c$ and one boundary component, as in Figure~\ref{fig:11} (I). 
Let $\PM(\cS)$ be the pure mapping class group of $\cS$. Sometimes we will also think $p_a$ and $p_c$ as marked points rather than punctures.
By \cite[Corollary 9]{soroko2020linearity}, the Artin group $A_\Lambda$ is isomorphic to $\PM(\cS)$, where $a,b,c,d$ correspond to the Dehn twists along simple closed curves $a,b,c,d$ in Figure~\ref{fig:11} (I) (we slightly abuse the notation and use $a$ for both a generator of $A_\Lambda$ and a curve on $\cS$). 
Take $z_a\neq z_c\in \partial \cS$. 
Let $\tau_a$ (resp. $\tau_c$) be the arc from $p_a$ to $z_a$ (resp. $p_c$ to $z_c$) as in Figure~\ref{fig:11} (I). 
\begin{figure}[h]
	\centering
	\includegraphics[scale=1]{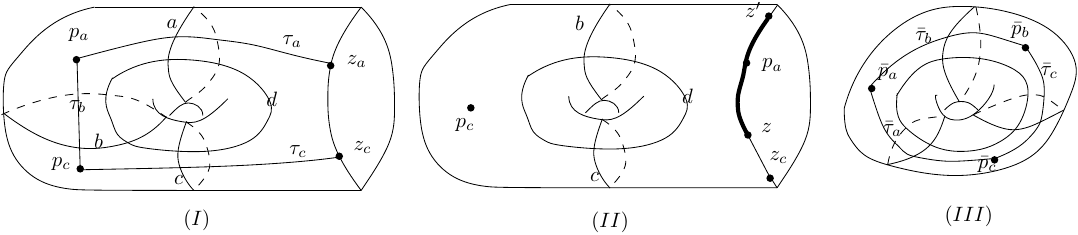}
	\caption{Surfaces associated to $D_4$.}
	\label{fig:11}
\end{figure}	

\begin{lem}
	\label{lem:correspondence}
The group $A_\Lambda\cong \PM(\cS)$ acts transitively on the set $\Omega_a$ of homotopy classes of simple arcs (rel endpoints) in $\cS$ from $p_a$ to $z_a$. Moreover, the stabilizer of $\tau_a$ (i.e. the collection of mapping classes in $\PM(\cS)$ preserving $\tau_a$ up to homotopy rel endpoints) is $A_{\hat a}$, i.e.  the subgroup of $A$ generated by vertices in $\Lambda\setminus\{a\}$.

Hence there is a 1-1 correspondence between left cosets of $A_{\hat a}$ in $A_\Lambda$ and elements in $\mathcal A$. More precisely, for each element in $\Omega_a$, the collection of all mapping classes in $A_\Lambda$ sending $\tau_a$ to this element is a left coset of $A_{\hat a}$.

The same statement holds if we replace $\Omega_a$ and $\tau_a$ by $\Omega_c$ and $\tau_c$.
\end{lem}

\begin{proof}
For any simple arc in $\cS$ from $p_a$ to $z_a$, if we cut $\cS$ open along this arc, the resulting surface is a torus with one boundary component and one puncture. Thus the transitivity of $A_\Lambda\act \Omega_a$ follows from the change of coordinate principle (cf. \cite[Chapter 1.3]{farb2011primer}). If a mapping class preserves $\tau_a$ up to homotopy rel endpoints, then we can find a representative in the class fixing $\tau_a$ pointwise (as homotopy between two arcs rel endpoints can be improved to isotopy, and this isotopy can be extended to an isotopy on $\cS$, see \cite[Chapter 1.2.7]{farb2011primer}). We scissor $\cS$ along $\tau_a$ to obtain the surface $\cS_a$ in Figure~\ref{fig:11} (II). The previous discussion gives a homomorphism Stab$_{\PM(\cS)}([\tau_a])\to \PM(\cS_a)$, which is injective. As $\PM(\cS_a)$ is generated by the Dehn twists around curves $b,c,d$ in Figure~\ref{fig:11} (II) and it is isomorphic to the subgroup $A_{\hat a}$ of $A_\Lambda$ generated by $\{b,c,d\}$ (\cite[Corollary 11]{soroko2020linearity}), this gives a lifting monomorphism $\PM(\cS_a)\to \PM(\cS)$, whose image is $A_{\hat a}$, and is contained in Stab$_{\PM(\cS)}([\tau_a])$. Thus the lemma follows.
\end{proof}

\begin{lem}
	\label{lem:disjoint}
Under the correspondence of Lemma~\ref{lem:correspondence},	
the two cosets $gA_{\hat a}$ and $hA_{\hat c}$ have non-empty intersection if and only if the corresponding two homotopy classes of arcs have disjoint representatives.
\end{lem}

\begin{proof}
Take $[\tau_1]\in \Omega_a$ and $[\tau_2]\in\Omega_c$ such that these two homotopy classes are represented by disjoint simple arcs $\tau_1$ and $\tau_2$. If we scissor $\cS$ along $\tau_1$ and $\tau_2$, then we obtain a torus with one boundary component. Thus by the change of coordinate principle, there is a homeomorphism $\phi:\cS\to \cS$ fixing the boundary and the punctures pointwise such that $\phi(\tau_a)=\tau_1$ and $\phi(\tau_c)=\tau_2$. Thus $[\tau_1]$ corresponds to the coset $[\phi]A_{\hat a}$, and $[\tau_2]$ corresponds to the coset $[\phi]A_{\hat c}$. These two cosets have non-empty intersection as $A_{\hat a}\cap A_{\hat c}\neq\emptyset$. For the only if direction, if $gA_{\hat a}\cap hA_{\hat c}\neq\emptyset$, then we can assume $g=h$. Then the corresponding two arcs are the images of $\{\tau_a,\tau_c\}$ under $g=h$, hence are disjoint.
\end{proof}

\begin{lem}
	\label{lem:common adjacent}
Take $[\tau_1],[\tau_2]\in \Omega_c$.  Let $x_1,x_2$ be the associated vertices in $\Delta_\Lambda$. Let $\delta$ denotes the Dehn twist along $\partial 
\cS$. Then 
\begin{enumerate}
	\item  if $[\tau_2]=\delta^n[\tau_1]$ for $n\neq 0$, then $x_1$ and $x_2$ can not be adjacent to a common vertex of type $\hat a$;
	\item  if there exists $[\tau'_2]\in \Omega_c$ such that $[\tau'_2]=\delta^n[\tau_1]$ some $n$ and there are more than one elements in $\Omega_a$ containing representatives that are disjoint from both $\tau_1$ and $\tau'_2$, then $x_1$ and $x_2$ are adjacent to a common vertex of type $\hat a$ if and only if they are adjacent to a common vertex of type $\hat b$.
\end{enumerate}
\end{lem}

\begin{proof}
For Assertion 1, if $x_1$ and $x_2$ are adjacent to a common vertex of type $\hat a$, then Lemma~\ref{lem:disjoint} implies that there is  $[\tau_3]\in \Omega_a$ such that $\tau_3\cap\tau_i=\emptyset$ for $i=1,2$. By our assumption on $\tau_1$ and $\tau_2$, $z_a$ and $p_a$ are not in the same connected component of $\cS\setminus (\tau_1\cup\tau_2)$, this is a contradiction.

For Assertion 2, let $x'_2\in \Delta_\Lambda$ be the vertex of type $\hat c$ associated with $[\tau'_2]$. Then there are two different vertices of type $\hat c$ that are adjacent to both $x_1$ and $x'_2$. By Theorem~\ref{thm:4 wheel}, there is a vertex $y_{\hat d}$ of type $\hat d$ which is adjacent to each of $x_1$ and $x'_2$. Let $y_{\hat b}$ and $y_{\hat a}$ be some vertices of type $\hat b$ and $\hat a$ respectively that are adjacent to $t_{\hat d}$. Then by Lemma~\ref{lem:link}, $\{x_1,y_{\hat d},y_{\hat b},y_{\hat a}\}$ and $\{x'_2,y_{\hat d},y_{\hat b},y_{\hat a}\}$ form vertex sets of two simplices in $\Delta_\Lambda$. Thus if $x_1$ corresponds to a coset of form $gA_{\hat c}$, then $x'_2$ corresponds to a coset of form $gc^m A_{\hat c}$. Hence $x_2$ corresponds to a coset of form $\delta ^ngc^mA_{\hat c}=g\delta^nc^mA_{\hat c}$. Up to a left translation, we can assume $x_1$ corresponds to the identity coset $A_{\hat c}$, and $x_2$ corresponding to the coset $\delta^nc^mA_{\hat c}$ where $\delta$ is the Garside element of $A_\Lambda$. By considering the automorphism of $A_\Lambda$ which exchanges $b$ and $a$ and fixes other generators and noting that this automorphism fix $\delta^n c^mA_{\hat c}$ setwise, 
we know there is a left $A_{\hat b}$-coset having nonempty intersection with both $A_{\hat c}$ and $\delta^n c^m  A_{\hat c}$ if and only if there is a left $A_{\hat a}$-coset having nonempty intersection with $A_{\hat c}$ and $\delta^n c^m A_{\hat c}$. 
\end{proof}

\subsection{A relative Artin complex as a complex of arcs}
\label{subsec:relarc}
Let $\Delta_0$ be the relative Artin complex $\Delta_{\Lambda,\{a,b,c\}}$. Note that the action of the center $Z_{A_\Lambda}$ on $\Delta$ is free, hence the same holds true for the action $Z_{A_\Lambda}\act \Delta_0$. Let $\bar \Delta_0$ be the quotient of this action, together with the covering map $\pi_Z:\Delta_0\to \bar \Delta_0$. As the action of $Z_{A_\Lambda}$ preserves the types of vertices in $\Delta$, each vertex in $\bar \Delta_0$ has a well-defined type, either $\hat a$ or $\hat b$ or $\hat c$.

Let $\bar \cS$ be a torus with three punctures $\{\bar p_a,\bar p_b,\bar p_c\}$, identified with the interior of $\cS$. This gives a surjective homomorphism $h:\PM(\cS)\to \PM(\bar \cS)$. The kernel of $h$ is exactly the center of $\PM(\cS)$, where the center is generated by $(abcd)^3$ under the isomorphism $\PM(\cS)\cong A_\Lambda$ (see \cite[Corollary 9]{soroko2020linearity}). Let $\bar A_\Lambda\cong \PM(\bar \cS)$ be the central quotient of $A_\Lambda$. Let $\{\bar a,\bar b,\bar c,\bar d\}$ be the image of the generators of $A_\Lambda$ in $\bar A_{\Lambda}$. Let $\bar A_{\hat a}$ (resp. $\bar A_{\hat b},\bar A_{\hat c}$) be the image of $A_{\hat a}$ (resp. $A_{\hat b},A_{\hat c}$) under the quotient. 
Let $\bar a$ be the simply closed curve on $\bar \cS$ corresponding to $a$ in $\cS$. Similarly we define $\bar b,\bar c,\bar \tau_a,\bar \tau_b,\bar\tau_c$. See Figure~\ref{fig:11} (III). Then under the identification $\bar A_\Lambda\cong \PM(\bar \cS)$, the generators $\{\bar a,\bar b,\bar c,\bar d\}$ corresponds to Dehn twists along the respective curves in $\bar S$. Let $\bar \Omega_a$ be the set of homotopy classes of simple arcs (rel endpoints) in $\bar\cS$ connecting two endpoints of $\bar \tau_a$. Similarly we define $\bar\Omega_b$ and $\bar\Omega_c$.

\begin{definition}
	\label{def:complex 3punctured}
Let $\bar \Delta'_0$ be the following 2-dimensional simplicial complex. The vertices of $\bar\Delta'_0$ are in 1-1 correspondence with elements in $\bar\Omega_a\sqcup \bar\Omega_b\sqcup \bar\Omega_c$. 
Two vertices are joined by an edge if they belong to different subsets of the partition $\bar\Omega_a\sqcup \bar\Omega_b\sqcup \bar\Omega_c$, and the corresponding two homotopy classes have representatives which are disjoint except at their endpoints. A 3-cycle in the 1-skeleton span a 2-face, if the concatenation of the associated three arcs form a homotopically non-trivial simple closed loop in $\bar \cS'$ which is defined to be the torus obtained from $\bar \cS$ by adding back its punctures.
\end{definition}

By a similar argument as in the proof of Lemma~\ref{lem:correspondence}, we know that $\bar A_\Lambda\cong \PM(\bar \cS)$ acts transitively on $\bar \Omega_a$, with the stabilizer of $\bar \tau_a$ be the subgroup $\bar A_{\hat a}$ generated by $\{\bar b,\bar c,\bar d\}$. This gives a 1-1 correspondence between vertices of type $\hat a$ in $\bar \Delta'_0$, and elements in $\bar\Omega_a$. By considering vertices of other types, we obtain a bijection $\Theta$ between the vertex set of $\bar\Delta_0$ and the vertex set of $\bar\Delta'_0$. 

\begin{lem}
	\label{lem:bar correspondence}
The bijection $\Theta$ extends to a simplicial isomorphism $\Theta:\bar\Delta_0\to \bar\Delta'_0$.
\end{lem}

\begin{proof}
By a similar argument as in Lemma~\ref{lem:disjoint}, we know $\Theta$ extends to an isomorphism on the 1-skeleton. Let $T$ be the torus obtained from $\bar \cS$ by filling back its punctures. Note that the 2-face $F$ in $\bar\Delta_0$ with its vertices being the identity cosets $\bar A_{\hat a},\bar A_{\hat b},\bar A_{\hat c}$ corresponds to three arcs $\bar \tau_a,\bar\tau_b,\bar\tau_c$ in $\bar \cS$ whose concatenation is a homotopcially non-trivial simple closed loop in $T$, hence gives a 2-face in $\bar\Delta'_0$. As any 2-face $F'$ in $\bar\Delta_0$ is a translate of $F$, thus the vertex set of $F'$ corresponds to three acts whose concatenation is the image of the loop in the previous sentence under a homeomorphism. Thus $F'$ corresponds to a 2-face in $\bar\Delta'_0$. Conversely, given $[\tau_1]\in\bar\Omega_a$, $[\tau_2]\in\bar\Omega_b$ and $[\tau_3]\in\bar\Omega_c$ such that the concatenation of $\tau_1,\tau_2,\tau_3$ is a homotopically non-trivial simple closed loop in $T$. Note that $T$ with any homotopically non-trivial simple closed loop in $T$ removed is always an open annulus. Thus there is a homeomorphism $\phi$ of $T$ fixing $\{\bar p_a,\bar p_b,\bar p_c\}$ pointwise such that $\phi(\tau_a)=\bar\tau_1$, $\phi(\tau_b)=\bar\tau_2$ and $\phi(\tau_c)=\bar\tau_3$. Thus $\{[\tau_1],[\tau_2],[\tau_3]\}$ corresponds to $\{[\phi]\bar A_{\hat a},[\phi]\bar A_{\hat b},[\phi]\bar A_{\hat c}\}$, which spans a 2-face in $\bar\Delta_0$.
\end{proof}

As $\Delta_0$ is simply-connected (\cite[Lemma 6.2]{huang2023labeled}), the fundamental group of $\bar\Delta_0\cong \bar\Delta'_0$ is $\mathbb Z$. Now we display an explicit generating cycle for the fundamental group of $\bar\Delta'_0$. A 3-cycle in $\bar\Delta'_0$ is \emph{degenerate} if it does not span a 2-face in $\bar\Delta'_0$.

\begin{lem}
	\label{lem:generator}
Let $x$ be a vertex of a degenerate 3-cycle $\omega$ in $\bar\Delta'_0$. Then $[\omega]$ represents a generator in $\pi_1(x,\bar\Delta'_0)$.
\end{lem}

\begin{proof}
Given $[\gamma_a],[\gamma'_a]\in \bar\Omega_a$, $[\gamma_b],[\gamma'_b]\in \bar\Omega_b$ and $[\gamma_c],[\gamma'_c]\in \bar\Omega_c$, such that the concatenation of $\gamma_a,\gamma_b,\gamma_c$ and the concatenation of $\gamma'_a,\gamma'_b,\gamma'_c$ form two homotopically trivial loops in $\bar\cS'$. There exists a self-homeomorphism $f$  of $\bar\cS$ preserving the punctures pointwise such that $f(\gamma_a)=\gamma'_a$, $f(\gamma_b)=\gamma'_b$ and $f(\gamma_c)=\gamma'_c$. So $\bar A_\Lambda$ acts transitively on the set of degenerate 3-cycles in $\bar\Delta'_0$. So it suffices to prove the lemma for one specific degenerate 3-cycle.

Let $\delta$ and $\delta_{\hat a}$ be the Garside element of $A_\Lambda$ and $A_{\hat a}$ respectively. 
Then $\delta=(abcd)^3$ and $\delta=\delta_{\hat a}\cdot adc\cdot bda$. Recall that there is a 1-1 correspondence between top-dimensional simplices in $\Delta_\Lambda$ and elements in $A_\Lambda$. Let $K_1,K_2,K_3,K_4$ be the top-dimensional simplices in $\Delta_\Lambda$ corresponding to identity, $\Delta_{bcd}$, $\Delta_{bcd}\cdot adc$ and $\Delta_\Lambda$ respectively. Then $K_1\cap K_2$ contains a vertex $x_1$ of type $\hat a$ represented by the identity left coset $A_{\hat a}$, $K_2\cap K_3$ contains a vertex $x_2$ of type $\hat b$ represented by $\Delta_{bcd}\cdot A_{\hat b}$, and $K_3\cap K_4$ contains a vertex $x_3$ of type $\hat c$ represented by $\Delta_{bcd}\cdot adc\cdot A_{\hat c}$. Let $x_4$ be the vertex of type $\hat a$ in $K_4$. Then $x_1,x_2,x_3,x_4$ forms an edge path $P$ in $\Delta_\Lambda$. As the action of the Garside element of $A_\Lambda$ brings $K_1$ to $K_4$, and it preserves type of vertices, this action also brings $x_1$ to $x_4$. Thus the path $P$ projects a 3-cycle in $\bar\Delta'_0\cong\bar\Delta_0$ which is a generator of the fundamental group. This 3-cycle must be degenerate.
\end{proof}

\subsection{6-cycles in terms of 6-arcs}

Let $\omega$ be an embedded 6-cycle of type I or II in $\Delta_\Lambda$ with consecutive vertices being $\{x_i\}_{i=0}^5$. 
Let $\{[\tau_i]\}_{i=0}^5$ be the homotopy classes of arcs in $\mathcal S$ associated with the vertices of $\omega$ - more precisely, if $x_i$ is of type $\hat a$ or $\hat c$, then we use Lemma~\ref{lem:correspondence}; if $x_i$ is of type $\hat b$, then we use Lemma~\ref{lem:bar correspondence} to obtain a class $[\bar\tau]\in \bar\Omega_b$ corresponding to the image $\bar x_i$ of $x_i$ under $\Delta_0\to \bar\Delta_0$, and $[\bar\tau]$ gives a homotopy class of arcs $[\tau]$ from $p_a$ to $p_c$ as $\bar\cS$ is identified with the interior of $\cS$.
We take a representative $\tau_i$ from each class, and assume the intersection number between any two of the six representatives are minimized (this is possible by putting a complete hyperbolic metric with geodesic boundary on $\cS$, and considering geodesic representatives). 
By Lemma~\ref{lem:correspondence}, Lemma~\ref{lem:bar correspondence} and our choice, we have:
\begin{lem}
Under the above assumption, $\tau_i\cap\tau_{i+1}=\emptyset$ for each $i\in \mathbb Z/6\mathbb Z$.
\end{lem}


We scissor $\cS$ along $\tau_0$ to obtain the surface $\cS_0$, see Figure~\ref{fig:11} (II). The arc $\tau_a$ of $\cS$ gives two subsegments in $\partial \cS_0$, which we denote by $\overline{zp_a}$ and $\overline{p_az'}$. There is a homeomorphism $\phi:\overline{zp_a}\to \overline{z'p_a}$ sending $z$ to $z'$ so that we can obtain $\cS$ back from $\cS_0$ by gluing $\overline{zp_a}$ and $\overline{z'p_a}$ along $\phi$. Define $\pi_0$ to be the arc $\overline{z'p_a}\cup\overline{zp_a}$ in $\partial \cS_0$.

For $i=1,5$, $\tau_i\subset \cS$ gives a simple arc on $\cS_0$ from $p_a$ to $z_c$, which we still denoted by $\tau_i$. For $i=2,3,4$, as $\tau_i$ might intersect $\tau_a$ multiple times, $\tau_i$ gives a disjoint union of simple arcs on $\cS_0$. We will still denote the disjoint union of these simple arcs by $\tau_i$, and each connected component of $\tau_i$ is called a \emph{trace} of $\tau_i$. Note that if $\tau_3\subset \cS_0$ is not connected, then it must contain a trace going from $p_a$ to a point in $\pi_0$, and possible more than one traces starting and ending in $\pi_0$.

\begin{definition}
	\label{def:boundary parallel}
	A $\pi_0$-trace of $\tau_i$ is a trace with endpoints in $\pi_0$. 
	Let $\cS'_0$ be the torus with boundary component obtained by filling the puncture $p_c$ in $\cS_0$.
	A trace $\gamma$ with two endpoints $x,y$ in the boundary is \emph{boundary parallel} in $\cS_0$ (resp. $\cS'_0$), if this trace and an arc $\alpha$ in the boundary together bound a disk in $\cS_0$ (resp. $\cS'_0$). Note that there are two arcs in $\cS_0$ or $\cS'_0$ from $x$ to $y$; if $\gamma$ is a $\pi_0$-trace, then one such arc contains $z_c$, and another does not. A $\pi_0$-trace is \emph{$c$-boundary-parallel} (resp. \emph{$a$-boundary-parallel}) in $\cS_0$ if it bound a disk in $\cS_0$ together with an arc on the boundary passing through $z_c$ (resp. not passing through $z_c$). A \emph{good trace} in $\cS_0$ is a $\pi_0$-trace that is not boundary parallel in $\cS'_0$.
\end{definition}

Similarly, for $i\neq 0$ and $x_i$ of type $\hat a$ or $\hat c$, we can scissor $\cS$ along $\tau_i$ to obtain $\cS_i$. The arc $\tau_i$ gives two subarcs in $\partial \cS_i$, whose union is denoted by $\pi_i$. We can repeat the above discussion for traces of $\tau_j$ in $\cS_i$, for $j\neq i$. Moreover, we define boundary parallel and good trace in $\cS_i$, as well as the surface $\cS'_i$, in a similar way.

\begin{lem}
	\label{lem:boundary parallel}
Given four arcs $\tau_0,\tau_1,\tau_2,\tau_3$ in $\cS$ in minimal position such that adjacent arcs have trivial intersection. Suppose $[\tau_1],[\tau_3]\in\Omega_c$ and $[\tau_0]\in \Omega_a$.
\begin{enumerate}
	\item If $\tau_2$ goes from $p_a$ to $z_a$, then each trace of $\tau_2$ is not $a$-boundary parallel in $\cS'_0$, and it is not boundary parallel in $\cS_0$. Moreover, each $\pi_0$-trace of $\tau_3$ is not $a$-boundary-parallel in $\cS_0$. 
	\item If $\tau_2$ goes from $p_a$ to $p_c$, then each $\pi_0$-trace of $\tau_2$ is not $a$-boundary parallel in $\cS'_0$, and it is not boundary parallel in $\cS_0$. Moreover, each $\pi_0$-trace of $\tau_3$ is not $a$-boundary-parallel in $\cS_0$. 
\end{enumerate}		
\end{lem}

\begin{proof}
	We only prove Assertion 1, as Assertion 2 is similar. Let $i=2$ or $3$.
	Take a trace $\tau$ of $\tau_i$ from $x$ to $y$ with $\{x,y\}\subset \pi_0$ such that $\tau$ is $a$-boundary-parallel in $\cS_0$. We also view $\tau$ as an arc $\tau'$ in $\cS$ going from $x'\in \tau_0$ to $y'\in\tau_0$. Then in $\cS$, $\tau'$ and the subsegment $\overline{x'y'}$ of $\tau_0$ bound a disk $D$ in $\cS$. If $\{x,y\}\subset \overline{z'p_a}$ or $\{x,y\}\subset \overline{zp_a}$, then the interior of the disk $D$ does not contain $p_a$, hence $\tau'$ and $\tau_0$ form a bigon, which contradicts our assumption that $\tau_i$ and $\tau_0$ have minimal intersection. It remains to consider the case $x$ and $y$ are in the interior of $\overline{z'p_a}$ and $\overline{zp_a}$ respectively, in which case $p_a$ is contained in the interior of the disk $D$. We can also assume $\tau$ is among all such boundary parallel traces such that the pair $\{x',y'\}$ is closest to $p_a$ in $\tau_0$. Suppose $x'$ is closer to $p_a$ than $y'$. Let $\tau''\neq \tau$ be the subsegment of $\tau_i$ starting at $x'$ until it hits $\tau_0$ at $x''$ (it is possible that $x''=p_a$). Note that the interior of $\tau''$ is disjoint from the boundary of the disk $D$, we have $\tau''\subset D$ because of orientation consideration. Then $\tau''$ and $\tau_0$ bound a disk $D'\subset D$. Our choice of $\{x',y'\}$ implies that $p_a$ is not in the interior of $D'$, thus we reach a contradiction as before. This shows that each $\pi_0$-trace of $\tau_2$ or $\tau_3$ is not $a$-boundary-parallel in $\cS_0$.
	
Now we show each trace of $\tau_2$ is not $a$-boundary parallel in $\cS'_0$ and it is not boundary parallel in $\cS_0$. Suppose by contradiction that there is a trace $\tau$ of $\tau_2$ such that $\tau$ and a sub-arc of  $\pi_0$ bound a disk $D\subset \cS_0$ with puncture $p_c$ in $D$. Note that $\tau_1$ is a simple arc in $\cS_0$ from $p_c$ in $D$ to $z_c$ outside $D$. However, $\tau_1\cap (\tau_0\cup\tau_2)=\emptyset$, in particular, $\tau_1$ disjoint from the boundary of $D$. This is a contradiction. Thus each trace of $\tau_2$ is not $a$-boundary-parallel in $\cS'_0$. If $\tau$ is boundary parallel in $\cS_0$, then by previous discussion, the only possibility left is that $\tau$ and a sub-arc $\gamma\subset\partial\cS_0$ with $z_c\in \gamma$ bound a disk $D'\subset \cS_0$; moreover, $p_c$ is outside $D'$ (if $p_c\in D'$, then $\tau$ is not boundary parallel in $\cS_0$). However, this is impossible as $\tau_1$ is an arc from $p_c$ to $z_c$ avoiding $\tau$.
\end{proof}

Suppose $\omega$ is of type I or II. Take $i=2,3$ or $4$.
Let $\tau,\tau'\subset \cS_0$ be two traces from $\tau_i$ such that they connect two points on $\partial \cS_0$. 
We say $\tau$ and $\tau'$ are \emph{parallel}  if $\tau\cap \tau'=\emptyset$ and $\tau,\tau'$ together with two arcs on $\partial \cS_0$ bound a disk $D\subset \cS_0$, with possibly the puncture $p_c$ in $D$. Now assuming $\tau$ and $\tau'$ are parallel. We define points $\xi\in\partial \tau$ and $\xi'\in\partial\tau'$ are \emph{aligned}, if $\xi$ and $\xi'$ are contained in the same component of $\partial D\cap \partial \cS_0$. A subset $A$ of $\cS_0$ is \emph{squeezed} by $\tau$ and $\tau'$ if $A$ is contained in the interior of the disk $D$.


\begin{lem}
	\label{lem:S3}
Given four arcs $\tau_0,\tau_1,\tau_2,\tau_3$ in $\cS$ in minimal position such that adjacent arcs have trivial intersection. Suppose $[\tau_0],[\tau_2]\in\Omega_a$ and $[\tau_1],[\tau_3]\in \Omega_c$.	
Suppose $\tau_0\cap \tau_3\neq\emptyset$. Suppose $\tau_2$ has more than one traces on $\cS_0$, and all traces of $\tau_2$ are good. Then at least one of the following hold:
\begin{enumerate}
	\item there is a good trace $S_3$ of $\tau_3$ in $\cS_0$ such that $S_3\cap \tau_1=\emptyset$, and the two endpoints of $S_3$ are in different components of $\pi_0\setminus\{p_a\}$;
	\item there are traces $S_3,R_3$ of $\tau_3$ in $\cS_0$ such that $(S_3\cup R_3)\cap \tau_1=\{p_c\}$; $S_3$ is a good trace with its endpoints in the same component of $\pi_0\setminus\{p_a\}$; and $R_3$ is a trace from $p_c$ to a point $x\in \pi_0$ with $x$ and the endpoints of $S_3$ in different components of $\pi_0\setminus\{p_a\}$;
	\item there are traces $S_3,R_3$ of $\tau_3$ in $\cS_0$ such that $(S_3\cup R_3)\cap \tau_1=\{p_c,z_c\}$; $S_3$ goes from $z_c$ to  $x'\in \pi_0$ such that $S_3$ is not boundary parallel in $\cS'_0$; and $R_3$ goes from $p_c$ to  $x\in\pi_0$ such that $x$ and $x'$ are in different components of $\pi_0\setminus\{p_a\}$.
\end{enumerate}
\end{lem}
\begin{figure}[h]
	\centering
	\includegraphics[scale=0.9]{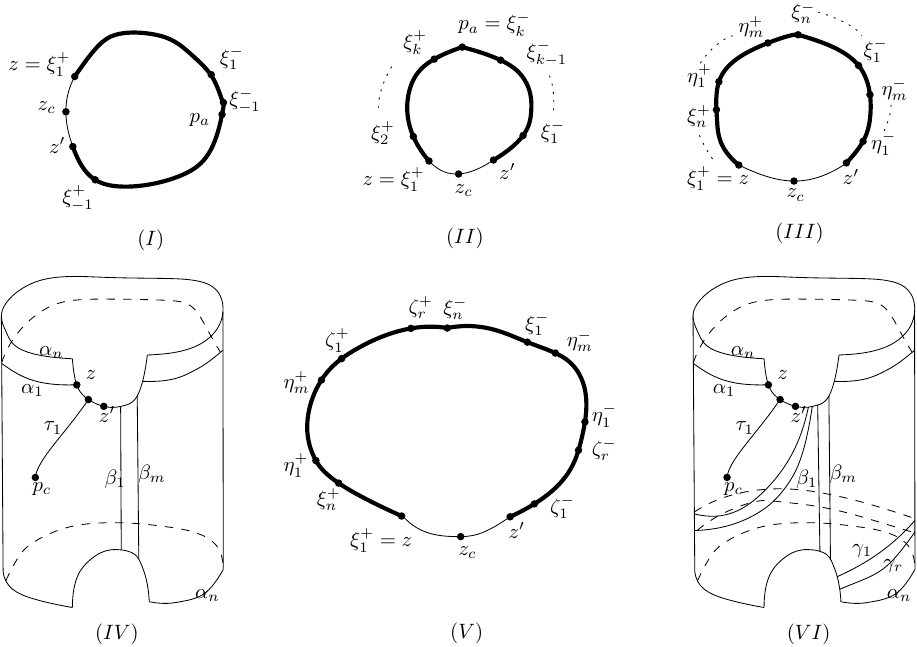}
	\caption{Proof of Lemma~\ref{lem:S3}.}
	\label{fig:12}
\end{figure}	
\begin{proof}
As there can be at most three different parallel classes of non-boundary-parallel traces that are mutually disjoint in a torus with one boundary component, and we assume all traces of $\pi_2$ are good, we know there are at most three different parallel classes of $\tau_2$. 

\smallskip
\noindent
\underline{Case 0: $p_c$ is squeezed by two parallel traces of $\tau_2$}. See Figure~\ref{fig:12} (I) for the following discussion. Let $\alpha_1$ and $\alpha_{-1}$ be the two ``innermost'' parallel traces of $\tau_2$ that squeezes $p_c$, i.e. there does not exist any other traces of $\tau_2$ which is squeezed by $\alpha_1$ and $\alpha_{-1}$. As $\tau_1\cap\tau_2=\emptyset$, we know $\tau_1$ is squeezed by $\alpha_1$ and $\alpha_{-1}$. As the arcs $\overline{z_cz}$ and $\overline{z_cz'}$ in $\partial\cS_0$ have empty intersection of with $\tau_2$ except possibly at $\{z,z'\}$, these two arcs (with $\{z,z'\}$ removed) are squeezed by $\alpha_1$ and $\alpha_{-1}$. Exactly one of $\{z,z'\}$ is contained in a trace $T'$ of $\tau_2$. If $T'\cap\alpha_i=\emptyset$ for $i=1,-1$, then $T'$ is squeezed by $\alpha_1$ and $\alpha_{-1}$, contradiction. So we can assume without loss of generality that $T'=\alpha_1$ and $z\in \alpha_1$. Let $\xi^\pm_i=\partial\alpha_i$. We assume $\xi^+_1$ and $\xi^+_{-1}$ are aligned. Let $D$ be the component of $\cS'_0\setminus \tau_2$ that contains $p_c$. Then $D$ is an open disk, its closure $\bar D$ in $\cS'_0$ is a  closed disk, and $\partial D$ is made of $\alpha_1,\alpha_{-1}$, and two arcs in $\cS_0$ which are $\gamma_1$ going from $\xi^+_{1}$ to $\xi^+_{-1}$ (passing through $z_c$ and $z'$) and $\gamma_2=\overline{\xi^-_1\xi^-_{-1}}$. Note that $p_a$ is not contained in the interior of $\gamma_2$, otherwise the trace of $\pi_2$ containing $p_a$ is squeezed by $\alpha_1$ and $\alpha_{-1}$, contradicting the ``innermost'' assumption.

 Suppose $\gamma_2$ and $\overline{z'\xi^+_{-1}}$ are in different components of $\pi_0\setminus\{p_a\}$. Let $R_3$ be the trace of $\tau_3$ starting from $p_c$. As $R_3$ and $\tau_1$ are two arcs from a point $p_c$ inside $D$ to some point in $\partial D$, $R_3\cap \tau_1=\{p_c\}$ as $\tau_1$ and $\tau_3$ have minimal intersection. As $\tau_0\cap \tau_3\neq\emptyset$, $R_3$ either ends in an interior point of $\overline{z'\xi^+_{-1}}$ or an interior point of $\gamma_2$. 
 
 If $R_3$ ends in $\overline{z'\xi^+_{-1}}$ at point $x$, then $x$ is identified with $x'\in \pi_0$ in different side of $p_a$ via the homeomorphism $\phi:\overline{zp_a}\to\overline{z'p_a}$. As $\xi^-_1\neq\xi^+_1$, we know $x'\notin \gamma_2$. Let $S_3$ be the trace of $\tau_3$ containing $x'$. Then $S_3$ is contained in a component $C$ of $\cS_0\setminus \tau_2$ different from $D$. As a point in $\partial C$ is either in a trace of $\tau_2$, or in $\partial\cS_0$, we deduce from $S_3\cap\tau_2=\emptyset$ and $\partial C\cap\partial \cS_0\subset\pi_0$ that two endpoints of $S_3$ are contained in $\pi_0$. 
 
 Now we show $S_3$ is a good trace. It suffices to show $S_3$ not boundary parallel in $\cS'_0$.
 If $S_3$ and an arc $A\subset \partial\cS'_0$ bound a disk $D'$ in $\cS'_0$, then any trace of $\tau_2$ has empty intersection with $D'$ except at $S_3\cap A$, otherwise $\tau_2$ will contain a trace which is boundary parallel in $\cS'_0$, contradiction. It follows that $D'\subset C$. In particular, $p_c\notin D'$ and $A\subset \pi_0$, thus $S_3$ is $a$-boundary-parallel in $\cS_0$. This contradicts
Lemma~\ref{lem:boundary parallel}. Hence $S_3$ is good.

 As $\tau_1\subset D$, we know $S_3\cap \tau_1=\emptyset$. Then we are in Lemma~\ref{lem:S3} (1) or (2) depending whether the two endpoints of $S_3$ are in different components of $\pi_0\setminus\{p_a\}$ or not. 
 
 If $R_3$ ends at $x\in\gamma_2$, then $x$ is identified via $\phi$ to a point $x'\in \overline{z'p_a}$. Note that $x'\notin \overline{z'\xi^+_{-1}}$. Thus the trace $S_3$ of $\tau_3$ containing $x'$ lies in a component of $\cS_0\setminus\tau_2$ which is different from $D$, and we can finish the argument as before. 
 
 It remains to consider the situation that $\gamma_2$ and $\overline{z'\xi^+_{-1}}$ are in the same side of $p_a\in\pi_0$. Let $R_3$ as before. Similarly in all cases the next trace $S_3$ of $\tau_3$ is contained in a component of $\cS_0\setminus\tau_3$ which is different from $D$. Then we finish as before.


\smallskip
\noindent
\underline{Case 1: there is only one parallel class of traces of $\tau_2$.} Because of Case 0, we assume $p_c$ is not squeezed by two traces of $\tau_2$. See Figure~\ref{fig:12} (II) for the following discussion. Order the traces of $\tau_2$ as $\{\alpha_i\}_{i=1}^k$ such that for each $1<i_0<k$, all of $\{\alpha_i\}_{i=i_0+1}^k$ are in one side of $\alpha_{i_0}$ in the sense that $\{\partial \alpha_i\}_{i=i_0+1}^k$ is contained in the same component of $\partial \cS_0\setminus \partial\alpha_{i_0}$, and all of  $\{\alpha_i\}_{i=1}^{i_0-1}$ are in one side of $\alpha_{i_0}$. Let $\xi^\pm_{i}=\partial \alpha_i$, and we assume $\{\xi^+_i\}_{i=1}^k$ are aligned in the sense defined before Lemma~\ref{lem:S3}. Hence $\{\xi^-_i\}_{i=1}^k$ are also aligned. We assume without loss of generality that $\xi^\pm_1$ is closest to $z_c$ in $\partial\cS_0$ among $\{\xi^\pm_{i}\}_{i=1}^k$. 

Note that in $\cS$, $\tau_2$ starts at $z_a$, and possibly hits the interior of $\tau_0$ several times before it ends at $p_a$. Thus the $2k$ points in $\{\xi^\pm_{i}\}_{i=1}^k$ must contain one point which is either $z$ or $z'$, one point which is $p_a$, $(k-1)$ points in the interior of $\overline{zp_a}$, and $(k-1)$ points in the interior of $\overline{z'p_a}$ which is the image of the previous $(k-1)$ points under the homeomorphism $\phi:\overline{zp_a}\to\overline{z'p_a}$. As we were assuming  that $\xi^\pm_1$ is closest to $z_c$, one of $\xi^\pm_1$, say $\xi^+_1$, must be one of $\{z,z'\}$. Assume without loss of generality that $\xi^+_1=z$. As all the $\alpha_i$'s are parallel to each other, the $2k$ points $\{\xi^+_i\}_{i=1}^k$ must be ordered as $\{\xi^+_1,\ldots,\xi^+_k,\xi^-_k,\ldots,\xi^-_1\}$ in the interval $\pi_0$. Thus $p_a=\xi^-_k$. Let $A$ be the only annulus region of $\cS_0\setminus(\cup_{i=1}^k\alpha_i)$. Then $p_c\in A$, and $\partial A\cap\partial \cS_0$ has two components, the first one $\gamma_1$ is the arc from $\xi^+_k$ to $\xi^-_k=p_a$, and the second one $\gamma_2$ is the arc from $\xi^+_1$ to $\xi^-_1$ (passing through $z_c$ and $z'$). 

As $\tau_0\cap \tau_3\neq\emptyset$, we know the trace $R_3$ of $\tau_3$ starting at $p_c$ must end in an interior point of $\overline{\xi^+_k\xi^-_k}$ or $\overline{z'\xi^-_1}$. If $R_3$ ends in an interior point of $\overline{z'\xi^-_1}$, then this point is identified with an interior point of $\overline{\xi^+_1\xi^+_2}$ via $\phi:\overline{zp_a}\to\overline{z'p_a}$. As $\tau_3\cap \tau_2=\emptyset$ and the intersection of $\tau_3$ and $\tau_1$ is minimal, the next trace of $\tau_3$, denoted by $S_3$, is squeezed by $\alpha_1$ and $\alpha_2$, ending in an interior point of $\overline{\xi^-_1\xi^-_2}$. In particular $S_3$ is a good trace and the two endpoints of $S_3$ are in different sides of $p_a$ in $\pi_0$. As $\tau_1$ is an arc in $\cS_0$ which is disjoint from all traces of $\tau_2$, we have $\tau_1\subset A$. However, $S_3\cap A=\emptyset$. Thus $\tau_1\cap S_3=\emptyset$ and we are in case 1 in the statement of Lemma~\ref{lem:S3}. Then case $R_3$ ends in $\overline{\xi^+_k\xi^-_k}$ is similar.

\smallskip
\noindent
\underline{Case 2: there are exactly two parallel classes of traces of $\tau_2$.} Again we assume $p_c$ is not squeezed by two parallel traces of $\tau_2$. Let $\{\alpha_i\}_{i=1}^n$ and $\{\beta_j\}_{j=1}^m$ be the two parallel classes of traces. Note that for each $\alpha_i$, the two endpoints of $\alpha_{i'}$ with $i'\neq i$ are contained in the same component of $\partial \cS_0\setminus\partial\alpha_i$; and the two endpoints of any $\beta_j$ are contained in different components of $\partial \cS_0\setminus\partial\alpha_i$. Let $\xi^\pm_i=\partial \alpha_i$ and $\eta^\pm_j=\partial \beta_j$. Then we can order $\alpha_i$ and $\beta_j$ such that their boundary points in $\partial \cS_0$ is as in Figure~\ref{fig:12} (III) (see also Figure~\ref{fig:12} (IV)) - here we are assuming without loss of generality that $z\in \alpha_1$. The collection of boundary points of all $\alpha_i$ and $\beta_j$ contains $\{z,p_a\}$, with the remaining $2(m+n-1)$ paired up via $\phi:\overline{zp_a}\to \overline{z'p_a}$. A counting argument implies that $\xi^-_n=p_a$. This implies in particular that $n>1$, otherwise $\tau_2$ only has a single trace $\alpha_1$ going from $z$ to $p_a$.

Note that components $\cS'_0\setminus\tau_2$ are all disks. There is a unique component $D$ whose boundary have two $\alpha$ traces and two $\beta$ traces. More precisely, $\partial D$ is made of $\alpha_1,\alpha_n,\beta_1,\beta_m$, and 4 arcs in $\partial \cS_0$ as follows: $\gamma_1$ goes from $\eta^-_1$ to $z$ (via $z'$ and $z_c$), $\gamma_2=\overline{\xi^+_n\eta^+_1}$, $\gamma_3=\overline{\eta^+_m\xi^-_n}$ and $\gamma_4=\overline{\xi^-_1\eta^-_m}$. By our assumption on $p_c$, we know $p_c\in D$. Let $R_3$ be the trace of $\tau_3$ containing $p_c$. By a similar argument as before, we know $R_3\cap \tau_1=\{p_c\}$. As $\tau_1\cap\tau_3\neq\emptyset$, $R_3$ ends on an interior point $x$ of either $\overline{z'\eta^-_1}$, or $\gamma_2$, or $\gamma_3$, or $\gamma_4$. If $x$ is in $\overline{z'\eta^-_1}$ or $\gamma_3$, then $\phi:\overline{zp_a}\to\overline{z'p_a}$ identifies $x$ with another point $x'$ which is squeezed by two $\alpha$ traces (this uses that $n>1$). Thus the trace $S_3$ of $\tau_3$ starting at $x'$ is squeezed by two $\alpha$ traces, in particular $S_3$ is a good. By Figure~\ref{fig:12} (III), the two endpoints of $S_3$ are in two different sides of $p_a$, thus we are in Lemma~\ref{lem:S3} (1). It remains to consider the case $x\in \gamma_2$ or $\gamma_4$.

Note that $\phi$ identifies $\gamma_2$ and $\gamma_4$ if and only if $n-1=m$. Thus if $n-1\neq m$, then $\phi$ identify $x$ with a point $x'$ which is either squeezed by two $\alpha$ traces, or squeezed by $\beta$ traces, and similar as before by taking $S_3$ to the trace starting at $x'$ we are in Lemma~\ref{lem:S3} (1). Now assume $n-1=m$. We first look at the case $x\in \gamma_2$. Then $x'\in \gamma_4$. Let $T$ be the trace of $\tau_3$ containing $x'$. Then $T\subset D\cup \partial D$. The possibilities of the endpoint $x''$ of $T$ are (1) $x''=z_c$; (2) $x''\in \overline{z'\eta^-_1}$; (3) $x''\in\gamma_4$; (4) $x''\in \gamma_3$; (5) $x''\in\gamma_2$. If $(1)$ happens, then the minimal intersection assumption of $\tau_1$ and $\tau_3$ implies that $T\cap \tau_1=\{z_c\}$, and we are in Lemma~\ref{lem:S3} (3) by taking $T=S_3$. If $(2)$ or $(4)$ happens, then we are reduced to the previous paragraph. If $(3)$ happens, 
then $T\cap\tau_2=\emptyset$, hence $T$ and $\gamma_4$ bound a disk $D'$ in $\cS'_0$. As $\gamma_3\subset\pi_0$, Lemma~\ref{lem:boundary parallel} imply that $p_c\in D'$. However, this is impossible as $T\cap R_3=\emptyset$. If $(5)$ happens, then we look at the next trace of $\tau_3$ and repeat the previous discussion, and we are done after finitely many steps. The case $x\in\gamma_4$ is similar.

\smallskip
\noindent
\underline{Case 3: there are exactly three parallel classes of traces of $\tau_2$.} Suppose as before that $p_c$ is not squeezed by two parallel traces of $\tau_2$.  Let $\{\alpha_i\}_{i=1}^n$, $\{\beta_j\}_{j=1}^m$ and $\{\gamma_\ell\}_{\ell=1}^r$ be the parallel classes. Suppose $\partial \alpha_i=\xi^\pm_i,\partial \beta_j=\eta^\pm_j$ and $\partial \gamma_\ell=\zeta^\pm_\ell$. We assume without loss of generality that $z\in\alpha_1$, and assume the traces and their boundaries are arranged in Figure~\ref{fig:12} (V) and (VI). A counting argument implies that $\xi^-_n=p_a$, and a similar argument as before implies that $n>1$.
Let $D$ be the connected component of $\cS_0\setminus\tau_2$ that contains $p_c$. Then $\partial D$ is made of $\alpha_1,\gamma_1$ and $\beta_m$, together with the following three arcs on $\partial \cS_0$: $\theta_1$ going from $\zeta^-_1$ to $z$ via $z'$ and $z_c$, $\theta_2=\overline{\eta^+_m\zeta^+_1}$ and $\theta_3=\overline{\xi^-_1\eta^-_m}$. Let $R_3$ be the trace of $\tau_3$ starting at $p_c$ and ending at $x$. Then $x$ is an interior point of either $\overline{z'\zeta^-_1}$, or $\theta_2$, or $\theta_3$. If $x\in \overline{z'\zeta^-_1}$, then we can argue as in Case 2, using $n>1$. 

Assume $x\in\theta_2$ or $\theta_3$. Note that $\phi:\overline{zp_a}\to\overline{z'p_a}$ identifies $\theta_2$ and $\theta_3$ if and only if $n-1=r$, in which case $\eta^+_m$ and $\eta^-_m$ are identified. This contradicts that $\tau_2$ is a simple arc in $\cS$. Thus $n-1\neq r$. Let $x'$ be the point which is identified with $x$ via $\phi$, and let $S_3$ be the trace of $\tau_3$ containing $x'$. Then $S_3$ and $\tau_1$ are contained in different components of $\cS_0\setminus\tau_2$. Similar to previous cases, $S_3$ is a good trace. As $\partial S_3\subset \pi_0$, we are either in Lemma~\ref{lem:S3} (1) or Lemma~\ref{lem:S3} (2). 
\end{proof}


%

\begin{cor}
	\label{cor:S4}
Suppose $\omega$ is of type $I$ or $II$.	
Suppose $\tau_0\cap \tau_3\neq\emptyset$.  Suppose $\tau_3$ has a trace $S_3$ satisfying at least one of the cases of Lemma~\ref{lem:S3}. Then in each cases of Lemma~\ref{lem:S3}, there exists a trace $S_4$ of $\tau_4$ such that one of the following holds:
\begin{enumerate}
	\item $S_4$ is good, and $S_3$ and $S_4$ are not parallel;
	\item $S_4$ is good, $S_3$ and $S_4$ are parallel and $p_c$ is not squeezed by $S_3$ and $S_4$;
	\item $S_4$ is $c$-boundary-parallel in $\cS'_0$.
\end{enumerate}
\end{cor}

\begin{proof}
Suppose Case 3 does not happen. Then Lemma~\ref{lem:boundary parallel} implies that all traces of $\tau_4$ are good.	
Suppose Case 1 does not happen. As $S_4\cap \tau_3=\emptyset$, we can assume all traces of $\tau_4$ are parallel to $S_3$. As $x_4$ is of type $\hat a$, one trace of $\tau_4$, denoted $\tau$, contains one of $\{z,z'\}$ (see Figure~\ref{fig:11} (II)). One trace of $\tau_4$, denoted $\tau'$, contains $p_a$. 

In Case (1) of Lemma~\ref{lem:S3}, as two endpoints of $\tau$ (or $\tau'$) are in the same component of $\partial \cS_0\setminus S_3$, we know $\tau\neq \tau'$ and $S_3$ is squeezed by $\tau$ and $\tau'$. Thus $p_c$ is not squeezed by $S_3$ and one of $\{\tau,\tau'\}$. In Case (2) of Lemma~\ref{lem:S3}, suppose without loss of generality that $\partial S_3=\{y,y'\}\subset \overline{z'p_a}$. Then $x\in \overline{zp_a}$ ($x$ is defined in Lemma~\ref{lem:S3} (2)). If $p_c$ is not squeezed by $\tau'$ and $S_3$, then we are done. Otherwise, $R_3$ is squeezed by $\tau'$ and $S_3$ as $R_3\cap (\tau'\cup S_3)=\emptyset$.  Let $C$ be the component of $\partial \cS_0\setminus \partial S_3$ that contains $p_a$. Then the two endpoints of $\tau'$ are contained in the same component of $C\setminus\{x\}$.  As one endpoint of $\tau$ (which is either $z$ or $z'$) and $\partial\tau'$ are contained in different components of $C\setminus\{x\}$, we know $x$, hence $R_3$ and $p_c$, are squeezed by $\tau$ and $\tau'$. Moreover, $\tau$ is squeezed by $S_3$ and $\tau'$. Thus $p_c$ is not squeezed by $\tau$ and $S_3$. Case (3) of Lemma~\ref{lem:S3} can be treated similarly.
\end{proof}

\subsection{Big points on the cycle $\omega$}
Let $\omega$ be a 6-cycle of type I or II.
We define the 6-cycle $\omega$ is \emph{big} at $x_i$, if $x_{i-1}$ and $x_{i+1}$ are not adjacent to a common vertex of type $\hat d$ in $\Delta_\Lambda$, otherwise $\omega$ is \emph{small} at $x_i$.
A simple arc in $\cS_0$ is \emph{good}, if it starts and ends at two different points in $\partial \cS_0$, and it is not boundary parallel in the surface $\cS'_0$ defined in Definition~\ref{def:boundary parallel}.

\begin{lem}
	\label{lem:not big}
Consider an embedded edge path in $\Delta_\Lambda$ with consecutive vertices $x_1,x_0,x_5$ with $x_1$ and $x_5$ having type $\hat c$ and $x_0$ having type $\hat a$.	Let $\tau_1,\tau_0,\tau_5$ be the associated arcs in $\cS$. 
Suppose the trace of $\tau_1$ and $\tau_5$ in $\cS_0$ satisfy one of the following conditions. Then $x_1$ and $x_5$ are adjacent to a common vertex of type $\hat d$ in $\Delta_\Lambda$.
\begin{enumerate}
	\item There is a good arc $\alpha\subset\cS_0$ such that $\alpha\cap \tau_i=\emptyset$ for $i=1,5$.
	\item The arcs $\tau_1$ and $\tau_5$ fit together to form a homotopically non-trivial simple closed curve on $\cS_0$ which is not homotopic to $\partial \cS_0$.
	\item There are two good arcs $\alpha\neq \alpha'$ of $\cS_0$ such that $\partial\alpha\cup\partial\alpha'$ gives four distinct points on $\partial\cS_0$, $\alpha$ and $\alpha'$ are parallel, $\tau_1\cap \alpha\subset\{z_c\}$, $\tau_5\cap \alpha'\subset\{z_c\}$ and $p_c$ is not squeezed by $\alpha$ and $\alpha'$.
    \item There is a good arc $\alpha\subset\cS_0$ such that $\alpha\cap \tau_i\subset\{z_c\}$ for $i=1,5$.
\end{enumerate}	
\end{lem}

\begin{proof}
For Assertion 1, we claim that in $\cS$, there are infinitely many homotopy classes of arcs from $p_a$ to $z_a$ such that each class contain a representative that is disjoint from $\tau_1\cup\tau_5$. This claim and Lemma~\ref{lem:disjoint} imply there are infinitely many vertices of type $\hat a$ which are adjacent to both $x_1$ and $x_5$ in $\Delta_\Lambda$, and Assertion 1 now follows from Theorem~\ref{thm:4 wheel}. It remains to show the claim. 

We first arrange $\alpha$ such that its two endpoints are contained a small open interval $I\subset \partial \cS_0$ containing $z_c$ with $I\cap \pi_0=\emptyset$, lying in different components of $I\setminus\{z_c\}$. To see this, we scissor $\cS_0$ along $\alpha$ to obtain an annulus $A_0$ punctured at $p_c$. Then $z_c$ is contained one boundary component $C$ of $A_0$, and $\tau_1$ and $\tau_5$ are simple arcs of $A_0$ from $p_c$ to $z_c$. Thus we can take $\alpha$ starting at a point on $C$ very close to $z_c$, going around $A_0$ (and staying close to $C$) once, and ending at a point on $C$ very close to $z_c$, but on different side. We still have $\alpha\cap \tau_i=\emptyset$ for $i=1,5$. This new choice of $\alpha$ can also be viewed as an arc on $\cS$ connecting two points on the boundary, such that $\alpha\cap (\tau_1\cup\tau_5)=\emptyset$ and $z_a$ and $z_c$ are in different components of $\partial \cS\setminus \alpha$. We scissor $\cS$ along $\alpha$ to obtain the annulus $A$ with two punctures $p_a$ and $p_c$. Note that $z_a$ and $z_c$ are contained in different boundary components of $A$, and $\tau_1$ and $\tau_5$ are simple arcs in $A$ from $p_c$ to $z_c$. As there is one simple arc, namely $\tau_0$, goes from $p_a$ to $z_a$ avoiding $\tau_1\cup\tau_5$, then there are infinitely many isotopy classes of simple arcs in $A$ traveling from $p_a$ to $z_a$ avoiding $\tau_1\cup\tau_5$, by taking $\tau_0$ and applying powers of Dehn twist along the boundary component of $A$ containing $z_a$. Thus the claim is proved.

Assertion 2 follows from Assertion 1 as one can find a good arc as in Assertion 1 under the assumption of Assertion 2.

For Assertion 3, if $z_c$ is squeezed by $\alpha$ and $\alpha'$, then $\tau_1\cup\tau_5$ gives a simple closed curve on $\cS_0$ which is not homotopic to $\partial \cS_0$ and we are reduced to Assertion 2. If $z_c\notin (\alpha\cup\alpha')$ and $z_c$ is not squeezed by $\alpha$ and $\alpha'$, then we can find a normal open neighborhood $N$ of $\alpha$ in $\cS_0$ with $\alpha'\subset N$ and $p_c,z_c\notin N$. As $\tau_1\cap \alpha=\emptyset$ and $\partial \tau_1\cap N=\emptyset$, up to homotopy we can assume $\tau_1\cap N=\emptyset$. Hence $\tau_1\cap \alpha'=\emptyset$. As $\tau_5\cap\alpha'=\emptyset$, we are reduced to Assertion 1.


It remains to treat the situation that $z_c\in (\alpha\cup\alpha')$. Assume without loss of generality that $z_c\in \alpha'$. By the argument in the previous paragraph, we can assume $\tau_1\cap \alpha'=\{z_c\}$. We scissor $\cS_0$ along $\alpha'$ to obtain an punctured annulus $B_0$, in which we have two copies of $\alpha'$ inside two boundary components of $B_0$, which are denoted by $(\alpha')^+$ and $(\alpha')^-$. Similarly we define $z^+_c\in (\alpha')^+$ and $z^-_c\in (\alpha')^-$. 
Then $\tau_i$ for $i=1,5$ can be viewed as a simple arc in $B_0$ with starting point $p_c$ and endpoint being either $z^+_c$ or $z^-_c$. If $\tau_1$ and $\tau_5$ in $B_0$ have the same endpoint, then we are reduced to Assertion 1; if $\tau_1$ and $\tau_5$ in $B_0$ have different endpoints, then we are reduced to Assertion 2.

Assertion 4 is similar to Assertion 3, and it is left to the reader.
\end{proof}

\begin{cor}
	\label{cor:boundary parallel 24}
Suppose $\omega$ is of type I.
Suppose $\tau_0\cap\tau_3\neq\emptyset$. Suppose at least one of $\tau_2$ and $\tau_4$ contains a trace that is boundary parallel in $\cS'_0$. Then $\omega$ is small at $x_0$.
\end{cor}

\begin{proof}
Suppose exactly one of $\tau_2$ and $\tau_4$, say $\tau_2$, enjoys the property that all traces are good. We first consider the situation that $\tau_2$ has multiple traces. Then let $S_3$ be as in Lemma~\ref{lem:S3}. Let $S_4$ be a trace of $\tau_4$ that is boundary parallel in $\cS'_0$. By Lemma~\ref{lem:boundary parallel}, there is an arc $\gamma\subset \partial\cS_0$ with $z_c\in\gamma$ such that $S_4$ and $\gamma$ bound a punctured disk $D\subset \cS_0$ with $p_c$ inside. As $\gamma_5$ is a simple arc from $p_c$ to $z_c$ avoiding $S_4$, we know $\gamma_5\subset D$. As $S_3\cap S_4=\emptyset$ and $S_3$ is not boundary parallel in $\cS'_0$, we know $S_3$ and $\gamma_5$ are in different components of $\cS_0\setminus S_4$. In particular $\gamma_5\cap S_3=\emptyset$.  On the other hand, $\gamma_1\cap S_3\subset \{z_c\}$. By possibly perturbing one endpoint of $S_3$ by a small amount, we can use Lemma~\ref{lem:not big} (1) to conclude that $\omega$ is small at $x_0$.

Now we consider the situation that $\tau_2$ has a single trace. Let $S_4,\gamma,D$ be as before. Note that $\tau_2$ and $S_4$ are in minimal position in $\cS'_0$, otherwise they form a bigon in $\cS'_0$ with the puncture $p_c$ inside the bigon, which is impossible as a trace of $\tau_3$ must start from $p_c$ and end in a point in $\partial \cS_0$ without intersecting the bigon because we are assuming $\tau_0\cap\tau_3\neq\emptyset$. As $S_4$ is isotopic to $\gamma$ in $\cS'_0$ rel endpoints, the number of intersection points between $S_4$ and $\tau_2$ is equal to the number of intersection points between $\gamma$ and $\tau_2$. 

Next we analyze the position of $\gamma$ and $S_4$. Note that $\tau_4$ must contain a trace $S'_4$ that contains one of $\{z,z'\}$. As $\{z,z'\}\subset D$, we know $S'_4$ is inside $D$ if $S'_4\neq S_4$, which implies $S'_4$ is $c$-boundary-parallel in $\cS'_0$. Thus up to replacing $S_4$ by $S'_4$, we can assume $S_4$ contains one of $\{z,z'\}$. If $S_4$ contains $z$, then the other endpoint $x'$ of $S_4$ is contained in $\overline{z'p_a}$, otherwise $\overline{z'p_a}\subset\gamma\subset D$ and $x'$ is identified via $\phi:\overline{zp_a}\to\overline{z'p_a}$ to a point $x''\in \overline{zp_a}$, which is not possible as the trace of $\tau_4$ containing $x''$ will be trapped in $D$ and have nowhere to go given that it must be $c$-boundary-parallel in $\cS'_0$ but not $a$-boundary-parallel. Similarly, if $S_4$ contains $z'$, then the other endpoint of $S_4$ is contained in $\overline{zp_a}$. 

First we consider the case $z\in S_4$ and $z\in \tau_2$. We claim that $\tau_2\cap S_4$ does not contain a point in the interior of $\tau_2$. Assume by contradiction that such a point exists. Let $w'$ be the point in $\tau_2$ after $z$ that is in $S_4\cap \tau_2$, and let $\tau_2'$ be the subarc of $\tau_2$ from $z$ to $w'$. Then $\tau_2'\cap D=\{z,w'\}$, otherwise $\tau_2'\subset D$, hence the minimal position between $\tau_2$ and $\tau_4$ means that $\tau_2'$ and part of $S_4$ form a bigon with $p_c$ inside, contradicting that $\tau_3$ has a trace from $p_c$ to a point in $\partial \cS_0$ avoiding $\tau_2\cup\tau_4$. Let $w''$ be the next point of $\tau_2$ after $w'$ that is in $S_4\cap \tau_2$ (note that $w''$ exists as $\tau_2$ does not end in the interior of $D$). Let $\tau_2''$ be the subarc of $\tau_2$ from $w'$ to $w''$. As $\tau_2'$ is outside $D$, by minimal position between $\tau_2$ and $\tau_4$, we know $\tau_2''\subset D$. We reach a contradiction by using the bigon argument. The claim follows. This claim implies that $\tau_5\cap \tau_2=\emptyset$, then $\omega$ is small at $x_0$ by $\tau_1\cap \tau_2=\emptyset$ and Lemma~\ref{lem:not big} (1).

Second we consider the case $z\in S_4$ and $z'\in \tau_2$. We claim $S_4\cap \tau_2$ has one point in the interior of $\tau_2$. Indeed, $S_4\cap \tau_2$ has at least one point in the interior of $\tau_2$, otherwise $\tau_0=\tau_2$.
Note that $S_4\cap \tau_2$ cut $\tau_2$ into subarcs. 
If $S_4\cap\tau_2$ has at least two points in the interior of $\tau_2$, then there are at least three such subarcs of $\tau_2$. Let $\tau'_2,\tau''_2$ and $\tau'''_2$ be the first three subarcs, starting from $z'\in \tau'_2\subset \tau_2$. Note that $p_c$ and $z_c$ are in different components of $D\setminus \tau'_2$, otherwise using $\tau_3\cap (\tau_2\cup\tau_4)=\emptyset$ we can conclude that $\tau_3=\tau_5$. As $\tau'_2\subset D$, we know $\tau''_2$ is outside $D$ excepts its endpoints, and $\tau'''_2\subset D$. As $\tau_2$ does not end in an interior point of $D$, we can reach a contradiction by considering the bigon between $\tau'''_2$ and a subarc of $S_4$, indeed, the trace of $\tau_3$ containing $p_c$ is trapped in this bigon, and it can not end at any point in $\cS_0$ except $p_a$, which is impossible. Thus the claim is proved. It follows that $\tau_5\cap \tau_2$ is exactly one point. We scissor $\cS_0$ along $\tau_2$ to obtain a punctured annulus $A_0$, with $z_c$ in one boundary component $C$ of $A_0$. As $\tau_5\subset D$, we know the trajectory of $\tau_5$ in $A_0$ first goes from $z_c\in C$ to another point in $C$, then jump to another boundary component and travel to $p_c$ via a simple arc. On the other hand, $\tau_1$ is a simple arc from $z_c$ to $p_c$ in $A_0$. Thus $\tau_1\cup\tau_5$ gives a homotopically non-trivial simple closed curve on $\cS_0$ which is not homotopic to $\partial \cS_0$. Hence $\omega$ is small at $x_0$ by Lemma~\ref{lem:not big} (2).

The remaining two cases ($z'\in S_4$ and $z'\in \tau_2$, and $z'\in S_4$ and $z\in \tau_2$) follow by symmetry. This finishes the discussion when $\tau_2$ has a single trace.

Suppose for $i=2,4$, $\tau_i$ contains at least one trace $\tau'_i$ which is not good. Then we deduce as before that for $i=2,4$, $\tau'_i$ and an arc $\gamma_i\subset \partial \cS_0$ containing $z_c$ bound a punctured disk $D_i$ with $p_c$ inside. Moreover, $\tau_1\subset D_2$ and $\tau_5\subset D_4$. Note that $\tau'_2\cap\tau'_4$ is at most one point in their interior, otherwise they form a bigon containing $p_c$ due to the minimal intersection of $\tau_2$ and $\tau_4$, which is impossible as $p_c$ is connected to a point in $\cS_0$ without touching $\tau'_2\cup\tau'_4$ by our $\tau_0\cap\tau_3\neq\emptyset$ assumption. It follows that $\tau'_2$ and $\tau_5$ has at most one interior intersection point, however, one interior intersection point is impossible as this would imply $p_c\notin D_2$. Thus $\tau'_2\cap 
\tau_5=\emptyset$.
This implies that $\tau_1$ and $\tau_5$ are isotopic rel endpoints in $\cS_0$, contradicting that the 6-cycle $\omega$ is embedded.
\end{proof}

\subsection{Filling type I 6-cycles}
\label{subsec:typeI}
The goal of this subsection is prove Proposition~\ref{prop:typeI}.
\begin{lem}
	\label{lem:not sequeeze 44}
Suppose $\omega$ is of type I.	
Suppose $\tau_0\cap \tau_3\neq\emptyset$.  Suppose $\tau_3$ has a trace $S_3$ satisfying at least one of the cases of Lemma~\ref{lem:S3}. Assume in additional that the 6-cycle $\omega$ is big at $x_0$. Then $p_c$ is not squeezed by two parallel traces of $\tau_4$.
\end{lem}

\begin{proof}
By Corollary~\ref{cor:boundary parallel 24}, all traces of $\tau_4$ in $\cS_0$ are good.	
We argue by contradiction and assume $T_4$ and $T'_4$ are two parallel traces of $\tau_4$ that squeeze $p_c$. As $\tau_5$ is a simple arc in $\cS_0$ from $p_c$ to $z_c$ and $\tau_5\cap \tau_4=\emptyset$, we know $\tau_5$ is squeezed by $T_4$ and $T'_4$. 
Let $S_3$ be as in Lemma~\ref{lem:S3}. Then $S_3\cap \tau_4=\emptyset$. If $S_3$ is not squeezed by $T_4$ and $T'_4$, then $S_3$ and $\tau_5$ are contained in different components of $\cS_0\setminus(T_4\cup T'_4)$. Thus $S_3\cap \tau_5=\emptyset$, moreover, as $z_c\notin S_3$, we are in Lemma~\ref{lem:S3} (1) or (2). Thus $\tau_1\cap S_3=\emptyset$ by Lemma~\ref{lem:S3}. By Lemma~\ref{lem:not big} (1), $\omega$ is small at $x_0$, contradiction.
It remains to consider the case when $S_3$ is squeezed by $T_4$ and $T'_4$. Then $S_3$ is parallel to $T_4$ and $T'_4$. It follows that either $p_c$ is not squeezed by $S_3$ and $T_4$, or $p_c$ is not squeezed by $S_3$ and $T'_4$. As $\tau_1\cap S_3\subset\{z_c\}$ by Lemma~\ref{lem:S3} and $\tau_5\cap (T_4\cup T'_4)=\emptyset$, we deduce from Lemma~\ref{lem:not big} (3) that $\omega$ is small at $x_0$, contradiction. \end{proof}
\begin{lem}
	\label{lem:nb}
	Suppose $\omega$ is of type I or II.
	Suppose $\tau_0\cap \tau_3\neq\emptyset$. Suppose $\tau_3$ has a trace $S_3$ satisfying at least one of the cases of Lemma~\ref{lem:S3}. If there is a trace $S_4$ of $\tau_4$ parallel to $S_3$ such that $p_c$ is not squeezed by $S_3$ and $S_4$, then $\omega$ is small at $x_0$.
\end{lem}
\begin{proof}
As $S_4\cap\tau_5=\emptyset$, and $S_3\cap\tau_1\subset\{z_c\}$, we are done by Lemma~\ref{lem:not big} (3). 
\end{proof}
\begin{lem}
	\label{lem:not squeeze 34}
	Suppose $\omega$ is of type I.
	Suppose $\tau_0\cap \tau_3\neq\emptyset$.  Suppose $\tau_3$ has a trace $S_3$ satisfying at least one of the cases of Lemma~\ref{lem:S3}.  
	Assume in additional that the 6-cycle $\omega$ is big at $x_0$. Then in each case of Lemma~\ref{lem:S3}, $p_c$ is not squeezed by $S_3$ and a trace  of $\tau_4$ which is parallel to $S_3$.
\end{lem}

\begin{proof}
By Corollary~\ref{cor:boundary parallel 24}, all traces of $\tau_2$ and $\tau_4$ are good.	
By contradiction we assume $p_c$ is squeezed by $S_3$ and a trace of $T_4$ of $\tau_4$. Note that not all traces of $\tau_4$ are parallel to $S_3$, otherwise by Corollary~\ref{cor:S4}, there is a trace $S_4$ of $\tau_4$ parallel to $S_3$ such that $p_c$ is not squeezed by $S_3$ and $S_4$. This and Lemma~\ref{lem:nb} lead to a contradiction.

Recall that one of $z$ and $z'$ is contained in a trace of $\tau_4$. We assume without loss of generality that $z$ is contained in a trace $T'_4$ of $\tau_4$. Let $P_{T_4}$ be the collection all traces of $\tau_4$ that are parallel to $T_4$. Similarly we define $P_{T'_4}$. Then $p_c$ is squeezed by $S_3$ and any member of $P_{T_4}$ (otherwise we will have a contradiction with Lemma~\ref{lem:not sequeeze 44}). So we are free to replace by $T_4$ by any other member of $P_{T_4}$, and the following holds.
\begin{claim}
	\label{claim:boundary position}
All boundary points of elements in $P_{T_4}$ are contained in the same component of $\partial \cS_0\setminus \partial S_3$, and any trace of $\tau_4$ not in $P_{T_4}$ has its two boundary points in different components of $\partial \cS_0\setminus \partial S_3$. Moreover, $p_c$ and $S_3$ are in the same component of $\cS_0\setminus\tau_4$.
\end{claim}


\begin{claim}
	\label{claim:zc}
The point $z_c$ is not squeezed by $S_3$ and a member of $P_{T_4}$. Moreover, $z_c\notin S_3$, hence $S_3$ is in Lemma~\ref{lem:S3} (1) or (2).
\end{claim}

\begin{proof}
We argue by contradiction and assume either $z_c\in S_3$, or $z_c$ is squeezed by $S_3$ and a member $T$ of $P_{T_4}$. Let $\tau'_5$ be a simple arc from $p_c$ to $z_c$ which is squeezed by $S_3$ and $T$ (except possibly at $z_c$). 
Take a trace $T'$ of $\tau_4$ which is not parallel to $T$. As $T'\cap (T\cup S_3)=\emptyset$, we know $T'$ and $\tau'_5\setminus\{z_c\}$ are in different components of $\cS_0\setminus(S_3\cup T)$. In particular, $\tau'_5\cap (T\cup T')=\emptyset$. We also have $\tau_5\cap (T\cup T')=\emptyset$. If we scissor $\cS_0$ along $(T\cup T')$, we obtain a disk $D$ puncture at $p_c$. In $D$ there is only one homotopy class of arcs from $p_c$ to $z_c\in \partial D$. Thus we can assume $\tau_5=\tau'_5$. In particular, $\tau_5\cap S_3\subset \{z_c\}$. On the other hand $\tau_1\cap S_3\subset \{z_c\}$ by Lemma~\ref{lem:S3}. Then Lemma~\ref{lem:not big} (4) implies that $\omega$ is not big at $x_0$, contradiction.
\end{proof}

\begin{claim}
	\label{claim:np}
$T'_4$ is not parallel to $S_3$. In particular, $z\notin T_4$ and $z'\notin T_4$.
\end{claim}

\begin{proof}
We argue by contradiction and assume $T'_4$ is parallel to $S_3$. Then $P_{T_4}=P_{T'_4}$. We repeat the discussion in Lemma~\ref{lem:S3} with $\tau_2$ replaced by $\tau_4$.
By Lemma~\ref{lem:not sequeeze 44}, the boundaries of traces of $\tau_4$ are as in Figure~\ref{fig:12} (III) or (V), where $\xi^\pm_i$ is the boundary of an element in $P_{T'_4}$. Let $x$ and $x'$ be two endpoints of $S_3$. By Claim~\ref{claim:boundary position} and Claim~\ref{claim:zc}, if we are in Figure~\ref{fig:12} (III), then up to exchanging $x$ and $x'$, either $x\in \overline{\xi^+_n\eta^+_1}$ and $x'\in \overline{\eta^+_m\xi^-_n}$, or $x\in \overline{\xi^-_1\eta^-_m}$ and $x'\in\overline{z'\eta^-_1}$. In the former case, $S_3$ must belong to Lemma~\ref{lem:S3} (2). Let $R_3$ be as in Lemma~\ref{lem:S3} (2). As $p_c$ is squeezed by $S_3$ and $T$, where $T\in P_{T'_4}$ goes from $\xi^+_n$ to $\xi^-_n$, and $R_3\cap (T\cup S_3)=\emptyset$, we know that $R_3$ is squeezed by $S_3$ and $T$.
As $x$ is aligned with $\xi^+_n$ and $x'$ is alighted with $\xi^-_n$, we know $R_3$ ends in a point in $\overline{\xi^+_n\eta^+_1}\cup\overline{\eta^+_m\xi^-_n}$. Thus the endpoint of $R_3$ and $\partial S_3$ are contained in the same side of $p_a$, contradicting Lemma~\ref{lem:S3} (2). In the latter case, let $T'\in P_{T'_4}$ be the trace going from $\xi^+_1$ to $\xi^-_1$ we still know $p_c$ is squeezed by $T'$ and $S_3$.  As $x'$ is aligned with $\xi^+_1$ and $x$ is alighted with $\xi^-_1$, we know $z_c$ is also squeezed by $T'$ and $S_3$, contradicting Claim~\ref{claim:zc}.
If we are in Figure~\ref{fig:12} (V), either $x\in \overline{\xi^+_n\eta^+_1}$ and $x'\in\overline{\zeta^+_r\xi^-_n}$, or $x\in\overline{\xi^-_1\eta^-_m}$ and $x'\in\overline{z'\zeta^-_1}$. This can be treated as before.
\end{proof}

\begin{claim}
	\label{claim5}
{\tiny }There exists a trace $S'_3$ of $\tau_3$ between two points in $\partial\cS_0$ such that $S'_3$ is not boundary parallel in $\cS'_0$ (see Definition~\ref{def:boundary parallel}), $S'_3$ is not parallel to $S_3$, and $S'_3\cap \tau_5\subset\{z_c\}$.
\end{claim}

\begin{proof}
We repeat the discussion of Lemma~\ref{lem:S3} with $\tau_2$ replaced by $\tau_4$.	
If $\tau_4$ has exactly two parallel families of traces, then by Lemma~\ref{lem:not sequeeze 44}, the boundaries of traces of $\tau_4$ is as in Figure~\ref{fig:12} (III), where $\eta^\pm_j$ is the boundary of an element in $P_{T_4}$ and $\xi^\pm_i$ is the boundary of an element in $P_{T'_4}$. By Claim~\ref{claim:boundary position}  and Claim~\ref{claim:zc}, $S_3$ either goes from an interior point of $\overline{\eta^+_m\xi^-_n}$ to an interior point of $\overline{\xi^-_1\eta^-_m}$, or from a point in $\overline{\eta^+_1\xi^+_n}$ to a point in $\overline{z'\eta^-_1}$. In each of the cases, one endpoint of $S_3$ is identified via $\phi:\overline{zp_a}\to\overline{z'p_a}$ to a point which is squeezed by two elements in $P_{T'_4}$ (recall that $n>1$ in Figure~\ref{fig:12} (III)). Thus $S_3$ is followed by another trace $S'_3$ of $\tau_3$ that is squeezed by two elements of $P_{T'_4}$. In particular, $S'_3$ is not boundary parallel in $\cS'_0$. Recall that $\tau_4\cap (\tau_3\cup\tau_5)=\emptyset$. As $p_c$ (hence $\tau_5$) and $S_3$ are in the same component of $\cS_0\setminus\tau_4$ (Claim~\ref{claim:boundary position}), we know $\tau_5$ and $S'_3$ are in different components of $\cS_0\setminus \tau_4$. Thus $\tau_5\cap S'_3=\emptyset$. Claim~\ref{claim:np} implies that $S'_3$ and $S_3$ are not parallel.

Now we assume $\tau_4$ has three parallel families of traces, then by Lemma~\ref{lem:not sequeeze 44}, the boundaries of traces of $\tau_4$ is as in Figure~\ref{fig:12} (V), where $\xi^\pm_i$ is the boundary of an element in $P_{T'_4}$. There are two subcases to consider. 
For the first subcase, suppose $\zeta^\pm_\ell$ is the boundary of an element in $P_{T_4}$. By Claim~\ref{claim:boundary position} and Claim~\ref{claim:zc}, $S_3$ goes from 
a point in $\overline{\zeta^+_1\eta^+_m}$ to a point in $\overline{z'\zeta^-_1}$. The rest of the argument is the same as the previous graph. For the second subcase, suppose $\eta^\pm_j$ is the boundary of an element in $P_{T_4}$. By Claim~\ref{claim:boundary position} and Claim~\ref{claim:zc}, $S_3$ goes from an interior point of $\overline{\eta^+_m\zeta^+_1}$ to an interior point of $\overline{\xi^-_1\eta^-_m}$.
Let $m,n,r$ be as in Case 3 of the proof of Lemma~\ref{lem:S3} and recall that $n-1\neq r$. Then either $y$ is identified via $\phi:\overline{zp_a}\to\overline{z'p_a}$ to a point between $\xi^-_1$ and $\xi^-_n$, or $y'$ is identified via $\phi$ to a point between $\zeta^+_1$ and $\zeta^+_r$ (this uses $n>1$). Thus $S_3$ is followed by a trace of $\tau_3$ which is squeezed by two parallel traces of $\tau_4$ that are not parallel to $S_3$, and we can finish in the same way as the previous paragraph.  \end{proof}

\begin{figure}[h]
	\centering
	\includegraphics[scale=0.81]{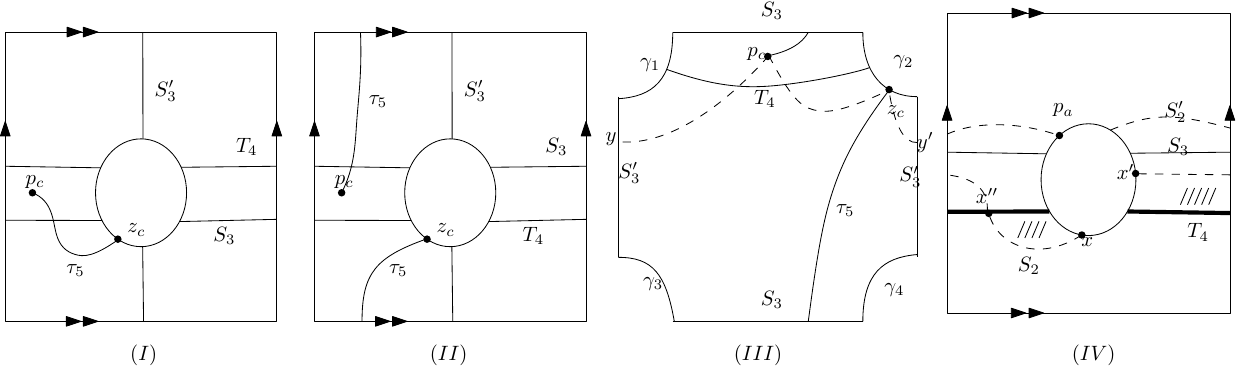}
	\caption{Proof of Lemma~\ref{lem:not squeeze 34}}
	\label{fig:squeeze}
\end{figure}


\begin{claim}
	\label{claim:T4S3}
The two endpoints of $T_4$ are contained in different components of $\pi_0\setminus\{p_a\}$. The same statement holds for $S_3$.
\end{claim}

\begin{proof}
The $T_4$ case follows from Claim~\ref{claim:np}, Lemma~\ref{lem:not sequeeze 44} and Figure~\ref{fig:12} (III) (V). The $S_3$ case follows by looking at each cases in the proof of Claim~\ref{claim5}. 
\end{proof}

\begin{claim}
	\label{claim:sides}
	The point $z_c$ and $\partial S_3$ are in different components of $\partial \cS_0\setminus\partial T_4$.
\end{claim}
\begin{proof}
By Claim~\ref{claim5}, $\tau_5\cap (T_4\cup S'_3)\subset \{z_c\}$. As $S'_3$ is not parallel to $T_4$, $\cS_0\setminus(T_4\cup S'_3)$ is an open disk with puncture $p_c$. Thus $T_4$ and $S'_3$ completely determine $\tau_5$, and up to a homeomorphism we can assume they are as in Figure~\ref{fig:squeeze} (I). As $S_3$ satisfies Lemma~\ref{lem:S3} (1) or (2), $\tau_1\cap S_3=\emptyset$. Thus $\tau_1\cup\tau_5$ forms a simple closed curve in $\cS_0$ which is homotopically non-trivial, and not homotopic to $\partial\cS_0$. The claim now follows from Lemma~\ref{lem:not big} (2).
\end{proof}


As $z_c$ and $\partial S_3$ are in different components of $\partial \cS_0\setminus\partial T_4$, up to a homeomorphism, we can assume we are in  Figure~\ref{fig:squeeze} (II). As in Claim~\ref{claim:sides}, $\tau_5$ is determined by $T_4$ and $S'_3$.
\begin{claim}
	\label{claim:npnb}
If $\tau_2$ has a trace $\gamma$ which is not parallel to $S_3$, then $\omega$ is small at $x_0$.
\end{claim}

\begin{proof}
 We scissor $\cS_0$ along $(S_3\cup S'_3)$ to obtain a closed disk $D$ with a puncture $p_c$, see Figure~\ref{fig:squeeze} (III). Let $\{\gamma_i\}_{i=1}^4$ be four corner arcs in Figure~\ref{fig:squeeze} (III). By Claim~\ref{claim:sides} and Claim~\ref{claim:zc}, $z_c$ must be in either $\gamma_1$ and $\gamma_2$ and it lies below the endpoints of $T_4$. Assume without loss of generality that $z_c\in \gamma_2$. As before, $\tau_5$ is completely determined by $T_4$ and $S'_3$, hence $\tau_5$ is as in Figure~\ref{fig:squeeze} (III). As $\gamma\cap(S_3\cup S'_3)=\emptyset$, $\gamma$ is a good trace in $\cS_0$, and $\gamma$ is not parallel to $S_3$, we know $\gamma$ starts in an interior point $x$ of $\gamma_1$ or $\gamma_2$, and ends in an interior $x'$ of $\gamma_3$ or $\gamma_4$. Let $\tau'_1$ be the dashed arc from $p_c$ to $z_c$ in Figure~\ref{fig:squeeze} (III), and let $\tau''_1$ be the union of the dashed arc from $p_c$ to $y$, and the dashed arc from $y'$ to $z_c$. We assume $y$ and $y'$ are identified if we glued sides of $D$ to form $\cS_0$, so $\tau''_1$ actually give an arc in $\cS_0$. We claim up to an isotopy, $\gamma$ has empty intersection with one of $\tau'_1$ and $\tau''_1$. Assuming this claim, noting that $\tau_1$ is characterized (up to isotopy) as the arc having empty intersection with $\gamma$ and $S_3$, thus $\tau_1=\tau'_1$ or $\tau''_1$ up to isotopy, and $\tau_1\cup\tau_5$ forms a simple closed curve on $\cS_0$ satisfying the requirement of  Lemma~\ref{lem:not big} (2), implying $\omega$ is small at $x_0$.

It remains to prove the claim. If $x$ is above $T_4$, then the claim is clear. If $x$ is contained or below $T_4$, then the claim reduces to showing that $\gamma\cap T_4$ is either empty or equal to $\{x\}$. If this is not the case, then $\gamma$ forms a bigon with $T_4$, with the puncture $p_c$ insider the bigon. Let $R_3$ be the trace of $\tau_3$ containing $p_c$. As $\tau_3\cap (\tau_2\cup \tau_4)=\emptyset$, we know $R_3\cap (\gamma\cup T_4)=\emptyset$. It follows that $R_3$ is contained in the interior of $\cS_0$. Since we assume $\tau_0\cap \tau_3\neq\emptyset$, $R_3$ ends on a point in $\cS_0$. This is a contradiction. 
\end{proof}

The only possibility left is that all traces of $\tau_2$ are parallel to $S_3$. Let $S'_2$ be the trace of $\tau_2$ containing $p_a$, and let $S_2$ be the trace of $\tau_2$ containing one of $\{z,z'\}$.



\begin{claim}
	\label{claim:squeeze}
Suppose all traces of $\tau_2$ are parallel to $S_3$.  Then $S'_2\neq S_2$ and $p_c$ is squeezed by $S'_2$ and $S_2$.  
\end{claim}

\begin{proof}
From now on, we choose $T_4$ such that there are no other traces of $\tau_4$ which is squeezed by $T_4$ and $S_3$. Claim~\ref{claim:T4S3} and Claim~\ref{claim:sides} still hold with such choice. By Claim~\ref{claim:T4S3}, $\{z,z'\}$ and $p_a$ are in different sides of $\partial\cS_0\setminus\partial S_3$. As $S'_2$ and $S_2$ are parallel to $S_3$, $\partial S'_2$ and $\partial S_2$ are in different sides of $\partial\cS_0\setminus\partial S_3$. Thus $S'_2\neq S_2$, and $S_3$ is squeezed by $S'_2$ and $S_2$.

Now suppose $\tau_4$ has two parallel families of traces. Thus we assume we are in Case 2 of the proof of Lemma~\ref{lem:S3}, with the role of $\tau_2$ in Lemma~\ref{lem:S3} replaced by $\tau_4$. We will use the same notation as in Lemma~\ref{lem:S3}, see Figure~\ref{fig:12} (III). By our choice of $T_4$, either $\partial T_4=\eta^\pm_m$ or $\partial T_4=\eta^\pm_1$.
Note that $S_3$ must go from an interior point of $\overline{\eta^+_m\xi^-_n}$ to an interior point of $\overline{\xi^-_1\eta^-_m}$, as the other possibility  -  $S_3$ going from a point in $\overline{\eta^+_1\xi^+_n}$ to a point in $\overline{z'\eta^-_1}$ (see the proof of Claim~\ref{claim5}), is ruled out by Claim~\ref{claim:sides}. It follows that $\partial T_4=\eta^\pm_m$. Then $\partial T_4$ and $p_a$ are in different components of $\partial \cS_0\setminus \partial S_3$. As $p_a\in \partial S'_2$ and $S'_2$ and $S_3$ are parallel, we know $\partial T_4$ and $\partial S'_2$ are in different components of $\partial \cS_0\setminus \partial S_3$. As each of $T_4$ and $S'_2$ is parallel and disjoint from $S_3$, we know $S'_2\cap T_4=\emptyset$,  and $S_3$ is squeezed by $S'_2$ and $T_4$.


We show Claim~\ref{claim:squeeze} holds that if $\partial S_2\cap\partial T_4=\emptyset$ and $\partial S_2$ is contained in the same component of $\partial \cS_0\setminus \partial T_4$. To see this, first we show $S_2\cap T_4=\emptyset$. Indeed, if we scissor $\cS_0$ along $S_3$ to obtain an annulus $A_0$ with its boundary components $C^+$ and $C^-$, then one of $\{T_4,S_2\}$ has its endpoints in $C^+$, another one has endpoints in $C^-$. Thus the only way for $T_4$ and $S_2$ to intersect, is that they form a bigon with $p_c$ inside. However, this is impossible as the trace of $S_3$ starting at $p_c$ is trapped in the interior of this bigon (as $\tau_3\cap (\tau_2\cup\tau_4)=\emptyset$), hence this trace can not reach a point in $\partial \cS_0$, contradicting that $\tau_0\cap \tau_3\neq\emptyset$. Thus $S_2\cap T_4=\emptyset$. Claim~\ref{claim:sides} and Claim~\ref{claim:T4S3} imply that $\{z,z_c,z'\}$ and $\partial S_3$ are in different components of $\partial\cS_0\setminus\partial T_4$. As $S_2$ contains one of $\{z,z'\}$, $\partial S_2$ and $\partial S_3$ are in different components of $\partial\cS_0\setminus\partial T_4$. This together with $S_2\cap T_4=\emptyset$ imply that $S_3$ is squeezed by $S_2$ and $T_4$. As $p_c$ is squeezed by $S_3$ and $T_4$, $p_c$ is squeezed by $S'_2$ and $S_2$. 

If $\partial S_2\cap\partial T_4\neq\emptyset$, then we can deduce Claim~\ref{claim:squeeze} by a similar argument.

It remains to consider $\partial S_2\cap\partial T_4=\emptyset$ and the two endpoints of $S_2$ are contained in different components of $\partial \cS_0\setminus \partial T_4$. We refer to Figure~\ref{fig:squeeze} (IV) for the following discussion. Let $\partial S_2=\{x,x'\}$ and suppose $x\in \{z,z'\}$. Then $x'$ is squeezed by $S_3$ and $T_4$. By looking at the annulus $A_0$ defined before, we know $S_2$ and $T_4$ intersect at exactly one point, denoted $x''$. Note that $S_2, T_4$ and part of $\partial \cS_0$ together bound two triangular regions with apex $x''$ (see the two shaded regions in Figure~\ref{fig:squeeze} (IV)). Let $D$ (resp. $D'$) be the triangular region containing $x$ (resp. $x'$).
If $p_c$ is not squeeze by $S'_2$ and $S_2$, then $p_c\in D'$. In the rest of the proof, we will assume $p_c\in D'$ and deduce a contradiction.


Let $\{y,y'\}=\partial S_3$. By previous discussion, we assume $y\in \overline{\eta^+_m\xi^-_n}$ and $y'\in \overline{\xi^-_1\eta^-_m}$. Let $L$ be the trace of $S_3$ containing $p_c$.
As $\tau_3\cap (\tau_2\cup\tau_4)=\emptyset$, $L$ ends at a point $w$ in $D'\cap\partial\cS_0$.
There are two possibilities of the location of $\eta^\pm_m=\partial T_4$ to consider.

 Suppose $\eta^+_m\in D$ and $\eta^-_m\in D'$. Then 
	$w$ is an interior point of $\overline{\eta^-_my'}$. Note that the arc $D\cap \partial\cS_0$ goes from $\eta^+_m$ to $x$, passing all the $\eta^+$ and $\xi^+$ points. 	As $n>1$ in Figure~\ref{fig:12} (III), the points in the interval $\overline{\xi^-_1\eta^-_m}$ (in particular $w$) are identified via $\phi:\overline{zp_a}\to\overline{z'p_a}$ to points which are contained in $D\cap \partial\cS_0$. We define \emph{the trace of $\tau_3$ following $(L,w)$} to be the trace of $\tau_3$ containing $w'$, where $w'$ is identified to $w$ via $\phi$.
	As $\tau_3\cap (S_{2}\cup T_{4})=\emptyset$, the trace of $S_3$ following $(L,w)$ is trapped in $D$. The only way this could happen is $z_c\in D\cap \partial\cS_0$ and this trace ends at $z_c$, however, this contradicts the existence of $S_3$. 

 Suppose $\eta^-_m\in D$ and $\eta^+_m\in D'$. Then 
	$w$ is an interior point of $\overline{\eta^+_my}$. By the argument in the first paragraph of the proof of Claim~\ref{claim5}, the traces $L_1$ and $L_2$ of $\tau_3$ following
 $(S_3,y)$ and $(L,w)$ respectively are squeezed by two elements of $P_{T'_4}$. Note that in $\overline{\eta^+_m\xi^-_n}$ we have $x'$ sitting between $w$ and $y$. As $\tau_2\cap\tau_3=\emptyset$, the trace of $\tau_2$ following $(S_2,x')$ is squeezed by $L_1$ and $L_2$, contradicting that all traces of $\tau_2$ are parallel to $S_3$.

Assume $\tau_4$ has three parallel families of traces. We use the same notation as Case 3 of proof of Lemma~\ref{lem:S3}, with $\tau_2$ replaced by $\tau_4$. By the proof of Claim~\ref{claim5}, either $S_3$ goes from 
a point in $\overline{\zeta^+_1\eta^+_m}$ to a point in $\overline{z'\zeta^-_1}$, or $S_3$ goes from a point of $\overline{\eta^+_m\zeta^+_1}$ to a point of $\overline{\xi^-_1\eta^-_m}$. However, the former can be ruled out by Claim~\ref{claim:sides}. Let $\{y,y'\}=\partial S_3$.
Up to exchange the role of $y$ and $y'$, we have $y\in\overline{\eta^+_m\zeta^+_1}$, $y'\in\overline{\xi^-_1\eta^-_m}$, $n-1\neq r$ and $\partial T_4=\eta^\pm_m$. We can repeat the previous argument, until we reach the case  $\partial S_2\cap\partial T_4=\emptyset$ and the two endpoints of $S_2$ are contained in different components of $\partial \cS_0\setminus \partial T_4$. Let $D,D',L$ be defined as before.

Suppose $\eta^+_m\in D$ and $\eta^-_m\in D'$. Then 
$w$ is an interior point of $\overline{\eta^-_my'}$. The the arc $D\cap \partial\cS_0$ goes from $\eta^+_m$ to $x$, passing all the $\eta^+$ and $\xi^+$ points. If $n-1>r$, then the points in the interval $\overline{\xi^-_1\eta^-_m}$ (in particular $w$) are identified via $\phi:\overline{zp_a}\to\overline{z'p_a}$ to points which are contained in $D\cap \partial\cS_0$, and we deduce a contradiction as before. If $n-1<r$, as $n>1$, interior points $\overline{\xi^-_1\eta^-_m}$ are identified via to points which are squeezed by two of $\{\gamma_k\}_{k=1}^r$ (as defined in Lemma~\ref{lem:S3} Case 3). Then the traces $L_1$ and $L_2$ of $\tau_3$ following
$(S_3,y')$ and $(L,w)$ respectively are squeezed by two elements of $\{\gamma_k\}_{k=1}^r$. In $\overline{\xi^-_1\eta^-_m}$ we have $x'$ sitting between $w$ and $y$. As $\tau_2\cap\tau_3=\emptyset$, the trace of $\tau_2$ following $(S_2,x')$ is squeezed by $L_1$ and $L_2$, contradicting that all traces of $\tau_2$ are parallel to $S_3$.


Suppose $\eta^-_m\in D$ and $\eta^+_m\in D'$. If $n-1>r$, then we conclude in the same way as the case of $\tau_4$ having two parallel families. If $n-1<r$, then the points in the interval $\overline{\eta^+_m\xi^-_n}$ (in particular $w$) are identified via $\phi$ to points which are contained in $D\cap \partial\cS_0$, and we reach a contradiction as before.
\end{proof}

If we apply Lemma~\ref{lem:not sequeeze 44} with the role of $\tau_4$ replaced by $\tau_2$, we know $p_c$ is not squeezed by two parallel traces of $\tau_2$, which contradicts Claim~\ref{claim:squeeze}. This finishes the proof.
\end{proof}


Let $\bar\cS$ be the three punctured torus defined in Section~\ref{subsec:translation}. Let $\bar\cS'$ be $\bar\cS$ with its punctures filled in. For an arc $\tau\in \cS$, let $\bar\tau$ be the associated arc in $\bar\cS$.
\begin{cor}
	\label{cor:b}
Suppose $\omega$ is of type I.	
Suppose $\tau_0\cap\tau_3\neq\emptyset$.  Suppose $\tau_3$ has a trace $S_3$ satisfying at least one of the cases of Lemma~\ref{lem:S3}. Assume in addition that the 6-cycle $\omega$ is big at $x_0$. Then there exists a simple arc $\tau$ in $\cS$ from $p_a$ to $p_c$ such that
\begin{enumerate}
	\item $\tau$ has empty intersection with $\tau_i$ for $i=1,0,5,4$ except at endpoints;
	\item the concatenation of $\bar\tau,\bar\tau_1,\bar\tau_0$ gives homotopically non-trivial loop on $\cS'$, moreover, the same holds true for $(\bar\tau,\bar\tau_0,\bar\tau_5)$ and $(\bar\tau,\bar\tau_5,\bar\tau_4)$.
\end{enumerate}
\end{cor}

\begin{proof}
By Corollary~\ref{cor:boundary parallel 24}, all traces of $\tau_2$ and $\tau_4$ are good. Note that $S_3$ is not parallel to a trace of $\tau_4$, otherwise $\omega$ is small at $x_0$ by Lemma~\ref{lem:nb} and Lemma~\ref{lem:not squeeze 34}. In particular, as $S_3\cap \tau_4=\emptyset$, $\tau_4$ can not have three parallel families of traces.
\begin{figure}[h]
	\centering
	\includegraphics[scale=0.76]{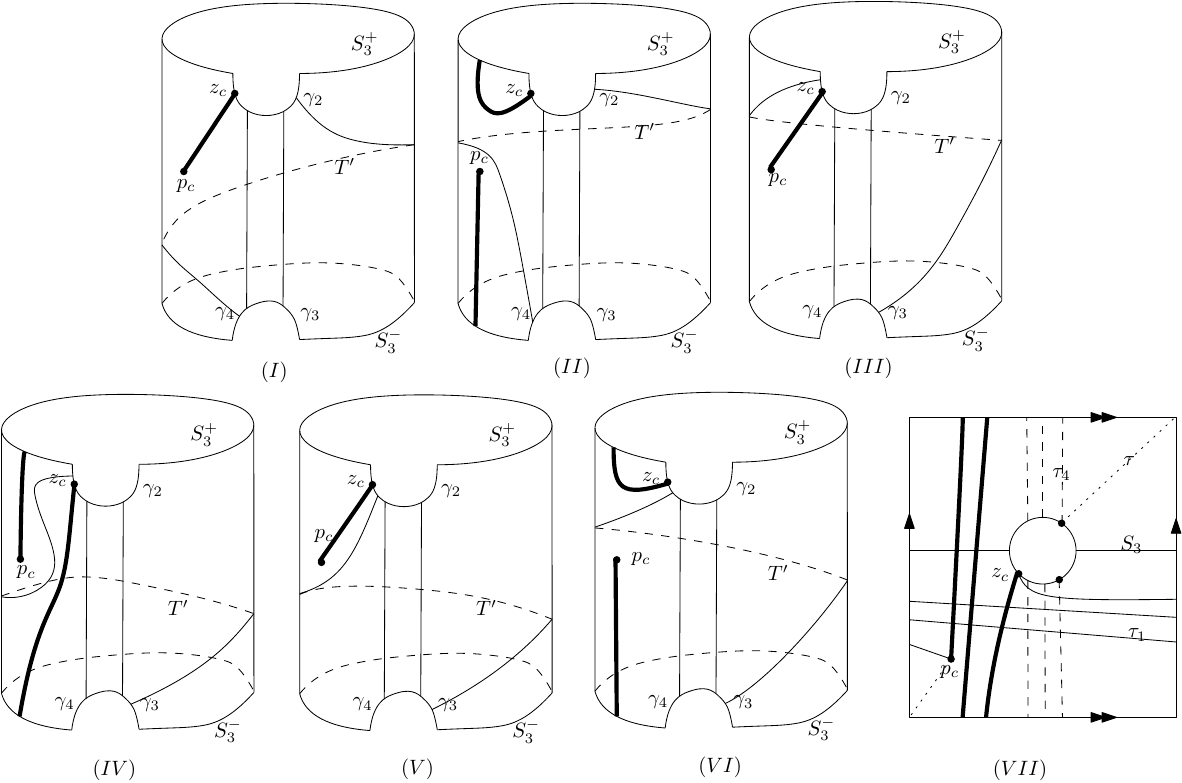}
	\caption{The thickened arcs are possibilities of $\tau_5$.}
	\label{fig:guan}
\end{figure}

Suppose $\tau_4$ has two parallel families of traces. Let $T$ and $T'$ be two traces of $\tau_4$ that are not parallel, and let $P_T$ be the traces of $\tau_4$ that are parallel to $T$. Similarly we define $P_{T'}$. As $S_3$ is not parallel to $T$, components of $\cS_0\setminus (S_3\cup P_T)$ are (possibly once punctured) open disks. Let $D$ be the component of $\cS_0\setminus (S_3\cup P_T)$ containing $p_c$. By Lemma~\ref{lem:not sequeeze 44}, $D$ is also the only component of $\cS_0\setminus (S_3\cup P_T)$ such that $\partial D\cap \partial \cS_0$ has four connected components, denoted $\{\gamma_i\}_{i=1}^4$. Suppose $z_c\in \gamma_1$. We scissor $\cS_0$ along $S_3$ to obtain a punctured annulus $A_0$, with each of its boundary components contain a copy of $S_3$. Let $C^+$ be the boundary component of $A_0$ containing $z_c$. Then $\tau_1$ is a simple arc in $A_0$ from $p_c$ to $z_c\in C^+$. Each arc of $P_T$ connects two boundary components of $A_0$, and up to homeomorphism, we can assume they are vertical arcs as in Figure~\ref{fig:guan}. Let $\{\gamma_i\}_{i=1}^4$ be components in $\partial D\cap \partial \cS_0$ as in Figure~\ref{fig:guan}. As $T'$ is not parallel to $T$ and $S_3$, $T'\subset D\cup\partial D$. Thus either $T'$ goes from a point in $\gamma_1$ to a point in $\gamma_3$, or from a point in $\gamma_2$ to a point in $\gamma_4$. Once we know the endpoints of $T'$, this is enough to determine $T'$ up to two different possibilities, depending on the position of $p_c$ relative to $T'$. We refer to Figure~\ref{fig:guan} for possibilities of $T'$ in $A_0$. Note that $T'$ and $P_T$ completely determine $\tau_5$, as $\tau_5$ is disjoint from them. By Lemma~\ref{lem:not big} (1) or (2), if $\omega$ is big at $x_0$, Figure~\ref{fig:guan} (IV) is the only possibility. Let $\gamma'_1$ be the subarc of $\gamma_1$ between $z_c$ and an endpoint of $S^+_s$. By the discussion in Case 2 of Lemma~\ref{lem:S3} (with role of $\tau_2$ replaced by $\tau_4$), $z_c$ is not squeezed by parallel traces of $\tau_4$. So each trace of $P_{T'}$ gives a simple arc in $A_0$ starting in an interior point of $\gamma'_1$, and ending in a point in $\gamma_3$. Let $D'$ be component of $\cS_0\setminus (P_T\cup P_{T'}\cup S_3)$ containing $p_c$. We also view $D'$ as a subset of $A_0$. As $p_c$ is not squeezed by two traces in $P_{T'}$ by Lemma~\ref{lem:not sequeeze 44}, we know $\partial D'\cap \gamma_3$ is a non-trivial subarc of $\gamma_3$, and $\partial D'\cap \gamma_3$ corresponds to the arc $\overline{\eta^+_m\xi^-_n}$ in Figure~\ref{fig:12} (III). Hence $p_a\in \partial D'\cap \gamma_3$, and there is a simple arc $\tau$ in $D'$ from $p_c$ to $p_a$. One readily verifies that $\tau$ satisfies the desired properties (note that the case of $(\bar\tau,\bar\tau_5,\bar\tau_4)$ follows from that one of $P_T$ and $P_{T'}$ contain at least two traces, as in Case 2 of Lemma~\ref{lem:S3}).




If $\tau_4$ only has one parallel class of traces, then let $D$ be the component of $\cS_0\setminus (S_3\cup\tau_4)$ that contains $p_c$. By the discussion in Case 1 of Lemma~\ref{lem:S3}, $p_a\in \partial D$. Let $\tau$ be a simple arc in $D$ from $p_c$ to $p_a$. See Figure~\ref{fig:guan} (VII). Then it satisfies all the desired properties.
\end{proof}

\begin{lem}
	\label{lem:replace}
Suppose $\omega$ has type I and $\omega$ is big at $x_0$.
Suppose $\tau_2$ has a single trace on $\cS_0$, and $\tau_4$ has multiple traces on $\cS_0$. Then there is a simple arc $\tau'_1$ on $\cS_0$ from $p_c$ to $z_c$ such that $\tau_2\cap \tau'_1=\emptyset$, and at least one of the three cases of Lemma~\ref{lem:S3} holds with $\tau_1$ replaced by $\tau'_1$.
\end{lem}

\begin{proof}
As $\omega$ is big at $x_0$, all traces of $\tau_2$ and $\tau_4$ in $\cS_0$ are good by Corollary~\ref{cor:boundary parallel 24}. We scissor $\cS_0$ along $\tau_2$ to obtain a punctured torus $A_0$. Let $C^+$ and $C^-$ be the two boundary components of $A_0$. We assume without loss of generality that $z_c\in C^+$ and $\overline{zp_a}\subset C^+$. Then $\overline{z'p_a}\subset C^-$. As $\tau_2\cap\tau_3=\emptyset$, each trace of $\tau_3$ gives a simple arc in $A_0$.
The conclusion of the lemma is clear when there is a trace of $\tau_3$ traveling from a point in $C^+$ to a point in $C^-$. It remains to consider the case that no traces of $\tau_3$ travel from a point in $C^+$ to a point in $C^-$. We will show this is not possible by deducing a contradiction.
As $\tau_4$ has multiple traces which are all good, we apply Lemma~\ref{lem:S3} with the role of $\tau_2$ played by $\tau_4$, and let $S_3$ be a trace of $\tau_3$ satisfying the conclusion of Lemma~\ref{lem:S3}, with $\tau_1$ replaced by $\tau_5$. As the two endpoints of $S_3$ are in the same component of $A_0$, and $S_3$ is not boundary parallel in $\cS'_0$, we know $S_3$ and $\tau_2$ are parallel. 
Applying Lemma~\ref{lem:not squeeze 34} with the role of $\tau_4$ replaced by $\tau_2$, we know $p_c$ is not squeezed by $S_3$ and $\tau_2$. Applying Lemma~\ref{lem:nb} with the role of $\tau_4$ replaced by $\tau_2$, we know $\omega$ is small at $x_0$, contradiction.
\end{proof}




A type I $6$-cycle $\omega$ has a \emph{good} point at $x_0$, if $\omega$ is big at $x_0$ and at least one of the following is true:
\begin{enumerate}
	\item both $\tau_2$ and $\tau_4$ have multiple traces on $\cS_0$;
	\item one of $\{\tau_2,\tau_4\}$, say $\tau_4$, has multiple traces on $\cS_0$, however, the other one, say $\tau_2$ only has unique trace on $\cS_0$; moreover, after the replacement in Lemma~\ref{lem:replace}, the new 6-cycle $\omega'$ obtained by replacing $x_1$ by $x'_1$ is still big at $x_0$;
	\item exactly one of $\{\tau_2,\tau_4\}$ has multiple traces on $\cS_0$; moreover, $\omega$ is small at both $x_1$ and $x_5$.
\end{enumerate}

\begin{cor}
	\label{cor:goodpoint}
Suppose $\omega$ has type I with a good point at $x_0$. Then $\omega$ has property $(*)$.
\end{cor}

\begin{proof}
We can assume $\tau_0\cap \tau_3\neq\emptyset$, otherwise $\omega$ already satisfies property $(*)$. Let $\bar \Delta_0$ be as in Section~\ref{subsec:relarc}. If we are in Case 1, by Lemma~\ref{lem:S3} and Corollary~\ref{cor:b}, there is a vertex $\bar z$ of type $\hat b$ which is adjacent to each of $\{\bar x_1,\bar x_0,\bar x_5,\bar x_4\}$ in $\bar \Delta_0$, where $\bar x_i$ denotes the image of $x_i$ under $\Delta_0\to\bar\Delta_0$. Moreover, Corollary~\ref{cor:b} implies that each of the 3-cycles $\bar z\bar x_1\bar x_0$, $\bar z\bar x_0\bar x_5$, $\bar z\bar x_5\bar x_4$ bound a 2-face. Thus there is a lift $z$ of $\bar z$ in $\Delta_0\subset \Delta_\Lambda$ such that $z$ is adjacent to each of $\{x_1,x_0,x_5,x_4\}$. Similarly, if we exchange the role of $\tau_2$ and $\tau_4$, we know there is a vertex $z'$ of type $\hat b$ which is adjacent to each of $\{x_5,x_0,x_1,x_2\}$. Note that $z=z'$, otherwise the 4-cycle $x_1zx_5z'$ would be embedded and Theorem~\ref{thm:4 wheel} implies that there is a vertex $y$ of type $\hat d$ adjacent to each of $\{x_1,z,x_5,z'\}$, contradicting that $\omega$ is big at $x_0$. Now we consider the 4-cycle $x_2x_3x_4z$ with vertex types being $\hat a,\hat c,\hat a,\hat b$. As this 4-cycle is embedded, Lemma~\ref{lem:special4cycle} implies that $z$ is adjacent to $x_3$. Thus $z$ is adjacent to each of $\{x_1,x_3,x_5\}$, and $\omega$ has property $(*)$. 

If we are in Case 2, we apply the previous argument to $\omega'$, namely, by Lemma~\ref{lem:S3} and Corollary~\ref{cor:b}, there is a vertex $z$ of type $\hat b$ which is adjacent to each of $\{x'_1,x_0,x_5,x_4\}$, and there is a vertex $z'$ of type $\hat b$ which is adjacent to each of $\{x_5,x_0,x'_1,x_2\}$. As $\omega'$ is big at $x_0$, $z=z'$ by the same argument as before, and we deduce as before that $z$ is adjacent to each of $x'_1,x_3,x_5$. As $z$ is adjacent to each of $x_0$ and $x_2$, by considering the embedded 4-cycle $zx_0x_1x_2$ and applying Lemma~\ref{lem:special4cycle}, we know $z$ is adjacent to $x_1$, as desired.

Suppose we are in Case 3. Assume without loss of generality that $\tau_2$ has multiple traces on $\cS_0$ and $\tau_4$ has single trace on $\cS_0$. We claim $\tau_2\cap\tau_4=\emptyset$ except at endpoints.
As $\omega$ is small at $x_1$, there are infinitely many vertices of type $\hat c$ that are adjacent to both $x_0$ and $x_2$. Thus in $\cS_0$ there are infinitely many homotopy classes of simple arcs from $p_c$ to $z_c$ such that each class has a representative with empty intersection with $\tau_2$. As $\omega$ is big at $x_0$, by Corollary~\ref{cor:boundary parallel 24}, all traces of $\tau_2$ and $\tau_4$ in $\cS_0$ are good. Thus among all the cases in Lemma~\ref{lem:S3}, the only possibility for $\tau_2$ is that $\tau_2$ has only one parallel class of traces, moreover, $p_c$ is not squeezed by any two traces of $\tau_2$. By the proof of Lemma~\ref{lem:S3}, there exists a trace $S_3$ of $\tau_3$ such that $S_3$ is squeezed by two traces of $\tau_2$. As all traces of $\tau_2$ are parallel to $S_3$ and $p_c$ is not squeezed by any two traces of $\tau_2$, there is an open tubular neighborhood $N$ of the submanifold $S_3$ in $\cS_0$ bounded by $\alpha_1$ and $\alpha_n$, where $\alpha_1$ and $\alpha_n$ are as in Case 1 of the proof of Lemma~\ref{lem:S3}. As $\tau_4$ has single trace, it gives an simple arc in $\cS_0$ from one of $\{z,z'\}$ to $p_a$. As none of $\{z,z',p_a\}$ is squeezed by two traces of $\tau_2$, and $\tau_4\cap S_3=\emptyset$, so we can homotopy $\tau_4$ outside $N$, which implies the claim.

We scissor $\cS_0$ along $\tau_4$ to obtain a punctured annulus $A_0$. As $\tau_2\cap(\tau_0\cup\tau_4)=\emptyset$ and $\tau_5\cap(\tau_0\cap\tau_4)=\emptyset$ except possibly at endpoints, we know the trajectory of each of $\tau_2$ and $\tau_5$ in $A_0$ is a simple arc from $p_c$ to a point in $\partial A_0$. Using the topology of punctured annulus, we can replace $\tau_5$ by another simple arc $\tau'_5$ in $A_0$ from $p_c$ to $z_c$ such that $\tau'_5\cap\tau_2=\emptyset$. Then $\tau'_5$ also gives a simple arc in $\cS_0$ such that $\tau'_5\cap (\tau_0\cup\tau_4)=\emptyset$. This means that there exists a vertex $z\in \Delta_\Lambda$ of type $\hat c$ such that $z$ is adjacent to each of $\{x_0,x_2,x_4\}$. Assume $z\notin \{x_1,x_3,x_5\}$, otherwise $\omega$ readily has property $(*)$. Applying Theorem~\ref{thm:4 wheel} to embedded 4-cycles $zx_0x_1x_2$, $zx_2x_3x_4$ and $zx_4x_5x_0$, we know that for $i=1,3,5$, there exists vertex $x'_i$ of type $\hat d$ such that $x'_i$ is adjacent to each of $\{z,x_{i-1},x_i,x_{i+1}\}$. Apply Theorem~\ref{thm:weakflagD} to the 6-cycle $x_0x'_1x_2x'_3x_4x'_5$ inside $\Delta_{\Lambda,\{a,d,b\}}$ to deduce that there is a vertex $z'$ of type $\hat b$ such that $z'$ is adjacent to each of $\{x_0,x_2,x_4\}$. Applying Lemma~\ref{lem:special4cycle} to the 4-cycles $z'x_0x_1x_2$, $z'x_2x_3x_4$ and $z'x_4x_5x_0$, we know $z'$ is adjacent to each of $\{x_1,x_3,x_5\}$.
\end{proof}



\begin{lem}
	\label{lem:b}
	Let $\tau_0,\tau_5,\tau_4,\tau_3$ be four arcs in $\cS$ corresponding to consecutive vertices $x_0,x_5,x_4,x_3$ in an edge path in $\Delta_\Lambda$ such that the vertices have type $\hat a,\hat c,\hat a,\hat c$ respectively. Suppose the following are true:
	\begin{enumerate}
		\item $\tau_0\cap \tau_4=\emptyset$ except at endpoints and $\tau_4$ is a good trace in $\cS_0$;
		\item $\tau_3\cap \tau_0$ has exactly one interior intersection point;
		\item a subarc of $\tau_0$ containing $z_a$, a subarc of $\tau_3$ containing $z_c$ and one of the two arcs in $\partial \cS$ from $z_a$ to $z_c$ bound a disk in $\cS$.
	\end{enumerate}
	Then there is a vertex of type $\hat b$ adjacent to each of $\{x_0,x_5,x_4,x_3\}$.
\end{lem}

	\begin{figure}[h]
		\centering
		\includegraphics[scale=1]{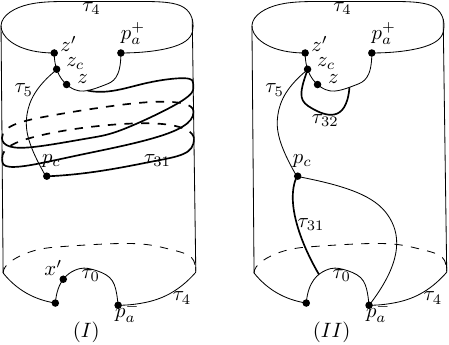}
		\caption{Proof of Lemma~\ref{lem:b}.}
		\label{fig:2intersect}
	\end{figure}

\begin{proof}
By Assumption 2, $\tau_3$ has two traces in $\cS_0$, one of them, denoted by $\tau_{31}$, contains $p_c$, and another trace $\tau_{32}$ contains $z_c$. 
Let $x\in \cS_0$ be the other endpoint of $\tau_{31}$. Then $x$ is identified to $x'\in \pi_0$ via $\phi:\overline{zp_a}\to\overline{z'p_a}$, and $\tau_{32}$ goes from $x'$ to $z_c$. We scissor $\cS_0$ along $\tau_4$ to obtained a punctured torus $A_0$ (by Assumption 1). Let $C^+$ and $C^-$ be the two boundary components of $A_0$.
Assume without loss of generality that $z_c\in C^+$. First we consider the case $x\in \overline{zp_a}$, see igure~\ref{fig:2intersect} (I). Then $x'\in \overline{z'p_a}$. Then the trajectory of $\tau_{31}$ in $A_0$ goes from $p_c$ to $x\in C^+$. Assumption 3 implies that $\tau_{32}$ and a subarc of $\partial \cS_0$ bound a disk in $\cS_0$. As the two endpoints of $\tau_{32}$ are contained in different components of $\partial\cS_0\setminus\partial \tau_4$, we know $\tau_{32}\cap\tau_4\neq\emptyset$, which is a contradiction. The remaining case is that $x\in \overline{z'p_a}$. Then $x'\in \overline{zp_a}$. Assumption 3 implies that $\overline{x'z}$, $\tau_{32}$ and $\overline{zz_c}$ bound a disk in $\cS_0$. Up to a homeomorphism of $A_0$ fixing the boundary, we can assume the trajectories of $\tau_5$ and $\tau_{31}$ are in Figure~\ref{fig:2intersect} (II). Let $\tau$ be the arc in Figure~\ref{fig:2intersect} (II) from $p_c$ to $p^-_a$ (which is the copy of $p_a$ in $C^-$) in $A_0$. Let $\bar\tau$  be the image of $\tau$ under $\cS\to \bar\cS$. 
Note that the concatenation of $\bar\tau,\bar\tau_0,\bar\tau_5$, the concatenation of $\bar\tau,\bar\tau_5,\bar\tau_4$ and the concatenation of $\bar\tau,\bar\tau_4,\bar\tau_3$ give three homotopically non-trivial simple loops in $\bar\cS'$ (which is defined to be $\bar\cS$ with punctured added back to the surface). Thus in $\bar\Delta'_0$, there exists a vertex $\bar z$ of type $\hat b$ represented by $\bar\tau$ satisfying that the 3-cycle $\bar z,\bar x_0,\bar x_5$, the 3-cycle $\bar z,\bar x_5,\bar x_4$ and the 3-cycle $\bar z,\bar x_4,\bar x_3$ are all non-degenerate. Then a lift of $\bar z$ in $\Delta_0$ will satisfy the desired properties in the lemma. 
\end{proof}

%
%



%

\begin{prop}
	\label{prop:typeI}
Suppose $\omega$ is a 6-cycle of type I in $\Delta_\Lambda$. Then there exists a vertex $z\in \Delta_\Lambda$ of type $\hat b$ or $\hat a$ such that $z$ is adjacent to each of $\{x_1,x_3,x_5\}$.
\end{prop}

\begin{proof}
If $\omega$ is small at $x_0,x_2$ and $x_4$, then for $i=0,2,4$, there exists $x'_i$ of type $\hat d$ such that $x'_i$ is adjacent to each of $x_{i-1}$ and $x_{i+1}$. We apply Theorem~\ref{thm:weakflagD} to the 6-cycle $x'_0x_1x'_2x_3x'_4x_5$ in $\Delta_{\Lambda,\{c,d,b\}}$ to find a vertex $z$ of type $\hat b$ such that $z$ is adjacent to each of $\{x_1,x_3,x_5\}$.

Now we assume $\omega$ is big at one of $\{x_0,x_2,x_4\}$. Suppose without loss of generality that $\omega$ is big at $x_0$. If $\omega$ is good at $x_0$, then we are done by Corollary~\ref{cor:goodpoint}. Now assume $\omega$ is not good at $x_0$. We claim $\omega$ is small at least one of $\{x_1,x_5\}$. We will show if $\omega$ is big at $x_1$, then $\tau_2$ has multiple traces in $\cS_0$. Indeed, in this case Theorem~\ref{thm:4 wheel} implies that there is a unique vertex of type $\hat c$ which is adjacent to both $x_0$ and $x_2$, and this implies $[\tau_1]$ is the only homotopy class of simple arcs from $p_c$ to $z_c$ in $\cS_0$ that has a representative which is disjoint from all traces of $\tau_2$ in $\cS_0$. This is only possible when $\tau_2$ has multiple traces as all traces of $\tau_2$ in $\cS_0$ are good by Corollary~\ref{cor:boundary parallel 24} and our assumption that $\omega$ is big at $x_0$. Similarly, if $\omega$ is big at $x_5$, then $\tau_4$ has multiple traces in $\cS_0$. Thus the claim follows.

\smallskip
\noindent
\underline{Case 1: $\omega$ is small at both $x_1$ and $x_5$.} If $\omega$ is small at $x_3$, then by the argument in the last paragraph of the proof of Corollary~\ref{cor:goodpoint}, $\omega$ has property $(*)$. Now assume $\omega$ is big at $x_3$. As $\omega$ is not good at $x_0$, we know $\tau_2$ and $\tau_4$ have single trace in $\cS_0$. We claim that $\tau_0\cap\tau_3$ has at most one interior intersection point, and in the case when $\tau_0\cap \tau_3\neq\emptyset$, assumption (3) of Lemma~\ref{lem:b} is satisfied. To see this, we scissor $\cS$ along $\tau_2$ to obtain $\cS_2$. As $\omega$ is big at $x_3$, by the argument in the previous paragraph, $[\tau_3]$ is the only homotopy class of simple arcs in $\cS_2$ from $p_c$ to $z_c$ which has a representative avoiding all traces of $\tau_4$ in $\cS_2$. Let $K$ be the component of $\cS_2\setminus \tau_4$ containing $p_c$. Then $K$ must be a punctured disk. As $\tau_0\cap(\tau_2\cup\tau_4)=\emptyset$ except at endpoints and $\tau_3\cap(\tau_2\cup\tau_4)=\emptyset$, we know $\tau_0$ is a simple arc in $K\cup \partial K$ from one of $\{z,z'\}$ in $\partial K$ to $p_a\in \partial K$, and $\tau_3$ is a simple arc in $K$ from $p_c$ to $z_c\in \partial K$. As $K$ is a punctured disk, we know that either $\tau_0\cap\tau_3=\emptyset$; or $\tau_0\cap\tau_3$ has exactly one point. In the latter case, a subarc of $\tau_3$ containing one of $\{z,z'\}$, a subarc of $\tau_0$ containing $z_c$, and one of $\overline{zz_c}$ or $\overline{z'z_c}$ bound a disk in $K\cup \partial K$. Thus the claim is proved.

If $\tau_0\cap\tau_3=\emptyset$, then $\omega$ readily satisfies property $(*)$. If $\tau_0\cap\tau_3$ is exactly one point, then applying Lemma~\ref{lem:b} twice, to $\{\tau_0,\tau_5,\tau_4,\tau_3\}$, and to $\{\tau_0,\tau_1,\tau_2,\tau_3\}$, we find a vertex $z$ of type $\hat b$ adjacent to each of $\{x_0,x_5,x_4,x_3\}$, and a vertex $z'$ of type $\hat b$ adjacent to each of $\{x_0,x_1,x_2,x_2\}$. Note that $z=z'$, otherwise Lemma~\ref{lem:special4cycle} applying to the 4-cycle $x_0zx_3z'$ implies that $x_0$ and $x_3$ are adjacent, contradicting $\tau_0\cap\tau_3\neq\emptyset$. Thus $\omega$ has property $(*)$.

The above argument implies the following claim. The ``more generally'' part follows from the first part of the claim, by possibly applying a symmetry of the Dynkin diagram.
\begin{claim}
Suppose $\omega$ has type I. If $\omega$ is small at $x_1,x_5$ and big at $x_0$, then $\omega$ has property $(*)$. More generally, if there are three consecutive vertices
$\{x_{i-1},x_i,x_{i+1}\}$ of $\omega$ such that $\omega$ is small at $x_{i-1},x_{i+1}$ and $\omega$ is big at $x_i$, then $\omega$ has property $(*)$.
\end{claim}

\smallskip
\noindent
\underline{Case 2: $\omega$ is big at exactly one of $\{x_1,x_5\}$.} As $\omega$ is not good at $x_0$, up to symmetry, we can assume $\omega$ is big at $x_1$ and small at $x_5$, $\tau_2$ has multiple traces in $\cS_0$ and $\tau_4$ has single trace in $\cS_0$. Moreover, up to replacing $x_5$ by $x'_5$ and $\omega$ by $\omega'=x_0x_1x_2x_3x_4x'_5$, we can assume $\omega'$ is small at both $x_5$ and $x_0$, though still big at $x_1$. We also assume $\omega'$ is embedded, otherwise $\omega$ has property $(*)$. If $\omega'$ is small at $x_2$. Then we can use the claim above to $x_0$, $x_1$ and $x_2$, to see that $\omega'$ has property $(*)$, hence $\omega$ satisfies property $(*)$ by Lemma~\ref{lem:6cycle replace} below.

Now we assume $\omega'$ is big at $x_2$. If $\omega'$ is good at $x_2$, then $\omega'$ satisfies property $(*)$ by Corollary~\ref{cor:goodpoint}, which implies $\omega$ has property $(*)$ by the same argument as before.
If $\omega'$ is not good at $x_2$, as $\omega'$ is already big at $x_1$, we must have $\omega'$ small at $x_3$, $\tau_4$ has single trace in $\cS_2$, and $\tau_0$ has multiple traces in $\cS_2$. Moreover, we can replace $x_3$ by $x'_3$ and $\omega'$ by $\omega''=x_0x_1x_2x'_3x_4x'_5$ so that $\omega''$ is small at $x'_3,x_2$. By previous discussion, $\omega''$ is small at $x_0$ and big at $x_1$. Now we apply the claim at the end of Case 1 to $\omega''$ at $x_0,x_1$ and $x_2$, implying $\omega''$ has property $(*)$. By applying Lemma~\ref{lem:6cycle replace} twice, we know $\omega'$ and $\omega$ both has property $(*)$. 
\end{proof}

\begin{lem}
	\label{lem:6cycle replace}
Suppose $\omega$ is a 6-cycle of type I. Let $x'_5$ be a vertex of type $\hat c$ that is adjacent to both $x_0$ and $x_4$, and let $\omega'$ be an embedded 6-cycle obtained from $\omega$ by replacing $x_5$ by $x'_5$. If $\omega'$ has property $(*)$, then $\omega$ has property $(*)$.
\end{lem}

\begin{proof}
There are two cases to consider.  If there is a vertex $z$ of type $\hat b$ adjacent to each of $\{x_1,x_3,x'_5\}$, then by applying Lemma~\ref{lem:special4cycle} to $zx'_5x_0x_1$, $zx_1x_2x_3$ and $zx_3x_4x'_5$, we know $z$ is adjacent to $\{x_0,x_2,x_4\}$. By applying Lemma~\ref{lem:special4cycle} to $zx_0x_5x_4$, $zx_4x_3x_2$ and $zx_2x_1x_0$, we know $z$ is adjacent to $\{x_1,x_3,x_5\}$, implying $\omega$ has property $(*)$. 

 Suppose there is a vertex $z$ of type $\hat a$ adjacent to each of $\{x_1,x_3,x'_5\}$. First we suppose $z\notin \{x_0,x_2,x_4\}$. Then by applying Theorem~\ref{thm:4 wheel} to $zx_1x_0x'_5$, $zx'_5x_4x_3$ and $zx_3x_2x_1$, we know for $i=0,2,4$, there is a vertex $y_i$ of type $\hat d$ such that $y_0$ is adjacent to $x_1$ and $x'_5$, $y_2$ is adjacent to $x_1$ and $x_3$, and $y_4$ is adjacent to $x_3$ and $x'_5$. Applying Theorem~\ref{thm:weakflagD} to the 6-cycle $x'_5y_0x_1y_2x_3y_4$ in $\Delta_{\Lambda,\{c,d,b\}}$, we know there is a vertex $z'$ of type $\hat b$ adjacent to each of $\{x_1,x_3,x_5\}$, reducing to the previous case. It remains to consider $z\in \{x_0,x_2,x_4\}$. If $z=x_0$ or $x_4$, then $\omega$ readily has property $(*)$. If $z=x_2$, then $x'_5$ is adjacent to each of $x_0,x_2,x_4$. We assume $x'_5\notin \{x_1,x_3,x_5\}$, otherwise $\omega$ readily has property $(*)$. Then we use the same argument as the last paragraph of the proof of Corollary~\ref{cor:goodpoint} to deduce that $\omega$ has property $(*)$.
\end{proof}

\section{Six-cycles of type II in the $D_4$ complex}
\label{sec:abccycle}
Throughout this section, $\Lambda$ is the Dynkin diagram of type $D_4$ with its vertex set $\{a,b,c,d\}$ such that $\{a,b,c\}$ are leaf vertices.
The goal of this section is to prove the following.
\begin{prop}
\label{prop:type II}
Suppose $\omega$ is an embedded 6-cycle in $\Delta_\Lambda$ of type II. Suppose $x_0$ and $x_3$ are not adjacent in $\Delta_\Lambda$. Then either $x_1$ and $x_5$ are adjacent to common vertex of type $\hat b$, or $\omega$ has property $(*)$.
\end{prop}
\begin{proof}
We consider the trace of $\tau_2$ in $\cS_0$. If $\tau_2$ has a $\pi_0$-trace which is not good, then we are done by Corollary~\ref{cor:tau2 has a pi0 trace}. If $\tau_2$ has a $\pi_0$-trace and all $\pi_0$-traces of $\tau_2$ are good, then we are done by Corollary~\ref{cor:tau2 has a pi0 trace} and Lemma~\ref{lem:20intersect}. It remains to consider the case that $\tau_2$ does not have any $\pi_0$-traces, i.e. $\tau_2$ has unique trace on $\cS_0$. If all traces of $\tau_4$ in $\cS_0$ are good and $\tau_4$ has at least two traces, then proposition follows from Lemma~\ref{lem:20e40ne} below; if $\tau_4$ only has a single trace in $\cS_0$ and it is a good trace, then the proposition follows from Lemma~\ref{lem:24 unique} below; if $\tau_4$ has a trace in $\cS_0$ which is not good, then the proposition follows from Lemma~\ref{lem:4notgood} below.
This finishes the proof.
\end{proof}

\begin{cor}
	\label{cor:type II}
Suppose $\omega$ is a 6-cycle in $\Delta_\Lambda$ of type II. Then $\omega$ has property $(*)$.
\end{cor}

\begin{proof}
If $x_0$ and $x_3$ are adjacent in $\Delta_\Lambda$, then $\omega$ has property $(*)$. Now we assume $x_0$ and $x_3$ are not adjacent.	
By Proposition~\ref{prop:type II} and symmetry, we know either $x_3$ and $x_5$ are adjacent to common vertex of type $\hat b$, or $\omega$ has property $(*)$. This together with Proposition~\ref{prop:type II} imply that either $\omega$ has property $(*)$, or $x_5,x_3$ are adjacent to a common vertex $x'_4$ of type $\hat b$ and $x_5,x_1$ are adjacent to a common vertex $x'_0$ of type $\hat b$. In the latter case, by applying Proposition~\ref{prop:typeI} to the 6-cycle $x'_0x_1x_2x_3x'_4x_5$ and noting the symmetry of the Dynkin diagram, we know there is a vertex of type $\hat b$ or $\hat a$ that is adjacent to each of $\{x_1,x_3,x_5\}$, and property $(*)$ follows.
\end{proof}
In the rest of this section, we prove Corollary~\ref{cor:tau2 has a pi0 trace}, Lemma~\ref{lem:20intersect}, Lemma~\ref{lem:20e40ne},  Lemma~\ref{lem:24 unique} and Lemma~\ref{lem:4notgood}. We start with several preparatory lemmas for Corollary~\ref{cor:tau2 has a pi0 trace}.

\begin{lem}
	\label{lem:S3'}
	Suppose $x_0,x_1,x_2,x_3$ are of type $\hat a, \hat c, \hat b,\hat c$ respectively in $\Delta$ such that $x_i$ and $x_{i+1}$ are adjacent in $\Delta$. Let $\{\tau_i\}_{i=0}^3$ be the associated arcs in $\cS$. We define $\cS_0$ as before.
	Suppose $\tau_0\cap\tau_3\neq\emptyset$.  Assume $\tau_2$ has at least two $\pi_0$ traces and all $\pi_0$ traces of $\tau_2$ are good. Then the conclusion of Lemma~\ref{lem:S3} holds true.
\end{lem}

\begin{proof}
	As in $\cS$ the arc $\tau_2$ travels from $p_c$ to $p_a$, we know that in $\cS_0$, $\tau_2$ has exactly one trace $\tau'_2$ which goes from $p_c$ to a point in $\pi_0$. All other traces of $\tau_2$ are $\pi_0$-traces, which are good traces by assumption. There are at most three parallel classes of good traces of $\tau_2$. Given a trace $\tau'_3$ of $\tau_3$ with one endpoint being $x\in \pi_0$, we define \emph{the trace after $(\tau'_3,x)$} to be the trace of $\tau_3$ containing $x'$, where $x'$ and $x$ are identified via $\phi:\overline{zp_a}\to\overline{z'p_a}$. Let $R_2$ be the trace of $\tau_2$ containing $p_c$ and denote the other endpoint of $R_2$ by $\theta$.
	\begin{figure}[h]
		\centering
		\includegraphics[scale=0.8]{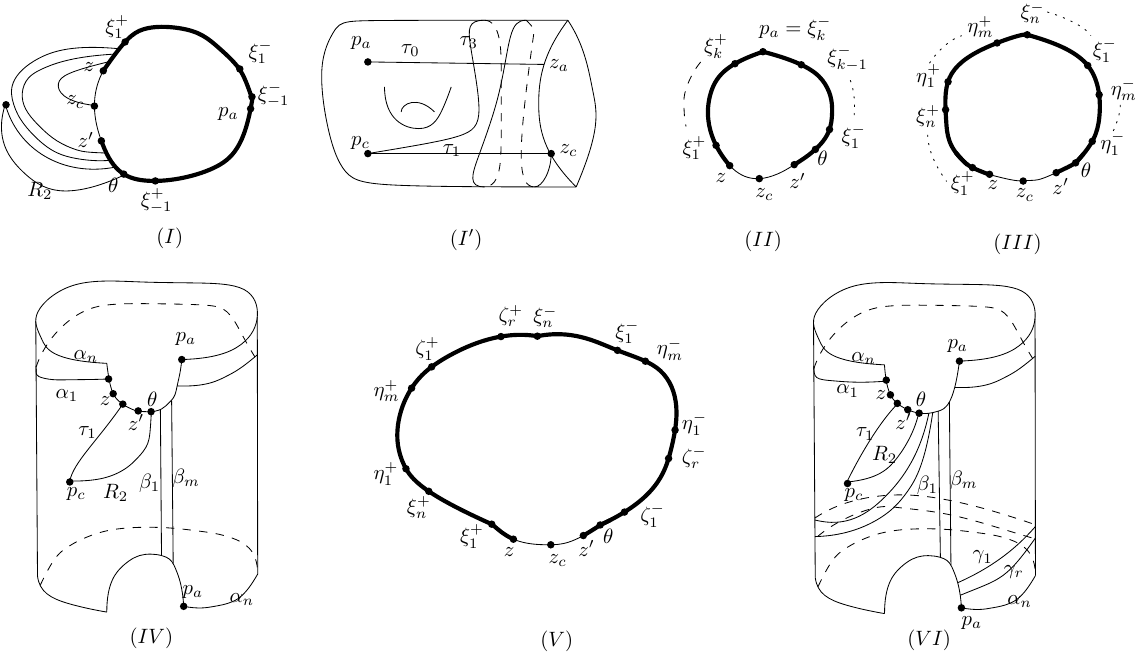}
		\caption{Proof of Lemma~\ref{lem:S3'}}
		\label{fig:12'}
	\end{figure}
	
	Suppose we are in Case 0 of the proof of Lemma~\ref{lem:S3}, i.e. $p_c$ is squeezed by two parallel $\pi_0$-traces of $\tau_2$. We use the same notation as in Lemma~\ref{lem:S3}. The difference with Lemma~\ref{lem:S3} is that $z=\xi^+_1$ is no longer true.
	Instead, $\xi^+_1$ and $\theta$ are identified via $\phi$. Moreover, $R_2$ is squeezed by $\alpha_1$ and $\alpha_{-1}$. If $R_3$ ends in $\overline{\xi^-_1\xi^-_{-1}}$ or $\overline{\theta\xi^+_{-1}}$, then we argue in the same way as in Lemma~\ref{lem:S3}. 
	
	Now we assume $R_3$ ends in $\overline{z'\theta}$ at point $x$, then $x$ is identified with $x'\in \overline{z\xi^+_1}$ via $\phi$. 
	Let $R'_3$ be the trace after $(R_3,x)$. If $R'_3$ ends in a point  $x''\in\overline{\xi^-_1\xi^-_{-1}}$ or $\overline{\theta\xi^+_{-1}}$, then we apply the argument in Lemma~\ref{lem:S3} to the trace after $(R'_3,x'')$. 
	We now show $x''\notin \overline{z\xi^+_1}$. If $x''\in \overline{z\xi^+_1}$, then Lemma~\ref{lem:boundary parallel} implies that $R'_3$ is not boundary parallel in $\cS_0$. As $\tau_3\cap\tau_2=\emptyset$, $R'_3$ is squeezed by $\alpha_1$ and $\alpha_{-1}$, thus $R'_3$ and a subarc of $\overline{z\xi^+_1}$ bound a disk with $p_c$ inside, which contradicts that $R'_3\cap R_3=\emptyset$.
	
	The only possibility left for $x''$ is $x''\in \overline{z'\xi^+_{-1}}$ or $x''=z_c$. We will show this is impossible. See Figure~\ref{fig:12'} (I). As $R'_3\cap R_3=\emptyset$, we know  $x''$ is closer to $z_c$ than $x$. 
	Moreover, $R'_3$ and a subarc of $\overline{\xi^+_{-1}\xi^+_1}$ bound a disk $D$ with $p_c\notin D$. Then $x''$ is identified via $\phi$ to $x'''\in\overline{z\xi^+_1}$ such that $x'''$ is closer to $z_c$ than $x'$. Let $R''_3$ be the trace after $(R'_3,x'')$. As $R''_3\cap R'_3=\emptyset$, we know $R''_3\subset D$, and $R''_3$ ends in a point $x''''$ in $\overline{z'\theta}$ which is closer to $z_c$ than $x''$. Repeating this process, we know $\tau_3$ is made of sequence of traces with endpoints closer and closer to $z_c$, until it eventually ends in $z_c$. This implies that in $\cS$, $\tau_3$ can be obtained from $\tau_1$ by applying a non-zero power of the Dehn twist along $\partial \cS$, see Figure~\ref{fig:12'} (I'). As this Dehn twist corresponds to the Garside element in $A_\Lambda$, we have a contradiction with that $x_1$ and $x_3$ are adjacent to a common vertex of type $\hat b_2$ by Lemma~\ref{lem:common adjacent}.
	
	It remains to consider that $R_3$ ends in $\overline{z'\xi^+_1}$ at point $x$. As $R'_3\cap (R_3\cup R_2)=\emptyset$, we know $R'_3$ is contained in the disk bounded by $R_3$, $R_2$ and a subarc of $\overline{\theta\xi^+_1}$. By a similar argument as in the previous paragraph, we deduce that $\tau_3$ and $\tau_1$ differ by a non-zero power of Dehn twist along $\partial\cS$. Thus $R_3$ can not end in $\overline{z'\theta}$.
	
If $p_c$ is not squeezed by two parallel traces of $\tau_3$ and $\tau_3$ only has one parallel class of $\pi_0$-traces, then we use a similar argument as in Case 1 of the proof of Lemma~\ref{lem:S3}, combined with the argument in the previous paragraphs. See Figure~\ref{fig:12'} (II) for an adjusted picture.

Suppose $p_c$ is not squeezed by two parallel traces of $\tau_3$ and $\tau_3$ has two parallel classes of $\pi_0$-traces. We will use the same notation as in Case 2 of Lemma~\ref{lem:S3}, and see Figure~\ref{fig:12'} (III) and (IV) for the adjusted pictures. Note that $z\neq \xi^+_1$, and we modify the definition of $\gamma_1$  (in the proof of Lemma~\ref{lem:S3}) accordingly so it goes from $\eta^-_1$ to $\xi^+_1$ (via $\theta,z',z_c$ and $z$). We still have $n>1$.
Note that $R_3$ ends on an interior point $x$ of either $\overline{\xi^+_1z}$, or $\overline{z'\theta}$, or $\overline{\theta\eta^-_1}$, or $\gamma_2$, or $\gamma_3$, or $\gamma_4$. If either $x\in \gamma_1\cup\gamma_3\cup\gamma_4$ or $n-1\neq m$, then we can either argue as in the previous paragraphs, or as in Case 2 of Lemma~\ref{lem:S3}. Now we assume $x\in \gamma_2$ and $n-1=m$. Then $\phi$ identifies $x$ with $x'\in \gamma_4$. Let $T$ be the trace of $\tau_3$ containing $x'$. Then $T\subset D\cup \partial D$, where $D$ is defined as in Case 2 of Lemma~\ref{lem:S3}. The possibilities of the other endpoint $x''$ of $T$ are (0) $x''\in \overline{z'\theta}$; (1) $x''=z_c$ or $x''\in \overline{\xi^+_1z}$; (2) $x''\in \overline{\theta\eta^-_1}$; (3) $x''\in \gamma_4$; (4) $x''\in \gamma_3$; (5) $x''\in \gamma_2$. These are similar to Case 2 of Lemma~\ref{lem:S3} except (0) and (5). If (5) happens, then $T\cap \tau_1=\emptyset$ as $T\cap R_2=\emptyset$, as desired. Suppose (0) happens. Then $T$ and $\alpha_1$ are parallel. We define an auxiliary arc, $S$, which is an arc in $D\cup \partial D$ from $x'$ to $z_c$ that is disjoint from $R_2$. Let $T'$ be the trace after $(T,x'')$. Then $T'$ is squeezed by $T$ and $\alpha_1$. Thus either $T'$ goes from a point in $\overline{\xi^+_1z}$ to a point in $\gamma_4$, in which case $T'\cap \tau_1=\emptyset$ and we are done; or $T'$ goes from a point in $\overline{\xi^+_1z}$ to either $z_c$ or a point in $\overline{z'x''}$, in which case the argument before implies that $\tau_3$ (viewed as an arc in $\cS$) is obtained from $\tau'_3$ by applying a nonzero power of Dehn twists along $\partial \cS$, where $\tau'_3$ is the concatenation of $R_3$ and $S$. Now we show the latter case is impossible. As $\tau_1\cup\tau'_3$ gives a homotopically non-trivial simple closed curve in $\cS'$ which is not homotopic to $\partial \cS'$ (recall that $\cS'$ is obtained from $\cS$ by filling back all the punctures), we know there are infinitely many elements in $\Omega_a$ which have representatives that are disjoint from both $\tau_1$ and $\tau'_3$. Let $x'_3$ be the vertex of type $\hat a$ in $\Delta_{\Lambda}$ associated with $[\tau'_3]$. Then there are infinitely many different vertices of type $\hat a$ in $\Delta_\Lambda$ that are adjacent to both $x_1$ and $x'_3$. By Lemma~\ref{lem:common adjacent}, we know $x_1$ and $x_3$ are adjacent to a common vertex of type $\hat b$ if and only if they are adjacent to common vertex of type $\hat a$. However, by definition of $\tau_1$ and $\tau_3$, there does not exist an arc in $\cS$ from $z_a$ to $p_a$ avoiding both $\tau_1$ and $\tau_3$, contradiction.

The case $p_c$ is not squeezed by two parallel traces of $\tau_3$ and $\tau_3$ has three parallel classes of $\pi_0$-traces is similar to  Lemma~\ref{lem:S3}, see Figure~\ref{fig:12'} (V) and (VI) for adjustments.
\end{proof}

\begin{lem}
	\label{lem:not big 24'}
	Given a cycle $\omega$ of type II. Let $\{\tau_i\}_{i=0}^5$ be the associated arcs in $\cS$. We define $\cS_0$ as before. Suppose $\tau_0\cap\tau_3\neq\emptyset$. Suppose at least one of the following holds:
	\begin{enumerate}
		\item $\tau_4$ has a trace in $\cS_0$ which is not good, and $\tau_2$ has at least one $\pi_0$ trace in $\cS_0$;
		\item $\tau_2$ has at least one $\pi_0$ trace, and $\tau_2$ has a $\pi_0$ trace which is not good.
	\end{enumerate} 
	Then $\omega$ is small at $x_0$.
\end{lem}

\begin{proof}
	We start with Assertion 1. The case when all $\pi_0$-traces of $\tau_2$ are good and there are at least two of them is identical to the first paragraph of the proof of Corollary~\ref{cor:boundary parallel 24}, except we use Lemma~\ref{lem:S3'} to obtain $S_3$ instead of Lemma~\ref{lem:S3}. The case that $\tau_2$ has exactly one $\pi_0$-trace $\tau'_2$ which is good, is similar to the argument in Corollary~\ref{cor:boundary parallel 24}, except we will work with $\tau'_2$ instead of $\tau_2$, and $\tau'_2$ starts in a point in the interior of either $\overline{zp_a}$ or $\overline{z'p_a}$, and ends at $p_a$ (in particular one of $\{z,z'\}$ is in $\tau'_2$ is no longer true, but this is harmless). The case $\tau_2$ has a $\pi_0$-trace which is not good is identical to the last paragraph of the proof of Corollary~\ref{cor:boundary parallel 24}. 
	
	For Assertion 2, up to a symmetry, we can assume $x_4$ has type $\hat b$ and $x_2$ has type $\hat c$ (types of other vertices remain unchanged), and in Assertion 2 we assume $\tau_4$ has a $\pi_0$-trace which is not good. Again if $\tau_2$ has a trace which is not good, or all traces of $\tau_2$ are good and $\tau_2$ has at least two traces, then we can argue in the same way as in Corollary~\ref{cor:boundary parallel 24}. Suppose $\tau_2$ has exactly one trace which is good. Let $S_4$ be a $\pi_0$-trace of $\tau_4$ which is not good. Up to exchange $z$ and $z'$, we can assume one endpoint of $S_4$ is contained in the interior $\overline{zp_a}$, and another endpoint is contained in $\overline{z'p_a}\setminus\{z'\}$. Indeed, if $\partial S_4$ are contained in the interior of one of $\overline{zp_a}$ or $\overline{z'p_a}$, then Lemma~\ref{lem:boundary parallel} implies that $S_4$ and an arc $\gamma\subset \partial\cS_0$ with $z_c\in \gamma$ bound a punctured disk $D$, and we can find the desired trace of $\tau_4$ in $D$. Given such choice of $S_4$, the argument in the proof of Corollary~\ref{cor:boundary parallel 24} goes through.
\end{proof}

\begin{lem}
	\label{lem:2142}
	Let $\omega$ be a cycle of type II with its consecutive vertices being $\{x_i\}_{i=0}^5$. Let $\{\tau_i\}_{i=0}^5$ be the associated arcs in $\cS$. We define $\cS_0$ as before.
	Suppose $\tau_0\cap\tau_3\neq\emptyset$.  Assume $\tau_2$ has exactly one $\pi_0$ trace which is a good trace. Assume all traces of $\tau_4$ are good, and $\tau_4$ has at least two traces. Then $x_1$ and $x_5$ are adjacent to a common vertex of type $\hat b$.
\end{lem}

\begin{proof}
By	Lemma~\ref{lem:S3}, there is a trace $S_3$ of $\tau_3$ with  $\partial S_3\subset \partial \cS_0$ such that $\tau_5\cap S_3\subset \{z_c\}$ and $S_3$ is not boundary parallel in $\cS'_0$. Let $\tau'_2$ be the trace of $\tau_2$ containing $p_c$, and let $\tau''_2$ be the good $\pi_0$ trace of $\tau_2$.
	Assume without loss of generality that $\tau'_2$ goes from $p_c$ to $x\in \overline{zp_a}$, and $\tau''_2$ goes from $x'\in \overline{z'p_a}$ to $p_a$, where $x$ and $x'$ are identified via $\phi:\overline{zp_a}\to\overline{z'p_a}$. 
	
	If $\tau''_2$ is not parallel to $S_3$, then up to a homeomorphism, we can assume $\tau''_2$ and $S_3$ give a pair of generators of the fundamental group of the torus obtained by smashing $\partial\cS_0$ to a point and fill in back $p_c$. As $\tau_5\cap S_3\subset\{z_c\}$ and $\tau_1\cap\tau''_2=\emptyset$, we can always finds an arc $\tau$ in $\cS_0$ from $p_c$ to $p_a$ such that $\tau\cap\tau_i=\emptyset$ for $i=1,5$ except at endpoints, moreover, the concatenation $\tau,\tau_0,\tau_i$ gives a homotopically non-trivial simple loop in $\cS'$ for $i=1,5$. See Figure~\ref{fig:15} (I) (we allow $z_c\in S_3$). The dashed arc indicates the choice of $\tau$ - it can end different points in $\cS_0$.  Thus $x_1$ and $x_5$ are adjacent to a common vertex of type $\hat b$.
	
	Suppose $\tau''_2$ is parallel to $S_3$. If $p_c$ is not squeezed by $\tau''_2$ and $S_3$, then we are in (2) by the argument in Lemma~\ref{lem:nb} (with $\tau_2$ and $\tau_4$ exchanged), there is a vertex $x'_0$ of type $\hat d$ adjacent to each of $x_1$ and $x_3$. Now the lemma follows by considering any vertex $x''_0$ of type $\hat b$ in $\Delta_\Lambda$ adjacent to $x'_0$ and applying Lemma~\ref{lem:link} (3) at $\lk(x'_0,\Delta_\Lambda)$. 
	So we assume $p_c$ is squeezed by $\tau''_2$ and $S_3$. Then $\tau'_2$ is also squeezed by $\tau''_2$ and $S_3$.	If $z_c$ is not squeezed by $\tau''_2$ and $S_3$ and $z_c\notin S_3$, then $\tau_1\cap (\tau'_2\cup\tau''_2)=\emptyset$ implies that there are two possibilities of $\tau_1$ (see the two dashed arcs in Figure~\ref{fig:15} (II)). In each case a subarc of $\tau_1$, a subarc of $S_3$ and a subarc of $\partial \cS_0$ bound a disk. As $\tau_5\cap S_3=\emptyset$, $\tau_1$ and $\tau_5$ satisfy Lemma~\ref{lem:not big} (2). Thus $\omega$ is small at $x_0$ and the lemma follows.


	If $z_c$ is squeezed by $\tau''_2$ and $S_3$, then either $\partial S_3\subset \overline{zp_a}$ or $\partial S_3\subset \overline{z'p_a}$. Let $R_3$ be as in Lemma~\ref{lem:S3} (2). If $\partial S_3\subset \overline{zp_a}$, then $R_3$ ends in $\overline{z'p_a}$. As $\tau_5\cap(R_3\cup S_3)=\emptyset$, we know $R_3$ and $S_3$ determine $\tau_5$ up to two choices. We refer to Figure~\ref{fig:15} (III), the four gray curves are $\tau''_2,\tau'_2,S_3$ and $R_3$. The arc labeled $5$, and the dotted arc, are the two possibilities of $\tau_5$; and the dashed arc labeled $1$, and the dotted arc, are the two possibilities of $\tau_1$. If one of $\tau_1$ and $\tau_5$ is the dotted arc, then $\omega$ is small at $x_0$ by Lemma~\ref{lem:not big} (2). If $\tau_i$ is the arc labeled by $i$ for $i=1,5$, then we define $\tau$ to be thickened arc in Figure~\ref{fig:15} (III) from $p_c$ to $p_a$. The concatenation of $\tau_5,\tau_0,\tau$, and the concatenation of $\tau_1,\tau_0,\tau$ descend to two homotopically non-trivial simple loops on $\cS'$. Thus $x_1$ and $x_5$ are adjacent to a common vertex of type $\hat b$. The case $\partial S_3\subset \overline{z'p_a}$ is similar, see Figure~\ref{fig:15} (IV).
	\begin{figure}[h]
		\centering
		\includegraphics[scale=0.88]{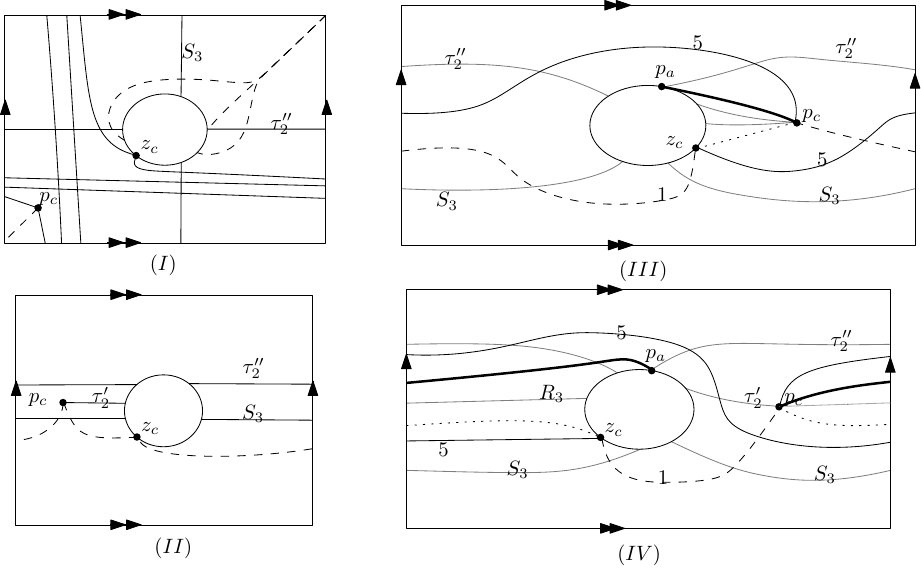}
		\caption{Proof of Lemma~\ref{lem:2142}}
		\label{fig:15}
	\end{figure} 
	
	The remaining case is $z_c\in S_3$. Figure~\ref{fig:15} stills apply to this case, except that we need to move $z_c$ so it coincidents with one endpoint of $S_3$. The possibilities of $\tau_1$ are as before. To analyze the possibilities of $\tau_5$, we scissor $\cS_0$ along $S_3$ to obtained a punctured annulus $A'_0$ with boundary components $C^+$ and $C^-$. Assume without loss of generality that $p_a\in C^-$. Then $R_3$ ends in $C^-$ and $\partial \tau''_2\subset C^-$. Note that $z_c$ gives $z^+_c\in C^+$ and $z^-_c\in C^-$. Thus the trajectory of $\tau_5$ in $A'_0$ either goes from $p_c$ to $z^+_c$ or from $p_c$ to $z^-_c$. In the former case, there are infinite many possible homotopy classes of $\tau_5$, however, in each case $\tau_1$ and $\tau_5$ satisfy Lemma~\ref{lem:not big} (2); in the latter case, there are only two possible homotopy classes of $\tau_5$ as $\tau_5\cap R_3=\emptyset$, and we are reduced to the previous paragraph.
\end{proof}

\begin{cor}
	\label{cor:tau2 has a pi0 trace}
	Given a cycle $\omega$ of type II in the Artin complex $\Delta_\Lambda$. Suppose $\tau_0\cap\tau_3\neq\emptyset$ and $\tau_2$ has at least one $\pi_0$-trace. Suppose one of the follow holds:
	\begin{enumerate}
		\item one of $\tau_2$ or $\tau_4$ has a $\pi_0$-trace which is not good;
		\item all $\pi_0$-traces of $\tau_2$ and $\tau_4$ are good, and one of $\tau_2$ and $\tau_4$ has at least two traces.
	\end{enumerate}
	 Then $x_1$ and $x_5$ are adjacent to a common vertex of type $\hat b$.
\end{cor}

\begin{proof}
The first part of the corollary follows from Lemma~\ref{lem:not big 24'}. For the second part, if $\tau_2$ has exactly one $\pi_0$-trace, then the corollary follows from Lemma~\ref{lem:2142}. Now assume $\tau_2$ has more than one $\pi_0$-traces. By Lemma~\ref{lem:S3'}, we are in a position to apply Corollary~\ref{cor:S4}. If we are in Corollary~\ref{cor:S4} (1), then $x_1$ and $x_5$ are adjacent to a common vertex of type $\hat b$ by the argument in the second paragraph of the proof of Lemma~\ref{lem:2142}. If we are in Corollary~\ref{cor:S4} (2), then $\omega$ is small at $x_0$ by Lemma~\ref{lem:nb} and Lemma~\ref{lem:S3'}. Corollary~\ref{cor:S4} (3) is ruled out by assumption.
\end{proof}

\begin{lem}
	\label{lem:25}
Suppose $\omega$ is of type II, and $\tau_0\cap\tau_3\neq\emptyset$. Suppose $\tau_2\cap \tau_5=\emptyset$. Then one of the following holds:
\begin{enumerate}
	\item $x_1$ and $x_5$ are adjacent to a common vertex of type $\hat b$.
	\item $\omega$ satisfies property $(*)$.
\end{enumerate}
\end{lem}

\begin{proof}
	We claim that up to replacing $x_4$ by $x'_4$ of type $\hat a$ adjacent to both $x_3$ and $x_5$, we can assume $\tau_2\cap \tau'_4=\emptyset$ except at endpoints, where $\tau'_4$ is an arc corresponding to $x'_4$. To see this, we consider $\cS_5$, which is obtained from $\cS$ by scissoring along $\tau_5$. If $\tau_3$ has a trace $\tau'_3$ on $\cS_5$ which is not good, then $\tau'_3$ and a part of $\partial \cS_5$ bound a disk $D'$ with $\tau_4\subset D'$. As $\tau_2\cap (\tau_3\cup \tau_5)=\emptyset$ except at endpoints, $\tau_2$ is trapped in $D'$. Thus $\tau_4\cap \tau_2=\emptyset$ except at endpoints. Now we assume all traces of $\tau_3$ are good in $\cS_5$. Then the component $K$ of $\cS_5\setminus \tau_3$ that contains $p_a$ is either a punctured disk or a punctured annulus. Note that $(\tau_4\cup\tau_2)\subset K$ except at endpoints. If $K$ is punctured disk, then $\tau_4\cap \tau_2=\emptyset$ (except at endpoints), as they have the same starting point $p_a$ but end at different points in $\partial K$. If $K$ is a punctured annulus $A_0$, then we can replace $\tau_4$ by an arc $\tau'_4$ from $p_a$ to $p_c$ so that $\tau_2\cap \tau_4=\emptyset$ except at endpoints.
	
	Similarly, up to replacing $x_0$ by $x'_0$ of type $\hat a$ adjacent to both $x_1$ and $x_5$, we can assume $\tau_2\cap \tau'_0=\emptyset$ except at endpoints, where $\tau'_0$ is an arc corresponding to $x'_0$. 	
	
	Let $\bar\Delta'_0$ be as Definition~\ref{def:complex 3punctured}. Note that $\bar\tau_1,\bar \tau_2,\bar\tau'_0$ gives a 3-cycle in $\bar\Delta'_0$, which we denote by $120'$. Similarly we define the 3-cycle $520'$ in $\bar\Delta'_0$. Recall that $120'$ is degenerate, if it does not bound a 2-face. If both $120'$ and $520'$ are non-degenerate, then the $\bar\tau_2$ vertex in $\bar\Delta'_0$ lifts under the covering $\Delta_0\to\bar\Delta_0\cong \bar\Delta'_0$ to a vertex of type $\hat b$ in $\Delta_0\subset\Delta_\Lambda$ which is adjacent to both $x_1$ and $x_5$. Suppose exactly one of $120'$ and $520'$, say $120'$, is a non-degenerate triangle. Then up to a homeomorphism, we can assume the arcs $\tau_1,\tau_2,\tau'_0$ in $\cS$ are as in Figure~\ref{fig:1250} (I). As $\bar \tau_2,\bar\tau'_0,\bar\tau_5$ bound a disk in $\bar\cS$, we know $\bar \tau_2$ and $\bar \tau_0$ determines the homotopy class of $\bar\tau_5$ (rel endpoints) up to two choices. Each choice of $\bar \tau_5$ determines $\tau_5$ up to powers of Dehn twists along $\partial \cS_0$, however, the extra constraint $\tau'_0\cap\tau_5=\emptyset$ implies that each choice of $\bar\tau_5$ determine a choice of $\tau_5$. Thus $\tau_2$ and $\tau'_0$ determine $\tau_5$ up to two choices, as in Figure~\ref{fig:1250} (I). After scissoring $\cS$ along $\tau_0$, we see that $\tau_1$ and $\tau_5$ satisfy Lemma~\ref{lem:not big} (2). Hence $x_1$ and $x_5$ are adjacent to a common vertex of type $\hat b$. It remains to consider the case that both $120$ and $520$ are degenerate.

	Now we can look at $3$-cycles $4'25$ and $4'32$ in $\bar\Delta'_0$, defined in the same fashion as before. If both $3$-cycles are non-degenerate, then by considering the lift of the 2-face in $\bar\Delta'_0$ filling $4'32$ under the covering map $\Delta_0\to \bar\Delta'_0$, we know $x_2$ and $x_4$ are adjacent in $\Delta_\Lambda$. Similarly, $x_2$ and $x_5$ are adjacent. Thus $x_0,x_1,x_2,x_5$ form a 4-cycle in $\Delta_\Lambda$. By Lemma~\ref{lem:special4cycle}, $x_2$ is adjacent to $x_0$. Thus $x_2$ is adjacent to each of $\{x_1,x_3,x_5\}$ and we are in Lemma~\ref{lem:25} (2).
    If exactly one of $4'25$ and $4'32$ is non-degenerate, then exactly one of the 3-cycles $520',120',4'25,4'32$ is non-degenerate. By Lemma~\ref{lem:generator}, each of these 3-cycles gives a generator in the fundamental group $\pi_1(\bar \tau_2,\bar\Delta'_0)\cong \mathbb Z$ if it is degenerate, and gives trivial element in the same group if it is non-degenerate. However, the image $\bar\omega$ of $\omega$ in $\bar\Delta'_0$ gives a trivial element in $\pi_1(\bar \tau_2,\bar\Delta'_0)$, which leads to a contradiction.
	The remaining case is that all of the 3-cycles $520',120',4'25,4'32$ in $\bar\Delta'_0$ are degenerate, but this can be ruled out by Lemma~\ref{lem:degenerate triangles} below.
\end{proof}

\begin{figure}[h]
	\centering
	\includegraphics[scale=0.86]{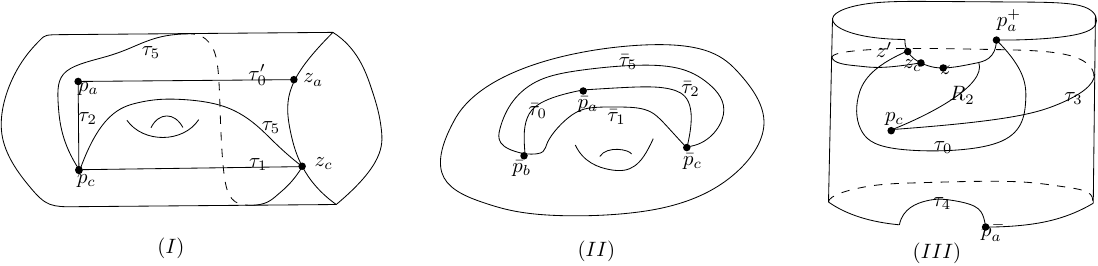}
	\caption{Some pictures.}
	\label{fig:1250}
\end{figure} 
\begin{lem}
	\label{lem:degenerate triangles}
	Suppose $\omega$ is type II. Assume $\tau_2\cap \tau_0=\emptyset$, $\tau_2\cap \tau_5=\emptyset$ and $\tau_2\cap \tau_4=\emptyset$ except at endpoints. Then it can not happen that all the 3-cycles $520,120,425,432$ in $\bar\Delta'_0$ are degenerate. 
\end{lem}

\begin{proof}
	As the 3-cycle $520$ in $\bar\Delta'_0$ is degenerate, we know the homotopy class of $\bar\tau_5$ is determined by $\bar\tau_2$ and $\bar\tau_0$ up to two choices. Similar $\bar\tau_1$ is determined by $\bar\tau_2$ and $\bar\tau_0$ up to two choices. If $\bar\tau_1$ is homotopic rel endpoints in $\bar\cS$ to $\bar\tau_5$, then $\tau_1$ and $\tau_5$ differ by a power of Dehn twist along $\partial\cS$. However, this is impossible by Lemma~\ref{lem:common adjacent} as $x_1$ and $x_5$ are adjacent to a common vertex $x_0$ of type $\hat a$. Thus $\bar\tau_5$ and $\bar\tau_1$ are not homotopic rel endpoints in $\bar\cS$. Thus up to a homeomorphism, we can assume we are in Figure~\ref{fig:1250} (II). In particular, $\bar\tau_1$ and $\bar\tau_5$ bound a disk with $\bar p_a$ inside.
	
	By Lemma~\ref{lem:generator}, each of the 3-cycle $520$ and the 3-cycle $120$ give a generator in $\pi_1(\bar\tau_2,\bar\Delta'_0)$. We claim these two generators are of the same sign. If this is not true, then the 4-cycle in $\bar\Delta'_0$ with vertices represented by $\bar \tau_1,\bar\tau_2,\bar\tau_5,\bar\tau_0$ is null-homotopic in $\bar\Delta'_0$. We denote this 4-cycle by $1250$. Thus any lift of $1250$ to $\Delta_0$ gives a 4-cycle. We consider a particular lift $\omega'\subset \Delta_0$ of $1250$ such that the vertex $\bar\tau_1$ lifts to the the vertex $x_1\in \Delta_0$ represented by $\tau_1$. As the path $x_1\to x_0\to x_5$ in $\Delta_0$ maps the path $\bar\tau_1\to\bar\tau_0\to\bar\tau_5$ in $\bar\Delta'_0$, we know that in $\omega'$, $\bar\tau_i$ lifts to $x_i$ for $i=0,5$. Let $y$ be the lift of $\bar\tau_2$ in $\omega'$. By \cite[Proposition 2.8]{huang2023labeled} applying to the 4-cycle $\omega'$ in $\Delta_\Lambda$, we know there exists $z\in \Delta_\Lambda$ of type $\hat d$ such that $z$ is adjacent to each of $x_1,x_0,x_5,y$. This implies that $y$ is adjacent to each of $\{x_0,x_1,x_5\}$ in $\Delta_0$. Thus there exists $\bar y$ in $\bar\Delta'_0$ such that $\bar y,\bar\tau_0,\bar\tau_1$ form a non-degenerate 3-cycle, and $\bar y,\bar\tau_0,\bar\tau_5$ form a non-degenerate 3-cycle. Thus the arc $\bar \gamma$ in $\bar\cS$  associated with $\bar y\in \bar\Delta'_0$ is disjoint from $\bar\tau_0,\bar\tau_1,\bar\tau_5$ except at endpoints, moreover the concatenation of $\bar \gamma,\bar\tau_0,\bar\tau_1$ and the concatenation of $\bar\gamma,\bar\tau_0,\bar\tau_5$ gives two homotopically non-trivial loop in the torus $\bar\cS'$.
	However, this is impossible as $\bar\tau_1$ and $\bar\tau_5$ bound a disk with $\bar p_a$ inside.
	
	Similarly, the 3-cycle $520$ and the 3-cycle $524$ give generators of the same sign in $\pi_1(\bar\tau_2,\bar\Delta'_0)$; and the 3-cycle $524$ and the 3-cycle $324$ give generators of the same sign in $\pi_1(\bar\tau_2,\bar\Delta'_0)$. This implies that the image of $\omega$ in $\bar\Delta'_0$ is not homotopically trivial - it represents $4$ times a generator in $\pi_1(\bar\tau_2,\bar\Delta'_0)$), a contradiction.
\end{proof}

\begin{lem}
	\label{lem:20e40ne}
	Suppose $\omega$ is of type II, and $\tau_0\cap\tau_3\neq\emptyset$.
	Suppose $\tau_2$ has a unique trace in $\cS_0$, all traces of $\tau_4$ are good, and $\tau_4$ has at least two traces. Then one of the following holds true:
	\begin{enumerate}
		\item $\omega$ satisfies property $(*)$;
		\item $x_1$ and $x_5$ is adjacent to a common vertex of type $\hat b$.
	\end{enumerate}
\end{lem}

\begin{proof}
	First we show $\tau_5\cap \tau_2=\emptyset$ except possibly at endpoints. As $\tau_4$ has at least two traces, Lemma~\ref{lem:S3} holds true with $\tau_1$ replaced by $\tau_5$. We scissor $\cS_0$ along $S_3$ to obtain a punctured annulus $A_0$. As $\tau_5\cap (\tau_0\cup S_3)=\emptyset$ and $\tau_2\cap (\tau_0\cup S_3)=\emptyset$ except possibly at endpoints, we know the trajectory of $\tau_5$ or $\tau_2$ in $A_0$ is an arc from $p_c$ to a point on $\partial A_0$. If $\tau_5$ and $\tau_2$ end in different components of $\partial A_0$, then $\tau_5\cap\tau_2=\emptyset$ except at endpoints. If $\tau_5$ and $\tau_2$ end in the same component, say $C$, of $\partial A_0$, then either $z_c\in S_3$, or $z_c$ and $p_a$ are in the same component of $\partial\cS_0\setminus \partial S_3$. It follows that the trajectory of $R_3$ (as defined in Lemma~\ref{lem:S3}) in $A_0$ ends in $C$ as well. As $\tau_5\cap R_3=\emptyset$ and $\tau_2\cap R_3=\emptyset$ except at endpoints, we know $\tau_2\cap \tau_5=\emptyset$ except at endpoints. Now the lemma follows from Lemma~\ref{lem:25}.
\end{proof}



\begin{lem}
	\label{lem:intersect03}
	Suppose $\omega$ is of type II,  $\tau_0\cap\tau_3\neq\emptyset$, and $\tau_4$ has a unique trace in $\cS_0$ which is good. Assume $\tau_2\cap\tau_4$ has an interior intersection point. Suppose
	\begin{enumerate}
		\item either $\tau_2$ has unique trace in $\cS_0$;
		\item or $\tau_2$ has two traces in $\cS_0$ and we are not in the situation that $\tau_2$ has a unique $\pi_4$-trace in $\cS_4$ which is a good trace in $\cS_4$.
	\end{enumerate}
	Then there is a vertex of type $\hat b$ adjacent to each of $\{x_0,x_5,x_4,x_3\}$.
\end{lem}

\begin{proof}
	By Lemma~\ref{lem:b}, it suffices to prove that
	$\tau_0\cap\tau_3$ has exactly one interior intersection point. Moreover, a subarc of $\tau_0$ containing $z_a$, a subarc of $\tau_3$ containing $z_c$ and one of the two arcs in $\partial \cS$ from $z_a$ to $z_c$ bound a disk in $\cS$. 
	
As $\tau_2\cap\tau_4$ has an interior intersection point, $\tau_2$ has at least one $\pi_4$-trace in $\cS_4$.
 Let $\tau_g$ be the union of all $\pi_4$-traces of $\tau_2$, and let $K$ be the component of $\cS_4\setminus \tau_g$ that contains the puncture $p_c\in \cS_4$. As $\tau_3\cap (\tau_4\cup \tau_2)=\emptyset$ except at endpoints, we know $\tau_3\subset K$ except at endpoints. In particular, $z_c\in K$. As $(\overline{zz_c}\cup\overline{z'z_c})\cap \tau_2=\emptyset$, we know $\{z,z'\}\subset K$. Let $\tau'_0$ be the subarc of $\tau_0$ starting from one of $\{z,z'\}$ until it hits $\tau_2$ the first time. Let $\tau''_0$ be the remaining subarc of $\tau_0$. Since we assume $\tau_2$ has at most two traces in $\cS_0$, either $\tau'_0=\tau_0$, or $\tau'_0\neq\tau_0$ and $\tau''_0\cap\tau_2=\emptyset$ except at endpoints.
Note that $\tau'_0\subset K$ except at endpoints. Also  $K$ is either a punctured disk or a punctured annulus.

We first consider the case $\tau_0=\tau'_0$. If $K$ is a punctured disk, then the claim is clear. If $K$ is a punctured annulus, then all $\pi_4$-traces of $\tau_2$ are good and they in the same parallel class. Thus the pattern of $\tau_2\cap\partial \cS_4$ is indicated in Figure~\ref{fig:12'} (II). Let $R_2$ be the trace of $\tau_2$ containing $p_c$ and ending at $\theta$. Let $C^+$ and $C^-$ be two boundary components of $K$. Then $\theta$ and $z_c$ are in the same component, say $C^+$, of $\partial K$ by Figure~\ref{fig:12'} (II). As $\tau_3\cap R_2=\emptyset$ except at endpoints and they end in $C^+$, we know $R_2$ determines $\tau_3$ up to two choices. As $\{z,z'\}\in C^+$, we know $\tau'_0$ starts from a point in $C^+$. If $\tau_2$ has more than one good $\pi_4$-traces, then $p_a\in C^-$. If $\tau_2$ has exactly one good $\pi_4$-trace, then $p_a$ gives $p^+_a\in C^+$ and $p^-_a\in C^-$.
If $\tau'_0$ ends in $C^-$, then the claim follows from $\tau'_0\cap\tau_2=\emptyset$ (except at endpoints). If $\tau'_0$ ends in $C^+$, then $\tau_2$ has exactly one good $\pi_4$ trace, and $\tau'_0$ goes from one of $\{z,z'\}$ to $p^+_a$. In this case, then only possibility for $\tau'_0\cap \tau_3$ to have more than one interior intersection point is as in Figure~\ref{fig:1250} (III). However, this can be ruled out since $\tau'_0$ is not a good trace in $\cS_4$, hence $\tau_4$ is not a good trace in $\cS_0$, contradiction.

If $\tau_0\neq\tau'_0$, then by the same argument in the previous paragraph, we know the conclusion of the lemma still holds, if $\tau''_0$ is not in $K$. This happens when $\cS_0\setminus\tau_g$ has more than one components. However, $\cS_0\setminus \tau_g$ has only one components if and only if $\tau_2$ has a unique $\pi_4$ in $\cS_4$ which is a good trace in $\cS_4$. Thus the lemma follows.
\end{proof}

\begin{lem}
	\label{lem:bg}
	Suppose $\omega$ is of type II. If there is a vertex $z$ of type $\hat b$ in $\Delta_\Lambda$ which is adjacent to each of $\{x_0,x_5,x_4,x_3\}$. Then $\omega$ has property $(*)$.
\end{lem}

\begin{proof}
We consider the 5-cycle $\omega'$ with consecutive vertices $\{x_0,z,x_3,x_2,x_1\}$, of type $\{\hat a,\hat b,\hat c,\hat b,\hat c\}$. We can assume $\omega'$ is embedded, otherwise $z=x_2$ and $x_2$ is adjacent to each of $\{x_1,x_3,x_5\}$, hence property $(*)$ follows.
By Proposition~\ref{prop:typeI} and Lemma~\ref{lem:filling 5 cycle} applied to $\omega'$ (while the statement of Lemma~\ref{lem:filling 5 cycle} is for cycles with vertex type $\hat b,\hat a,\hat c,\hat a,\hat c$, up to a symmetry of the Dynkin diagram $\Lambda$, it also applies to cycles with vertex type $\hat a,\hat c,\hat b,\hat c,\hat b$), either $x_1$ and $z$ are adjacent, which implies $z$ is adjacent to each of $\{x_1,x_3,x_5\}$ and property $(*)$ follows, or $x_2$ is adjacent to $x_5$ and property $(*)$ follows.
\end{proof}

\begin{lem}
	\label{lem:24 unique}
	Suppose $\omega$ is of type II.
	Suppose $\tau_2$ has a unique trace in $\cS_0$ and $\tau_4$ has a unique trace which is good. Then property $(*)$ holds for $\omega$.
\end{lem}

\begin{proof}
If  $\tau_0\cap\tau_3=\emptyset$, then $\omega$ readily has property $(*)$. Now we assume $\tau_0\cap\tau_3\neq\emptyset$.
If $\tau_2$ and $\tau_4$ have an interior intersection point, then the lemma follows from Lemma~\ref{lem:intersect03} and Lemma~\ref{lem:bg}. Now we assume $\tau_2\cap\tau_4$ except at endpoints.

We consider the 3-cycles $432$ and $120$ in $\bar\Delta'_0$. If at least one of them is non-degenerate, then $\omega$ has property $(*)$. Indeed, by symmetry we assume $120$ is non-degenerate. Then $x_2$ and $x_0$ are adjacent. By applying Proposition~\ref{prop:typeI} and Lemma~\ref{lem:filling 5 cycle} to the 5-cycle $x_2,x_0,x_5,x_4,x_3$, we know either $x_0$ and $x_3$ are adjacent, which is ruled out by our assumption, or $x_2$ is adjacent to $x_5$, and property $(*)$ follows. 

It remains to consider that both $432$ and $120$ are degenerate.
First we claim up to replacing $x_5$ by a vertex $x'_5$ of type $\hat c$ adjacent to both $x_0$ and $x_4$, we can assume the associated arc $\tau'_5$ in $\cS$ satisfies that $\tau'_5\cap\tau_2=\emptyset$ except at endpoints. For this, we consider $\cS_0$ and scissor it along $\tau_4$. As $\tau_4$ gives a single good trace in $\cS_0$, we obtain a punctured annulus $A_0$ with two components $C^+$ and $C^-$. The point $p_a$ gives $p^+_a\in C^+$ and $p^-_a\in C^-$. As $\tau_2\cap(\tau_0\cup\tau_4)=\emptyset$, and $\tau_5\cap(\tau_0\cup\tau_4)=\emptyset$, we know the trajectories of $\tau_2$ and $\tau_5$ go from $p_c$ to one of the boundary components of $A_0$. More precisely, if we assume without loss of generality that $z_c\in C^+$, then $\tau_5$ ends in $z_c$ and $\tau_2$ ends in either $p^+_a$ or $p^-_a$. If $\tau_2$ and $\tau_5$ end in different boundary components of $A_0$, then we already have $\tau_2\cap \tau_5=\emptyset$ except at endpoints. If they ends at the same boundary component of $A_0$, then there is an arc $\tau'_5$ in $A_0$ from $p_c$ to $z_c$ avoiding $\tau_2$. Thus the claim is proved.

By Lemma~\ref{lem:degenerate triangles}, it is impossible that each of $210,5'20,5'24$ and $324$ gives degenerate 3-cycle in $\bar\Delta'_0$. The case when exactly three of these four 3-cycles are degenerate is ruled out by Lemma~\ref{lem:generator}. Thus the only possibility left is that both $5'20$ and $5'24$ are non-degenerate. As $120$ is degenerate and $5'20$ is non-degenerate, by the third paragraph of the proof of Lemma\ref{lem:25}, $x_1$ and $x'_5$ are adjacent to a common vertex $x'_0$ of type $\hat b$. By symmetry, we know $x'_5$ and $x_3$ are adjacent to a common vertex $x'_4$ of type $\hat b$. As the Dynkin diagram $\Lambda$ has symmetries inducing arbitrary permutations of $\{a,b,c\}$, we can apply Proposition~\ref{prop:typeI} to the 6-cycle $x'_0x_1x_2x_3x'_4x'_5$ and deduce that there is a vertex $z$ of type $\hat a$ or $\hat b$ such that $z$ is adjacent to each of $\{x_1,x_3,x'_5\}$. Thus the 6-cycle $x_0x_1x_2x_3x_4x'_5$ has property $(*)$.
By Lemma~\ref{lem:6cycle replace}, $\omega$ has property $(*)$.
\end{proof}

\begin{lem}
	\label{lem:20intersect}
	Suppose $\omega$ is of type II,  $\tau_0\cap\tau_3\neq\emptyset$, and $\tau_4$ has a unique trace in $\cS_0$ which is good. If $\tau_2$ has a unique $\pi_0$-trace in $\cS_0$ which is good, then $\omega$ has property $(*)$.
\end{lem}

	\begin{figure}
	\centering
	\includegraphics[scale=0.9]{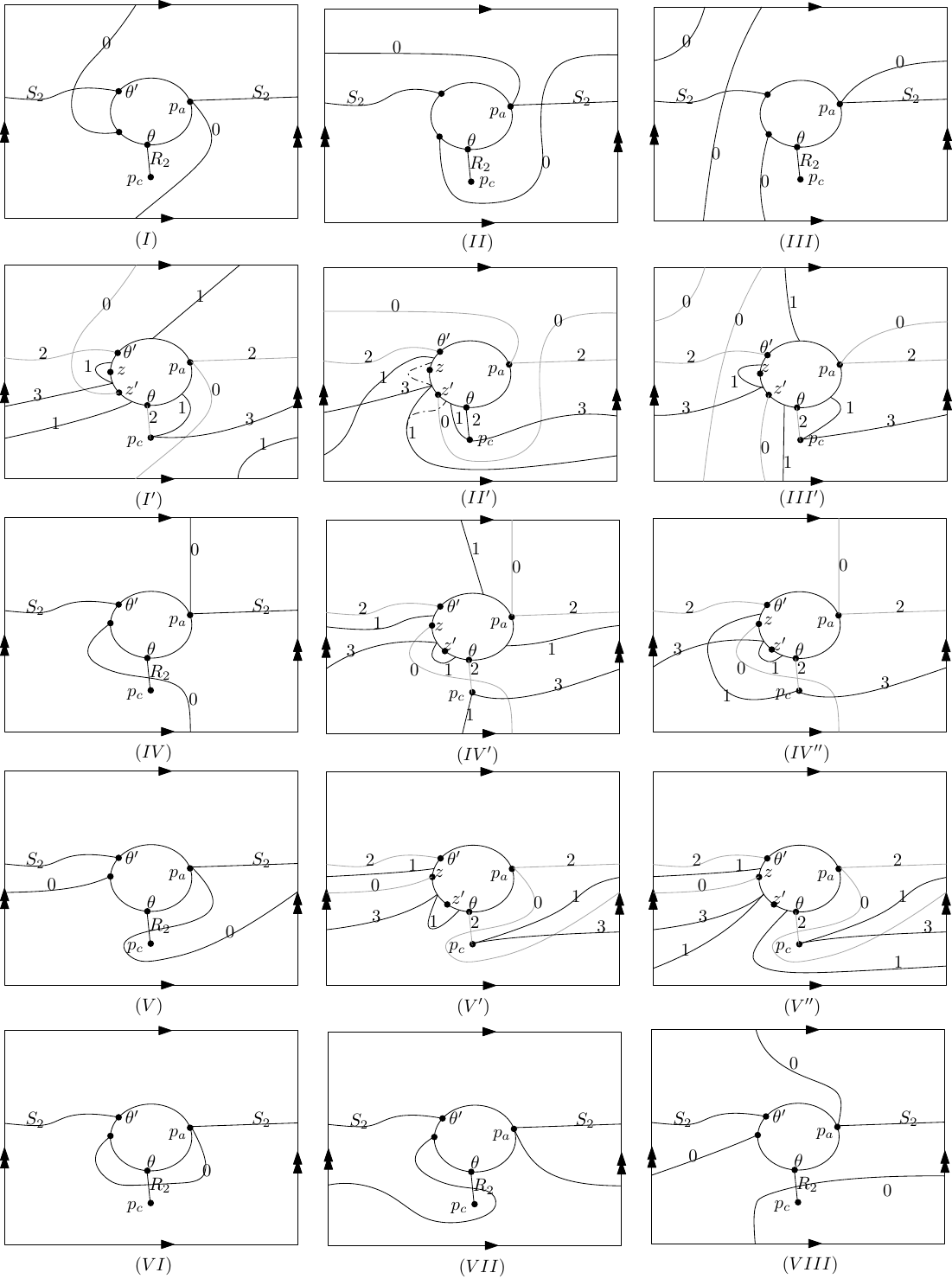}
	\caption{Proof of Lemma~\ref{lem:20intersect}}
	\label{fig:onetrace}
\end{figure} 

\begin{proof}
The case $\tau_2\cap \tau_4=\emptyset$ follows by applying Lemma~\ref{lem:24 unique} with the role of $\{x_0,x_1\}$ and $\{x_4,x_3\}$ exchanged. Now suppose $\tau_2\cap\tau_4$ has at least one interior intersection point. It suffices to consider the case that $\tau_2$ has a unique $\pi_4$-trace in $\cS_4$ which is good, otherwise we are done by Lemma~\ref{lem:intersect03} and Lemma~\ref{lem:bg}.
	
	Let $R_2$ be the trace of $\tau_2$ in $\cS_4$ containing $p_c$, and let $S_2$ be the other trace of $\tau_2$. We scissor $\cS_4$ along $S_2$ to obtain an annulus $A_0$, with two boundary components $C^+$ and $C^-$. We assume without loss of generality that $R_2$ ends on $\theta\in \overline{z'p_a}\subset C^+$. Then $z,z_c\in C^+$. Let $p_a^+$ and $p^-_a$ be the copy of $p_a$ in $C^+$ and $C^-$ respectively. Note that there are two homotopy classes of arcs (rel endpoints) in $A_0$ from $p_c$ to $z_c$ that have representatives disjoint from $R_2$. We choose a representative from each class and denote them $\gamma$ and $\beta$. Assume without loss of generality that $\gamma$, $R_2$ and $\overline{z_c\theta}\subset C^+$ bound a disk $D$. As $\tau_3\cap(\tau_2\cup\tau_4)=\emptyset$ except at endpoints, we can assume either $\tau_3=\gamma$ or $\tau_3=\beta$. By a similar argument, we also know that there are only two homotopy classes of arcs (rel endpoints) in $\cS$ from $z_c$ to $p_c$ which contain representatives avoiding both $\tau_0$ and $\tau_2$. Thus $\tau_0$ and $\tau_2$ determine $[\tau_1]$ up to two choices.

By our assumption, $\tau_0$ has a single trace in $\cS_4$ which is good. Moreover,  $\tau_0$ and $\tau_2$ have exactly one interior intersection point.

\medskip
\noindent
\underline{Case 1: $\tau_0$ and $S_2$ have an interior intersection point.} Let $\theta'$ be the endpoint of $S_2$ that is identified with $\theta$ via $\phi$. Then $\tau_0\cap R_2=\emptyset$. Suppose $\tau_3=\gamma$. If $z\in \tau_0$, then $\tau_0\cap \tau_3=\emptyset$, otherwise $\tau_0$ will enter the disk $D$. As $\tau_0\cap\partial\cS_4=\{z',p_a\}$ and $\tau_0\cap R_2=\emptyset$, we know $\tau_0$ will exit $D$ at another point, contradicting the minimal position assumption. If $z'\in \tau_0$, then by a similar argument we know the assumptions of Lemma~\ref{lem:b} are satisfied. Then $\omega$ has property $(*)$ by Lemma~\ref{lem:bg}.

Now we assume $\tau_3=\beta$. Suppose in $A_0$, $\tau_0$ goes from one of $\{z,z'\}$ and first hits $C^+$. As $\tau_0$ does not bound a (possibly punctured) disk with a subarc of $\partial \cS_4$ in $\cS_4$, up to a homeomorphism, we are either in Figure~\ref{fig:onetrace} (I) or Figure~\ref{fig:onetrace} (II). We only discuss the case $z'\in \tau_0$, as the other case is similar. Recall that up to homotopy, $\tau_1$ must be one of the two arcs in $\cS$ from $z_c$ to $p_c$ avoiding $\tau_0$ and $\tau_2$. Suppose we are in Figure~\ref{fig:onetrace} (I). In Figure~\ref{fig:onetrace} (I'), we display the less obvious possibility of $\tau_1$ when $z'\in \tau_0$, following the arcs labeled by $1$, and noting that $\overline{p_a\theta'}$ and $\overline{p_a\theta}$ are identified, and $\overline{\theta'z}$ and $\overline{\theta z'}$ are identified.
Note that both possibilities of $\tau_1$ will have empty intersection with $\tau_3$ (except at endpoints), moreover $\tau_1\cup\tau_3$ forms a homotopically non-trivial simple closed curve in $\cS'$ (which is obtained from $\cS'$ by filling back the punctures) which is not homotopic to $\partial \cS'$. Thus there is an arc in $\cS$ from $p_a$ to $z_a$ that is disjoint from both $\tau_1$ and $\tau_3$, implying $x_1$ and $x_3$ is adjacent to a common vertex of type $\hat a$ in $\Delta_\Lambda$. Then $\omega$ has property $(*)$ by Proposition~\ref{prop:typeI}. Now suppose we are in Figure~\ref{fig:onetrace} (II). Assume $z'\in \tau_0$ (the case $z\in \tau_0$ is similar and easier). In Figure~\ref{fig:onetrace} (II'), we display the less obvious possibility of $\tau_1$. Let $\tau'_1$ be the arc obtained by apply a Dehn twist along $\partial \cS$ to $\tau_1$ - this has the effect of replace the end of $\tau_1$ by the dashed part as displayed in Figure~\ref{fig:onetrace} (II). Note that there are infinitely many homotopy classes of arcs from $z'$ to $p_a$ which has representatives avoiding $(\tau'_1\cup\tau_3)$. Thus by Lemma~\ref{lem:common adjacent} (2), either $x_1$ and $x_3$ are not adjacent to a common vertex of type $\hat b$, which rules out this case, or $x_1$ and $x_3$ are adjacent to a common vertex of type $\hat a$, and we finish as before. For the other choice of $\tau_1$, $\tau_1\cup\tau_3$ forms a homotopically non-trivial simple closed curve in $\cS'$ which is not homotopic to $\partial \cS'$, and we finish as before.

Suppose in $A_0$, $\tau_0$ goes from one of $\{z,z'\}$ and first hits $C^-$. Then up to a homeomorphism, we are in Figure~\ref{fig:onetrace} (III). In this case, $\tau_1\cup\tau_3$ forms a homotopically non-trivial simple closed curve in $\cS'$ which is not homotopic to $\partial \cS'$ (see Figure~\ref{fig:onetrace} (III') for the less obvious possibility of $\tau_1$ in the case $z'\in \tau_0$), and we finish as before.

\medskip
\noindent
\underline{Case 2: $\tau_0$ and $R_2$ have one interior intersection point.}
Then $\tau_0\cap S_2=\emptyset$ except at endpoints. Up to a homeomorphism, we are in either Figure~\ref{fig:onetrace} (IV), (V), (VI), (VII) or (VIII). Note that (VII) is ruled out as the $\pi_0$-trace of $\tau_2$ on $\cS_0$ is a good trace, and (VI) is ruled out as the $\tau_0$ gives a good $\pi_0$ trace in $\cS_4$.
We will only discuss the case $z\in \tau_0$ as the case $z'\in \tau_0$ is similar. We will only discuss (IV) and (V) as (VIII) follows from a similar argument.
If we are in Figure~\ref{fig:onetrace} (IV) (resp. (V)), then the two possibilities of $\tau_1$ are displayed in Figure~\ref{fig:onetrace} (IV') and (IV'') (resp. (V') and (V'')). In Figure~\ref{fig:onetrace} (IV') and (V''), $\tau_1$ and $\tau_3$ give a homotopically non-trivial simple closed curve in $\cS'$ which is not homotopic to $\partial\cS'$, and we finish as before. In Figure~\ref{fig:onetrace} (IV''), $\tau_1$ and $\gamma$ differ by a power of Dehn twist along $\partial \cS$. By considering $\tau_1,\gamma,\tau_3$, Lemma~\ref{lem:common adjacent} (2) implies that either $x_1$ and $x_3$ are not adjacent to a common vertex of type $\hat b$, or $x_1$ and $x_3$ are adjacent to a common vertex of type $\hat a$, and we finish as before. In Figure~\ref{fig:onetrace} (V'), $\tau_1$ and $\tau_3$ differ by a Dehn twist along $\partial \cS$, however, this is impossible by Lemma~\ref{lem:common adjacent}. 
\end{proof}

\begin{lem}
	\label{lem:4notgood}
	Suppose $\omega$ has type II and $\tau_0\cap \tau_3\neq\emptyset$ in $\cS$.
	Suppose $\tau_4$ has a trace in $\cS_0$ which is not good. Suppose $\tau_2\cap\tau_0=\emptyset$ except at endpoints. Then one of the following holds true:
	\begin{enumerate}
		\item $x_1$ and $x_5$ are adjacent to a common vertex of type $\hat b$;
		\item $\omega$ satisfies condition $(*)$.
	\end{enumerate}
\end{lem}

\begin{proof}
	We first claim $\tau_2\cap \tau_5=\emptyset$. Take a trace of $\tau'_4$ of $\tau_4$ which is not good. Lemma~\ref{lem:boundary parallel} implies that $\tau'_4$ and a subarc of $\partial \cS_0$ bound a punctured disk $D$ with $p_c\in D$. Then $\tau_5\subset D$. 
	As in Corollary~\ref{cor:boundary parallel 24} we assume without loss of generality that $\tau'_4$ goes from $z$ to a point $x\in \overline{z'p_a}$. If $\tau_2$ stays inside $D$, then the claim follows. If $\tau_2$ goes out of $D$, then it can not reenter again, as once it reenters, then it can not end in the interior of $D$, and it can not end in $\partial\cS_0\cap \partial D$ except at $p_a$  (this uses that $\tau_2\cap \tau_0=\emptyset$ except at endpoints), hence it forms a bigon with $\tau'_4$. We must have $p_c$ inside this bigon as $\tau_2$ and $\tau_4$ are in minimal position. However, this is a contradiction, as $\tau_0\cap\tau_3\neq\emptyset$ implies that a trace of $\tau_3$ goes from $p_c$ to a point in $\partial \cS_0\setminus\{p_a\}$ avoiding $\tau_2\cup\tau_4$. As $\tau_2$ can not reenter, the claim is proved. We are done by Lemma~\ref{lem:25}.
\end{proof}

\section{Ending remarks}
\label{sec:remark}
In this last section, we explain in general what are the types of 6-cycles we need to handle in proving $K(\pi,1)$-conjecture, and conjecture that all of these 6-cycles has a quasi-center, Conjecture~\ref{conj:zigzag}. We also explain the relationship between Conjecture~\ref{conj:zigzag} and a conjecture of Haettel, see Conjecture~\ref{conj:h}. We also deduce that these conjectures are true for type $D_n$ Artin groups with $n=3,4$.
\begin{definition}
	\label{def:zigzag}
	Let $\Lambda$ be the Dynkin diagram of type $D_n$ with its vertices as in Figure~\ref{fig:ad}. Let $\Delta_\Lambda$ be the associated Artin complex. We define a partial order on the set of types of vertices in $\Delta_{\Lambda}$ by declaring that 
	\begin{itemize}
		\item $
		\{\hat\delta_1,\hat\delta_2\}<\hat \delta_3<\hat \delta_4<\cdots<\delta_n$,
		\item $\hat\delta_1$ and $\hat\delta_2$ are not comparable.
	\end{itemize}
	A 6-cycle $\omega$ in $\Delta_\Lambda$ with consecutive vertices $\{x_i\}_{i\in \mathbb Z/6\mathbb Z}$ is \emph{admissible} if the types of $x_i$ and $x_{i+1}$ are comparable for each $i$. We say $x_i$ is a local max if the types of $x_{i-1}$ and $x_{i+1}$ are less than the type of $x_i$. Similarly we define local min vertex of $\omega$. A \emph{zigzag} 6-cycle in $\Delta_\Lambda$ is an admissible 6-cycle in $\Delta_\Lambda$ whose vertices alternate between local max and local min.
\end{definition}

\begin{conj}
	\label{conj:zigzag}
Suppose $\omega$ is a zigzag $6$-cycle in the type $D_n$ Artin complex $\Delta_\Lambda$. Then $\omega$ has a quasi-center which is adjacent to each of the local max vertices of $\omega$.
\end{conj}

\begin{cor}
	\label{cor:zigzag}
This conjecture holds when $n=3,4$.
\end{cor}

\begin{proof}
We will only treat $n=4$ case, as the $n=3$ case is much easier and follows from the same argument.	
	
Let the 6-cycle $\omega$ be $y_1x_1y_2x_2y_3x_3$ where for each $i$, $x_i$ is a local min and $y_i$ is a local max. By Lemma~\ref{lem:link} (3), up to replacing $x_i$ by a different vertex, we can assume without loss of generality that each $x_i$ has type $\hat \delta_1$ or $\hat \delta_2$. Then each $y_i$ has type $\hat \delta_3$ or $\hat \delta_4$.

The case when each $y_i$ has type $\hat \delta_3$.  Note that if $x_i$ and $x_{i+1}$ do not have the same type, then they are adjacent in $\Delta_\Lambda$ by looking at the link of a type $\hat \delta_3$ vertex in $\Delta_\Lambda$ and applying Lemma~\ref{lem:link} (3). If $\{x_1,x_2,x_3\}$ do not have the same type, then up to rotation of indices, we can assume $x_1$ and $x_3$ have the same type, but their type is different from that of $x_2$. Then $x_1$ and $x_2$ are adjacent in $\Delta_\Lambda$, and $x_2$ and $x_3$ are adjacent in $\Delta_\Lambda$. By Lemma~\ref{lem:special4cycle}, $x_2$ and $y_1$ are adjacent. Hence $x_2$ is a quasi-center of $\omega$.
It remains to consider that all $x_i$ have the same type, in which case we are done by Theorem~\ref{thm:weakflagD}.

The case that each $y_i$ has type $\hat \delta_4$ follows from Theorem~\ref{thm:flagD4} (where $\delta_4$ plays the role of $c$ in Theorem~\ref{thm:flagD4}).

Now suppose two of $\{y_1,y_2,y_3\}$, say $y_1$ and $y_2$, has type $\hat\delta_4$, but $y_3$ has type $\hat \delta_3$. Let $y'_3$ be an vertex of type $\hat \delta_4$ that is adjacent to $y_3$. We can assume $y'_3\neq y_i$ for $i=1,2$, indeed, if $y'_3=y_1$, then $y_3$ and $y_1$ are adjacent, and Lemma~\ref{lem:link} applied to $\lk(y_3,\Delta_\Lambda)$ implies that $y_1$ and $x_2$ are adjacent, hence $x_2$ is a quasi-center of $\omega$; a similar argument applies if $y'_3=y_2$.
By Theorem~\ref{thm:weakflagD}, there is a vertex $z$ of type $\hat \delta_1$ or $\hat \delta_2$ such that $z$ is adjacent to each of $\{y_1,y_2,y'_3\}$. If $z=x_2$ or $x_3$, then $z$ is a quasi-center of $\omega$. Thus we assume $z\neq x_i$ for $i=2,3$. Then $y'_3x_3y_1z$ and $y'_3x_2y_2z$ give two embedded 4-cycles. By Theorem~\ref{thm:4 wheel}, there exist are vertices $x'_2,x'_3$ of type $\hat \delta_3$ such that $x'_3$ is adjacent to each of $\{y'_3,x_3,y_1,z\}$ and $x'_2$ is adjacent to each of $\{y'_3,x_2,y_2,z\}$. Consider the 6-cycle $x_3x'_3zx'_2x_2y_3$. By previous discussion, there is a vertex $z'$ of type $\hat\delta_1$ or $\hat \delta_2$ such that $z'$ is adjacent to each of $\{x'_3,x'_2,y_3\}$. By applying Lemma~\ref{lem:link} (3) to $\lk(x'_3,\Delta)$, we know $z'$ is adjacent to $y_1$. Similarly $z'$ is adjacent to $y_2$. Thus $z'$ is a quasi-center of $\omega$.

The remaining case is that one of $\{y_1,y_2,y_3\}$, say $y_1$, has type $\hat \delta_4$, and the other two have type $\hat \delta_3$. We define $y'_3$ as in the previous paragraph. By the previous paragraph, there is a vertex $z$ of type $\hat \delta_1$ or $\hat \delta_2$ such that $z$ is adjacent to each of $\{y_1.y_2,y'_3\}$. Similarly, we can assume $y'_3\neq y_i$ for $i=1,2$ and $z\neq x_i$ for $i=2,3$. We define $x'_3$ as before. By considering the 6-cycle $x_3x'_3zy_2x_2y_3$, we know there is a vertex $z'$ of type $\hat \delta_1$ or $\hat \delta_2$ such that $z'$ is adjacent to each of $\{y_3,x'_3,y_2\}$. We know $z'$ is adjacent to $y_1$ by the same argument as in the previous paragraph. Thus $z'$ is a quasi-center of $\omega$.
\end{proof}

Conjecture~\ref{conj:zigzag} is a reformulation of a conjecture of Haettel as follows.

\begin{definition}
	\label{def:subdivision}
Suppose $A_\Lambda$ is an Artin group whose Dynkin diagram $\Lambda$ is of type $D_n$ with its vertex set as in Figure~\ref{fig:ad}. Let $\Delta=\Delta_{\Lambda}$ be the associated Artin complex. We subdivide each edge of $\Delta$ connecting a vertex of type $\hat \delta_1$ and a vertex of type $\hat \delta_2$. We say the middle point of such edge is of type $m$. Cut each top dimensional simplex in $\Delta$ into two simplices along the codimensional 1 simplex spanned by vertices of type $m$ and $\{\delta_i\}_{i=3}^{n}$. This gives a new simplicial complex, which we denoted by $\Delta'$. Define a map $t$ from the vertex set $V\Delta'$ of $\Delta'$ to $\{1,2,\ldots,n\}$ by sending vertices of type $\hat \delta_1,\hat \delta_2$ to $1$, vertices of type $m$ to $2$, vertices of type $\hat \delta_i$ to $i$ for $i\ge 3$. Define a relation $<$ on $V\Delta'$ as follows. For $x,y\in V\Delta'$, $x<y$ if $x$ and $y$ are adjacent and $t(x)<t(y)$. The simplicial complex $\Delta'$, together with the relation $<$ on its vertex set, is called the \emph{$(\delta_1,\delta_2)$-subdivision of $\Delta_{\Lambda}$}. 
\end{definition}

\begin{conj}[Haettel]
	\label{conj:h}
Suppose $\Lambda$ is of type $D_n$. The vertex set $V$ of the $(\delta_1,\delta_2)$-subdivision $\Delta'$ of $\Delta_\Lambda$, endowed with the partial order in Definition~\ref{def:subdivision}, is downward flag in the following sense: if three elements $\{y_1,y_2,y_3\}\subset V$ satisfy that each pair of them has a lower bound in $V$, then there is a common lower of $\{y_1,y_2,y_3\}$ in $V$.
\end{conj}

\begin{lem}
	If Conjecture~\ref{conj:zigzag} holds, then Conjecture~\ref{conj:h} holds.
\end{lem}

\begin{proof}
Let $V$ and $\{y_1,y_2,y_3\}$ be as in Conjecture~\ref{conj:h}. For $i\in \mathbb Z/3\mathbb Z$, let $x_i$ be a lower bound of $y_{i-1}$ and $y_{i}$. We can assume without loss of generality that $t(x_i)=1$ for each $i$. If $t(y_i)\neq 2$ for each $i$, then $y_1x_1y_2x_2y_3x_3$ gives a zigzag 6-cycle in $\Delta_\Lambda$, and we are done by Conjecture~\ref{conj:zigzag}.

If $t(y_i)=2$, then $x_i$ and $x_{i+1}$ are adjacent in $\Delta_\Lambda$. Thus if at least two of $\{y_1,y_2,y_3\}$ satisfy $t(y_i)=2$, then we are reduce analyze 4-cycles in $\Delta_\Lambda$, and it is not hard to deduce from Theorem~\ref{thm:4 wheel} that $\{y_1,y_2,y_3\}$ has a common lower bound. 

It remains to consider exactly one of $\{y_1,y_2,y_3\}$, say $y_1$, satisfies $t(y_1)=2$. Let $y'_1$ be a vertex with $t(y'_1)=3$ such that $y'_1$ and $y_1$ are adjacent in $\Delta'$. By apply the previous discussion to the 6-cycle $y'_1x_1y_2x_2y_3x_3$, we know $\{y'_1,y_2,y_3\}$ has a common lower bound $z\in V$. We can assume without loss of generality that $t(z)=1$. We will assume $x_3$ and $z$ are not adjacent in $\Delta_\Lambda$, otherwise $\{x_3,z\}$ has a common upper bound $y'_3$ with $t(y'_2)$, and by applying the previous paragraph to $\{y_1,y_2,y'_3\}$, we know $\{y_1,y_2,y'_3\}$ has a lower bound, hence $\{y_1,y_2,y_3\}$ has a lower bound. Similarly, we assume $x_1$ and $z$ are not adjacent in $\Delta_\Lambda$. We will assume $y'_1\neq y_3$, otherwise $x_1$ is a common lower bound for $\{y_1,y_2,y_3\}$ in $V$. Similarly, we assume $y'_1\neq y_2$.

We claim $y'_1$ is adjacent to $y_3$ in $\Delta_\Lambda$.
By applying Theorem~\ref{thm:4 wheel} to the 4-cycle $x_3y'_1zy_3$, we know that either $y'_1$ and $y_3$ are adjacent in $\Delta_\Lambda$, or there is a vertex $w$ of type $\hat\delta_i$ which is adjacent to each of $\{x_3,y'_1,z,y_3\}$. As the type of $w$ is different from the types of each of $\{x_3,y'_1,z,y_3\}$, Theorem~\ref{thm:4 wheel} implies that $\delta_i$ is contained in the subsegment of $\Lambda$ between $\delta_3$ and $\delta_k$ where $y_3$ has type $\hat \delta_k$. Then Lemma~\ref{lem:link} implies that $y'_1$ and $y_3$ are adjacent in $\Delta_\Lambda$, as desired. 
This claim implies that $y_3\ge y'_1\ge x_1$, thus $x_1$ is a common lower bound for $\{y_1,y_2,y_3\}$, as desired.
\end{proof}

\begin{cor}
Conjecture~\ref{conj:h} holds when $n=3,4$.
\end{cor}

\bibliographystyle{alpha}
\bibliography{mybib}

\end{document}